\documentclass[10pt, a4paper]{article}

\usepackage[english]{babel}

\usepackage{mathrsfs}

\usepackage{todonotes}  

\usepackage[noadjust]{cite} 
\usepackage{xcolor}

\usepackage{hyperref}
\hypersetup{
    colorlinks,
    citecolor=black,
    filecolor=black,
    linkcolor=black,
    urlcolor=black    
}

\usepackage[utf8]{inputenc}
\usepackage{amsmath}
\usepackage{amssymb}
\usepackage{mathtools}
\usepackage{cancel}
\usepackage {graphicx}
\usepackage{floatrow}
\usepackage{multirow}
\usepackage{mathalfa}
\usepackage{dutchcal}
\usepackage{tikz}

\usepackage{tikz-cd}
\usetikzlibrary{arrows}
\usetikzlibrary{decorations.pathmorphing}
\usetikzlibrary{cd}

\usepackage{amssymb,amsthm}

\usepackage{amsfonts}
\usepackage{latexsym}
\usepackage{color}
\usepackage{subcaption}
\usepackage{caption}
\usepackage{multicol}
\usepackage{fullpage}
\usepackage{bbold}

\newtheorem{teorema}{Theorem}[section]
\newtheorem{definizione}[teorema]{Definition}
 \newtheorem{corollario}[teorema]{Corollary}
\newtheorem{lemma}[teorema]{Lemma}
\newtheorem{proposizione}[teorema]{Proposition}
\newtheorem{osservazione}[teorema]{Remark}

\newcommand{\numberset}{\mathbb}

\newcommand{\R}{\numberset{R}}

\newcommand{\EMAIL}[1]{\href{mailto:#1}{\texttt{#1}}}

\title{On the relation between transformations of SDEs and of associated Kolmogorov equations}
\date{}
\author{
Susanna Dehò \footnote{Dipartimento di Matematica, Universit\`a degli Studi di Milano, \EMAIL{susanna.deho@unimi.it}} \and
	Francesco C. De Vecchi\footnote{Dipartimento di Matematica ``Felice Casorati'', Universit\`a degli Studi di Pavia
		\EMAIL{francescocarlo.devecchi@unipv.it}}
	\and 
	Paola Morando \footnote{DISAA, Universit\`a degli Studi di Milano,  \EMAIL{paola.morando@unimi.it}}
	\and
	Stefania Ugolini \footnote{Dipartimento di Matematica, Universit\`a degli Studi di Milano, \EMAIL{stefania.ugolini@unimi.it}}
}	

\begin{document}
\maketitle

\begin{abstract}
The relationship between transformations of one-dimensional stochastic differential equations and those of the associated Kolmogorov backward equations is investigated. Stochastic transformations combining deterministic time changes, space-time diffeomorphisms, and Girsanov-type changes of probability measure are shown to be in correspondence with admissible transformations of the associated linear parabolic equations, namely the transformations preserving the parabolic structure up to the possible introduction of a zero-order potential term. 
A distinguished subclass of transformations, characterized by the absence of any potential term in the transformed Kolmogorov equation, is shown to coincide exactly with Doob transformations of the underlying probability measure of the SDE. 
This characterization further establishes a correspondence between Doob symmetries of an SDE and Lie point symmetries of the associated Kolmogorov backward equation. This correspondence, in turn, yields a constructive (probabilistic) Lie symmetry classification of autonomous Kolmogorov backward equations.\end{abstract}
\medskip

\medskip

AMS SUBJ {60H10; 58D19; 60H07}

\section*{Introduction}

In his seminal work \cite{kolmogorov1931}, Kolmogorov posed the fundamental problem of determining the conditions under which a one-dimensional diffusion process can be reduced to a standard Wiener process. This question was subsequently investigated by several authors (e.g., \cite{bluman,ricciardi,cherkasov1957}), who shared a common methodological approach: exploiting the deep connection between diffusion processes and parabolic equations provided by the Feynman--Kac formula. Within this framework, reducing a one-dimensional stochastic differential equation (SDE) driven by a Brownian motion,
\begin{equation}
    dX_t = \mu(X_t,t) \, dt + \sigma(X_t,t) \, dW_t,
\end{equation}
to a standard Brownian motion is analytically equivalent to transforming the associated Kolmogorov backward equation (KBE),
\begin{equation}
    u_t + \mu(x,t) u_x + \frac{1}{2}\sigma^2(x,t) u_{xx} = 0,
\end{equation}
into the heat equation. Naturally, this perspective requires a rigorous correspondence between the transformations applied to the SDE and those applied to the associated KBE.

The significance of this deterministic viewpoint is underscored by Lie's classical theorem \cite{lie1881}, which establishes that a one-dimensional linear parabolic equation can be mapped to the heat equation if and only if it admits a six-dimensional Lie algebra of point symmetries. Consequently, Kolmogorov's original question essentially revolves around three interconnected problems, which form the core of the present paper: the reducibility of a one-dimensional stochastic process to a Brownian motion; the exact correspondence between stochastic transformations and those on the associated KBE; and the classification of the symmetry algebras of a one-dimensional KBE (and, by extension, its underlying SDE).
\\
\\
The primary novelty of this work lies in its inherently probabilistic methodology. We approach the reducibility problem by operating directly on the SDE, subsequently demonstrating a probabilistically meaningful correspondence with the transformations of the associated KBE. This highlights a fundamental limitation in the classical deterministic approaches of \cite{bluman, cherkasov1957, ricciardi}: stochastic transformations of SDEs are intrinsically richer than those of deterministic partial differential equations (PDEs). A diffusion process is not governed solely by its local dynamics in the state space; its probabilistic structure also crucially relies on the driving Brownian motion, the underlying filtration, and the specific probability measure. Thus, alongside spatial transformations and time reparametrizations, it is natural to incorporate transformations of the driving noise and changes of the probability measure.

This observation naturally raises the following question: which transformations of an SDE preserve its analytic counterpart, the Kolmogorov equation, in a structurally meaningful sense? While the celebrated Feynman--Kac formula links these two domains by representing KBE solutions as conditional expectations of diffusion functionals, this does not imply that their natural transformation theories coincide. As noted in the literature, deterministic space-time transformations yield a restrictively narrow class of stochastic transformations \cite{simmetriekolmogorov, gaetarodriguez1}, whereas arbitrary Girsanov measure changes are too broad to preserve the specific parabolic structure of the KBE \cite{paper2020}. The main objective of this work is to identify the precise intermediate class of stochastic transformations and formulate precisely its analytic counterpart within the KBE framework.

To this end, we define a general stochastic transformation of a one-dimensional SDE as a composition of three distinct operations: a deterministic time change, a space-time diffeomorphism, and a Girsanov-type change of probability measure. We formulate their action on both the SDE solutions and the equation's coefficients to ensure structural consistency. Specifically, if the pair $(X,W)$ solves the original SDE, the transformed pair $(P_T(X),P_T(W))$ must solve the transformed SDE with respect to the new filtered probability space. This finite-transformation perspective is highly natural from a probabilistic standpoint, as it simultaneously tracks the transformed process, the new Brownian motion, and the Radon--Nikodym density governing the change of measure.
\\
\\
The central result of this paper is the establishment of a rigorous correspondence between stochastic transformations of one-dimensional SDEs and form-preserving transformations of the associated KBE (Theorem \ref{teo quasi doob}). On the analytic side, we designate as \textit{admissible transformations} those that map a linear parabolic equation into another linear parabolic equation, allowing for the potential addition of a zero-order term. These act on the dependent variable via a multiplicative factor and, optionally, an additive solution term reflecting the equation's linearity. On the stochastic side, we demonstrate that the corresponding operations are \textit{quasi-Doob transformations}, also known as generalized Doob transformations (see, e.g., \cite{chetrite2015nonequilibrium}). Specifically, the multiplicative factor in the PDE transformation corresponds precisely to a Girsanov density of quasi-Doob type for the SDE. In dimension one, under standard non-degeneracy assumptions on the diffusion coefficient, every Girsanov-type transformation can be locally expressed in this form \cite{paper2020}.

A distinguished subclass emerges when the KBE transformation strictly preserves the structure without introducing a potential term. We demonstrate that this rigid structural preservation corresponds precisely to \textit{Doob-type transformations} on the stochastic side (Corollary \ref{doob kbe fin}). Consequently, Doob transformations represent the exact probabilistic analogue of transformations that rigidly preserve the Kolmogorov backward structure, thereby identifying the optimal intermediate class between deterministic space-time transformations and unconstrained Girsanov measure changes. We recall that Doob changes of measure are associated with harmonic functions of the infinitesimal generator and naturally arise in the context of conditioning Markov processes \cite{chetrite2015nonequilibrium}.

Furthermore, we establish a direct probabilistic interpretation of this correspondence. If the multiplicative transformation of the PDE is denoted by $\hat{u}(\hat{x},\hat{t}) = \theta(x,t)u(x,t)$, then its stochastic counterpart, $u(\hat{x},\hat{t}) = \frac{\hat{u}(\hat{x},\hat{t})}{\theta(x,t)}$, is precisely Bayes' rule for conditional expectations under a change of probability measure (Theorem \ref{teo interpretazione probabilistica}). Under a quasi-Doob transformation, the resulting KBE acquires a potential term, and the corresponding Feynman--Kac representation absorbs the associated exponential weight. In the strict Doob case, the potential term vanishes, and the Radon--Nikodym density reduces to the classical Doob form \eqref{dQdP D}, affording a transparent probabilistic meaning to the multiplicative component of parabolic transformations. We also detail the infinitesimal counterparts of these correspondences for both quasi-Doob and Doob cases (Theorem \ref{doob kbe inf}).
\\
\\
Having fully resolved the structural correspondence between SDE and KBE transformations, it is natural to ask why our framework incorporates broader Girsanov-type changes of measure, given that the most pertinent intermediate classes are of (quasi-)Doob type. The answer is deeply tied to Kolmogorov's original reducibility problem. We demonstrate that, under standard regularity and non-degeneracy assumptions, any one-dimensional diffusion can be reconstructed from a standard Brownian motion by coupling a Lamperti transformation with an appropriate Girsanov change of measure. This provides a canonical methodology for viewing Brownian motion as a universal reference model from which general one-dimensional diffusions can be derived.

Consequently, our answer to Kolmogorov's question of when a stochastic diffusion is reducible to a Wiener process is, essentially, ``always'' - provided that standard non-degeneracy holds and probability measure changes are permitted. At first glance, this might appear to contradict the conclusions of \cite{bluman, ricciardi, cherkasov1957} and, notably, Lie's condition \cite{lie1881}, given that there exist KBEs with non-maximal symmetry algebras. However, this apparent paradox is resolved by our purely probabilistic perspective. By establishing the map that transforms Brownian motion into a given SDE, we leverage our transformation correspondence theorem to determine the image of the corresponding heat equation. The necessary addition of a zero-order potential term in the resulting KBE guarantees that the number of symmetries is rigorously preserved, fully respecting Lie's theorem.

For example, applying a change of measure $\mathbb{P} \mapsto \mathbb{Q}$ to a $\mathbb{P}$-Brownian motion, $dX_t = dW_t$, with the Radon--Nikodym density
\begin{equation}
    \left.\frac{d\mathbb{Q}}{d\mathbb{P}}\right|_{\mathcal{F}_t} = \exp \left( \int_0^t \frac{2}{X_s} \, dW_s - \int_0^t \frac{2}{X^2_s} \, ds \right),
\end{equation}
yields the transformed SDE
\begin{equation}\label{SDE2 intro}
    dX_t = \frac{2}{X_t} \, dt + dW^{\mathbb{Q}}_t,
\end{equation}
where $W^{\mathbb{Q}}_t = W^{\mathbb{P}}_t - \int_0^t \frac{2}{X_s} \, ds$ is a $\mathbb{Q}$-Brownian motion via Girsanov's theorem. The KBE associated with this SDE is
\begin{equation}
    u_t + \frac{2}{x} u_x + \frac{1}{2} u_{xx} = 0,
\end{equation}
which admits an algebra of 4 (rather than 6) symmetries. Lie's result is preserved because our correspondence theorem reveals that the mapping from Brownian motion to the SDE \eqref{SDE2 intro} corresponds to mapping the heat equation to
\begin{equation}
    u_t + \frac{2}{x} u_x + \frac{1}{2} u_{xx} + \frac{1}{x^2} u = 0.
\end{equation}
The inclusion of the potential term $\frac{1}{x^2} u$ precisely increases the symmetry count from 4 to 6.
\\
\\
This structural insight directly informs the third focus of this paper: the classification of the symmetries of a one-dimensional KBE via a purely probabilistic approach, and their exact correspondence with the symmetries of the underlying SDE. Beyond its utility in reducibility problems, symmetry analysis provides powerful tools for model reduction, classification, and explicit integration \cite{Olver,olver2,Stephani,gaetanonlinear}. While Lie symmetries robustly characterize deterministic invariants, defining symmetries for SDEs requires considerably more delicacy.

When restricted to symmetry transformations, the mapping established above yields a one-to-one correspondence between the stochastic symmetries of the SDE and the Lie point symmetries of the associated KBE. Crucially, Doob symmetries of the SDE correspond to genuine Lie point symmetries of the KBE, while quasi-Doob symmetries correspond to transformations that preserve the parabolic structure only up to a potential term (Corollary \ref{corollario quasi doob}). This clarifies why space-time diffeomorphisms yield too few stochastic symmetries, while arbitrary Girsanov transformations yield too many, to perfectly mirror the deterministic Lie theory of the PDE.

Finally, reconstructing a diffusion from Brownian motion induces a natural transport mechanism for symmetries, enabling the symmetry structure of a general diffusion to be derived directly from that of the Wiener process (Proposition \ref{sym sde generiche}). This perspective culminates in a complete, constructive characterization of the possible symmetry structures in the autonomous one-dimensional setting. We prove that only three regimes exist, corresponding to maximal, intermediate, and trivial symmetry algebras (Theorem \ref{classification theorem}, Theorem \ref{classification mu sigma}). By applying a Lamperti transform to normalize the diffusion coefficient and subsequently transporting the symmetry structure from Brownian motion, we obtain a fully constructive classification scheme.

From a fundamentally probabilistic viewpoint, our results elegantly recover the classical symmetry classifications of one-dimensional parabolic equations traditionally obtained via deterministic methods \cite{cicogna,gungor}. Rather than solving the determining equations of the PDE, we deduce the symmetries through the transport of Brownian motion invariants coupled with the Doob condition. The defining contribution of the present work thus lies in this unifying probabilistic approach and the complete structural isomorphism established between the stochastic and analytic domains.
\\
\\
\\
\\
The paper is organized as follows. In Section~1 we recall the reconstruction
of one-dimensional diffusions from Brownian motion through Lamperti
transformations and Girsanov-type changes of measure. We also introduce Doob
and quasi-Doob transformations and the general class of stochastic
transformations used throughout the paper. In Section~2 we establish the main
correspondence between stochastic transformations of SDEs and admissible
transformations of the associated KBE, including the
specific Doob case as well as the probabilistic interpretation in terms of
Feynman--Kac representations and Bayes' rule. We then formulate the
infinitesimal version of this correspondence through the associated flow
reconstruction. We specialize the general theory
to symmetries in Section \ref{section symmetries}, where we show how symmetries may be transported from Brownian motion
to general one-dimensional SDEs, and we discuss the  relation with Lie
point symmetries of the corresponding KBE. In Section \ref{sezione classificazione} we provide the complete classification of Lie symmetries of a given autonomous SDE. The theory is finally applied to several classical stochastic models, including Bessel processes, Ornstein--Uhlenbeck dynamics, and the CIR
model in Section \ref{sezione esempi}.

\section{Reconstruction of one-dimensional SDEs from Brownian motion}
In this section we recall the classical result that every one-dimensional SDE with constant-sign diffusion coefficient can be reconstructed from  one-dimensional Brownian motion via a stochastic transformation $T$ that combines a spatial diffeomorphism and a Girsanov-type change of the underlying probability measure. The diffeomorphism, which is related to the well-known classical \textit{Lamperti transformation} (see Section \ref{sezione lamperti}), primarily handles the diffusion term, while the change of measure allows the introduction of an additive drift term.

In order to fix notation, let us recall the definition of one-dimensional SDE.

\begin{definizione}[Weak and strong solutions]\label{weak and strong solutions}
Let \(T>0\), \(x\in\mathbb{R}\), and let
\[
\mu,\sigma:\mathbb{R}\times[0,T]\longrightarrow\mathbb{R}
\]
be Borel measurable functions. Consider the one-dimensional stochastic differential equation
\begin{equation}\label{SDE}
    dX_t=\mu(X_t,t)\,dt+\sigma(X_t,t)\,dW_t,
    \qquad X_0=x.
\end{equation}

A \emph{weak solution} of \eqref{SDE} is a filtered probability space
\[
(\Omega,\mathcal F,(\mathcal F_t)_{t\in[0,T]},\mathbb P),
\]
satisfying the usual conditions, together with an
\((\mathcal F_t)\)-Brownian motion \(W\) and an
\((\mathcal F_t)\)-adapted continuous process \(X\), such that
\[
\int_0^T |\mu(X_s,s)|\,ds<\infty,
\qquad
\int_0^T |\sigma(X_s,s)|^2\,ds<\infty
\quad \mathbb P\text{-a.s.},
\]
and, for every \(t\in[0,T]\),
\begin{equation}\label{sol}
    X_t
    =
    x+\int_0^t\mu(X_s,s)\,ds
    +\int_0^t\sigma(X_s,s)\,dW_s,
    \qquad \mathbb P\text{-a.s.}
\end{equation}

A \emph{strong solution} of \eqref{SDE}, with respect to a given Brownian motion \(W\) on a given filtered probability space
\[
(\Omega,\mathcal F,(\mathcal F_t)_{t\in[0,T]},\mathbb P),
\]
is a continuous process \(X\), adapted to the completed natural filtration generated by \(W\), such that the above integrability conditions hold and \eqref{sol} is satisfied for every \(t\in[0,T]\), \(\mathbb P\)-almost surely.

The SDE \eqref{SDE} is closely connected to a second-order linear parabolic PDE, known as the Kolmogorov backward equation  (see, e.g., \cite[Chapter 8]{oksendal}):
\begin{equation}\label{kbe a}
    u_t + \mu(x,t)u_x + \frac{1}{2}\sigma^2(x,t)u_{xx} = 0,
\end{equation}
which is commonly written in the compact form
\[ L_t u = 0, \]
where 
\begin{equation}\label{generator}
L_t = \partial_t + \mu\partial_x + \frac{1}{2}\sigma^2 \partial_{xx}
\end{equation}
is the space-time infinitesimal generator associated with \eqref{SDE}.

\end{definizione}
\begin{osservazione}
If $X$ is a strong solution, the filtered probability space and the version of the Brownian motion are given a priori, and the solution $X_t$ is then constructed accordingly, ensuring that it is adapted to the filtration generated by the given Brownian motion. On the other hand, if $ X$ is a weak solution, only the functions $\mu$ and $\sigma$ are specified in advance, and we seek a pair $ (X,W)$ along with a filtered probability space such that $ X$ is adapted to the filtration and 
$W$ is a Brownian motion with respect to it. In this case, $ X$ does not necessarily have to be adapted to the filtration generated by $ W.$ A strong solution is of course also a weak solution, but the converse is not true in general (see \textit{Tanaka equation} in \cite{oksendal}).
\end{osservazione}

In the following we assume that there always exists a weak solution to equation \eqref{SDE}.

\subsection{The Lamperti transformation of an SDE}\label{sezione lamperti}

In this section, we briefly recall some fundamental facts about the Lamperti transformation. The core idea behind this procedure is to apply a suitable spatial mapping $F$ to the solution of an SDE in order to obtain, via It\^{o}'s lemma, a new stochastic dynamics characterized by a constant diffusion coefficient. This transformation is widely employed to bypass theoretical and numerical issues typically arising from state-dependent diffusion coefficients. Since our ultimate goal is to relate general SDEs to the standard Brownian motion, our first step is to exploit this transformation to reduce the original SDE to one with a unit diffusion coefficient.

Following mainly \cite{luschgy2006functional}, we introduce the Lamperti transformation by assuming the following regularity conditions:
\medskip

\noindent\textbf{Assumption A }
For a diffusion process $X$ taking values in an open interval of the real line $I \subseteq \mathbb{R}$, we assume that:
\begin{enumerate}
    \item $\sigma \in \mathcal{C}^1(I \times [0,T] ,\R)$
    \item $ |\mu(x,t)| \leq C(1+|x|)$ and  $ 0<\sigma(x,t) \leq C(1+|x|)$ $\forall (x,t)\in I \times [0,T]. $
\end{enumerate}

\medskip

\begin{definizione}[Lamperti transformation]\label{def Lamperti transform}
    Under Assumption A, if \( I\equiv \R\) the Lamperti transformation is defined for every \( (x,t) \in \R \times [0,T]\) by 
    \begin{equation}\label{Lamperti}
    F(x,t) := \int_{x_0}^x \frac{1}{\sigma(\xi,t)}d\xi,
\end{equation}
where $x_0 \in \R$ is an arbitrary constant, typically chosen such that the transformed process satisfies $F(X_0,0)=0$.

On the other hand, if \( I \subsetneq \R\) the Lamperti transformation is defined for every \( (x,t) \in I \times [0,T]\) by
  \begin{equation}\label{Lamperti estesa}
    F(x,t) := \int_{x_1}^x \frac{1}{\sigma(\xi,t)}d\xi,
\end{equation}
where \( x_1\) is an arbitrary fixed value lying inside \( I\) such that
\[ \forall t \in [0,T], \quad \int_{[x_1,+\infty)\cap I} \frac{d\xi}{\sigma(\xi,t)}=\infty \quad \text{and} \quad \int_{(-\infty,x_1]\cap I} \frac{d\xi}{\sigma(\xi,t)}=\infty. \]
\end{definizione}

\medskip

\begin{teorema}[Well-posedness of the Lamperti transformation]\label{th lamperti}
Let $X_t$ be an It\^{o} process with dynamics governed by \eqref{SDE}, satisfying Assumption A. Let $F(x,t)$ be the Lamperti transformation given in Definition \ref{def Lamperti transform}. Then, the transformed process $Z_t := F(X_t,t)$ satisfies the following SDE with a unit diffusion coefficient:
\begin{equation*}
	dZ_t = \biggl(\frac{\partial F(F^{-1}(Z_t,t),t)}{\partial t}+ \frac{\mu(F^{-1}(Z_t,t),t)}{\sigma(F^{-1}(Z_t,t),t)}-\frac{1}{2}\frac{\partial \sigma (F^{-1}(Z_t,t),t)}{\partial x}\biggr) dt + dW_t.
\end{equation*}
\end{teorema}

\begin{proof} See, e.g., \cite{luschgy2006functional}.

\end{proof}

\begin{osservazione}\label{invertible lamperti}
    For every \( t \in [0,T]\), the map \( x \mapsto F(x,t)\) is continuous, strictly increasing and one-to-one from \(\R\) to \(\R\). Furthermore, it has a continuous and differentiable inverse denoted by \( F^{-1}(x,t)\). See \cite{luschgy2006functional} for a complete discussion.
\end{osservazione}

\begin{osservazione}
    Although multidimensional generalizations are possible under highly restrictive integrability conditions, the Lamperti transformation is inherently a one-dimensional procedure. For this reason, the focus of the present paper is restricted to one-dimensional SDEs.
\end{osservazione}

\begin{osservazione}\label{remark segno sigma}
    Regarding the regularity conditions in Assumptions A, the strict positivity of $\sigma$ is assumed primarily for convenience and without loss of generality. The fundamental requirement for the Lamperti transformation is that the diffusion coefficient must not vanish on the reference interval. Therefore, one could equally assume $\sigma$ to be strictly negative.
\end{osservazione}

\begin{osservazione}
    In the autonomous case (i.e., \( \sigma=\sigma(x)\) ), the regularity assumptions on \(\sigma\) may be relaxed into
    \[ \sigma(y)-\sigma(x)=\int_x^y \sigma'(u)du, \quad \sigma' \in L^1_{loc}(dx),\]
    where \(dx\) denotes the classical Lebesgue measure on \(\R\). See \cite{luschgy2006functional} for a comprehensive discussion.
\end{osservazione}

In the remainder of this paper, we will implicitly assume that the underlying regularity conditions ensuring the well-posedness of the Lamperti transformation hold.

 \subsection{ Doob transformation of an SDE}\label{section_3}
We briefly recall the concepts of Doob transformations and quasi-Doob transformations (also known in the literature as generalized Doob transformations), mainly following \cite{chetrite2015nonequilibrium, paper2020, paperriduzione}. In the following, we consider only conservative transformations. For non-conservative (generalized) Doob transformations, see \cite{chetrite2015nonequilibrium}. As we will see in Section \ref{Section sde kbe}, the class of Doob and quasi-Doob transformations provides a natural setting when one wants to investigate the link between the transformations of a given SDE and those of the related KBE, also providing a probabilistic interpretation of the correspondence between the stochastic setting and the deterministic framework.

Although they can be introduced in a more general setting, for our future convenience we present Doob and quasi-Doob transformations as special cases of Girsanov-type transformations. These were studied in \cite{paper2020} and are briefly recalled at the beginning of this section.
\medskip

Let \( h:\mathbb{R}\times \mathbb{R}_+ \rightarrow \mathbb{R}\). We denote by \( Z_t\) the Doléans-Dade exponential defined by the process:
\begin{equation}\label{mtg esponenziale}
    Z_t = \exp\left(\int_0^t h(X_s,s)dW_s - \frac{1}{2}\int_0^t h^2(X_s,s)ds\right).
\end{equation}
Then, the following result is well-known (see, e.g., \cite{baldi}):

\begin{teorema}[Girsanov]
    Let \( W=(W_t)_t\) be a Brownian motion with respect to the filtered probability space \( (\Omega,\mathcal{F},(\mathcal{F}_t)_t,\mathbb{P})\) and assume that the Doléans-Dade exponential defined in \eqref{mtg esponenziale} is a \(\mathbb{P}\)-martingale. Then, the process \( W'=(W'_t)_t\) defined by
    \[ W'_t = W_t - \int_0^t h(X_s,s)ds\]
    is a Brownian motion with respect to the filtered probability space \( (\Omega,\mathcal{F},(\mathcal{F}_t)_t,\mathbb{Q}) \), where the Radon-Nikodym derivative of the probability measure \(\mathbb{Q}\) with respect to the original measure \( \mathbb{P}\) is given by
    \[ \left.\frac{d\mathbb{Q}}{d\mathbb{P}}\right|_{\mathcal{F}_t} = Z_t.\]
\end{teorema}

In general, proving that \( Z\) is a \(\mathbb{P}\)-martingale in order to apply Girsanov's theorem is not an easy task, which is why the literature often uses sufficient conditions known as Novikov's criterion or Kazamaki's criterion (see, e.g., \cite{baldi}). In this context, the same goal can be achieved using a different condition, expressed in the following definition and applied in the subsequent lemma (see \cite{karatzas-shreve} for a detailed discussion).

\begin{definizione}\label{def_non_explosive}
    A smooth SDE with coefficients \(\mu\) and \(\sigma\) (denoted as \(\mathrm{SDE}_{\mu,\sigma}\)) is called non-explosive if any solution \( (X,W)\) is defined for all times \( t \geq 0\). A smooth function \( h\) is called non-explosive for the non-explosive \( \mathrm{SDE}_{\mu,\sigma}\) if the \( \mathrm{SDE}_{\mu+\sigma h, \sigma}\) is a non-explosive SDE.
\end{definizione}
    
\begin{lemma}\label{lemma non expl}
    Suppose that \( \mathrm{SDE}_{\mu,\sigma}\) is non-explosive with solution \( (X,W)\), and let \( h\) be a non-explosive function. Then the Doléans-Dade exponential \(Z_t\)  associated with \( h\) via \eqref{mtg esponenziale} is a (global) \(\mathbb{P}\)-martingale.
\end{lemma}

With the aid of Girsanov's theorem, we can define a particular stochastic transformation of a given SDE as a change of the underlying probability measure, identified by the function \( h\) inside the Radon-Nikodym derivative \( Z\).

\begin{definizione}  \label{trasformazione di girsanov} 
    Consider a non-explosive \( \mathrm{SDE}_{\mu,\sigma}\), where \( W_t\) is a \( (\Omega, \mathcal{F},(\mathcal{F}_t)_t,\mathbb{P})\)-Brownian motion, and let \( h\) be a non-explosive function for \( \mathrm{SDE}_{\mu,\sigma}\). We define \( T\), denoted specifically by \(h\), as a Girsanov-type measure change if it changes the original probability measure from \( \mathbb{P}\) to \( \mathbb{Q}\), with Radon-Nikodym derivative
    \[ \left.\frac{d\mathbb{Q}}{d\mathbb{P}}\right|_{\mathcal{F}_t} = Z_t = \exp\left(\int_0^t h(X_s,s)dW_s - \frac{1}{2}\int_0^t h^2(X_s,s)ds\right). \]
\end{definizione}

\begin{teorema}[Well-posedness of Girsanov-type change of measure]\label{teo trasf girsanov}
    Let \( (X,W)\) be a solution to the non-explosive \( \mathrm{SDE}_{\mu,\sigma}\) on the filtered probability space \( (\Omega, \mathcal{F},(\mathcal{F}_t)_t,\mathbb{P})\), and let \( h\) be a non-explosive function for \( \mathrm{SDE}_{\mu,\sigma}.\) Then, \( (X,W')\) is a solution to the \( \mathrm{SDE}_{\mu',\sigma'}\) on the filtered probability space \( (\Omega, \mathcal{F},(\mathcal{F}_t)_t,\mathbb{Q})\), where
    \begin{align*}
        & W'_t = W_t - \int_0^t h(X_s,s)ds, \\
        & \mu' = \mu + \sigma h, \quad \sigma' = \sigma, \\
        & \left.\frac{d\mathbb{Q}}{d\mathbb{P}}\right|_{\mathcal{F}_t} = Z_t = \exp\left(\int_0^t h(X_s,s)dW_s - \frac{1}{2}\int_0^t h^2(X_s,s)ds\right).
    \end{align*}
\end{teorema}

\begin{proof}
    This result was proved in \cite{paper2020}, but we give a sketch of the proof for the sake of clarity. From Lemma \ref{lemma non expl}, \( Z_t\) is a \(\mathbb{P}\)-martingale. By Girsanov's theorem, \( W'\) is a \(\mathbb{Q}\)-Brownian motion. Finally, since \( (X,W)\) solves the \( \mathrm{SDE}_{\mu,\sigma}\), we have:
    \begin{align*}
        dX_t &= \mu(X_t,t)dt + \sigma(X_t,t)dW_t \\
        &= \mu(X_t,t)dt + \sigma(X_t,t)(dW_t - h(X_t,t)dt + h(X_t,t)dt) \\
        &= (\mu(X_t,t) + \sigma(X_t,t)h(X_t,t))dt + \sigma(X_t,t)(dW_t - h(X_t,t)dt) \\
        &= \mu'(X_t,t)dt + \sigma'(X_t,t)dW'_t.
    \end{align*}
\end{proof}

A subclass of particular interest among Girsanov-type changes of measure, which will be a fundamental tool in the remainder of this paper, is given by Doob transformations (and their extensions).

\begin{definizione}\label{def doob e qd}
    Let \( \phi\) be a strictly positive function. The stochastic transformation of a diffusion process \( X_t\) that produces a new process with infinitesimal generator given by
    \begin{equation}
        L_t^{\phi} := \phi^{-1}L_t \phi - \phi^{-1}(L_t\phi),
    \end{equation}
    is called a quasi-Doob transformation characterized by the function \(\phi\). In this expression, \( \phi^{-1} L_t \phi\) must be understood as the composition of three different operations: the multiplication operator by \(\phi^{-1}\), the operator \( L_t\) itself, and the multiplication operator by \(\phi\). That is, for every sufficiently smooth function \( g\),
    \[ L_t^\phi g = \phi^{-1} L_t(\phi g) - \phi^{-1} (L_t\phi)g.\]
    This transformation also involves a change of the underlying probability measure, passing from an original \( \mathbb{P} \) to a measure \( \mathbb{Q} \) whose Radon-Nikodym derivative is given by
    \[ \left.\frac{d\mathbb{Q}}{d\mathbb{P}}\right|_{\mathcal{F}_t} = \phi^{-1}(X_0,0) \exp\left(-\int_0^t \phi^{-1}L_t\phi (X_s,s)ds\right)\phi(X_t,t).\]
    If, in addition, \(\phi\) is a harmonic function for \( L_t\), that is,
    \[ L_t\phi = 0, \]
    then the transformation is called a Doob-type transformation (or simply, a Doob transformation). In this case, of course, the transformed generator reduces to
    \[ L_t^\phi = \phi^{-1}L_t\phi, \]
    and the Radon-Nikodym derivative to
    \[ \left.\frac{d\mathbb{Q}}{d\mathbb{P}}\right|_{\mathcal{F}_t} = \phi^{-1}(X_0,0)\phi(X_t,t).\]
\end{definizione}

It is possible to view both quasi-Doob and Doob transformations as particular cases of the Girsanov-type change of measure introduced in Definition \ref{trasformazione di girsanov}, by introducing a function \(k \coloneqq \ln\phi\). With this notation, we can rewrite the previous definition by stating that a quasi-Doob transformation related to the function \(k \) transforms the infinitesimal generator \( L_t\) to
\[ L_t^k = e^{-k}L_te^{k} - e^{-k}(L_te^k), \]
and the probability measure from \(\mathbb{P}\) to \( \mathbb{Q}\) with
\begin{equation}\label{dQdP qd} 
    \left.\frac{d\mathbb{Q}}{d\mathbb{P}}\right|_{\mathcal{F}_t} = \exp{({{k}(X_t,t)-k(X_0,0)})}\exp\left(-\int_0^t e^{-k(X_s,s)}L_t(e^{k})(X_s,s)ds\right).
\end{equation}
On the other hand, if in addition
\( L_t(e^{k}) = 0, \)
a Doob transformation characterized by \(k \) transforms the infinitesimal generator \( L_t\) to
\[ L_t^k = e^{-k}L_te^{k}, \]
and the original probability \(\mathbb{P}\) to \(\mathbb{Q}\) with
\begin{equation}\label{dQdP D}
    \left.\frac{d\mathbb{Q}}{d\mathbb{P}}\right|_{\mathcal{F}_t} = \exp{({{k}(X_t,t)-k(X_0,0)})}.
\end{equation}

By comparing \eqref{dQdP qd} and \eqref{dQdP D} with Definition \ref{trasformazione di girsanov}, it is straightforward to verify that Doob and quasi-Doob transformations are particular Girsanov-type changes of measure, where the exponent of the Radon-Nikodym derivative admits a more explicit expression related to the function \( k \). Indeed, we have the following characterization:

\begin{proposizione}\label{caratterizzazione finita d e qd}
    Let \( h\) be a Girsanov-type change of measure for \( \mathrm{SDE}_{\mu,\sigma} \) in the sense of Definition \ref{trasformazione di girsanov}. Then:
    \begin{itemize}
        \item \( h\) is a Doob transformation if and only if there exists a function \(\ k \) such that
        \[ h = \sigma k_x \quad \text{and} \quad L_t(k) = -\frac{1}{2}h^2 \quad (\text{or, equivalently, } L_t(e^k)=0).\]
        \item \( h\) is a quasi-Doob transformation if and only if there exists a function \( k \) such that
        \[ h = \sigma k_x \quad \text{and} \quad L_t( k ) + \frac{1}{2}h^2 = e^{-k}L_t(e^{k}).\]
    \end{itemize}
\end{proposizione}

\begin{proof}
    The proof can be (partially) found in \cite{paperriduzione} for the quasi-Doob case and in \cite{paper2020} for the Doob one. We repeat the proof for the first assertion for completeness, adapting it to the one-dimensional setting in which we are working. The proof for the quasi-Doob transformation is an easy generalization.
    
    Let \( h\) be a Doob transformation. Then, from Definition \ref{def doob e qd} and by setting \( k=\ln\phi\), there exists a function \(k\) such that
    \begin{equation}\label{densità doob proof}
        \left.\frac{d\mathbb{Q}}{d\mathbb{P}}\right|_{\mathcal{F}_t} = \exp\left(k (X_t,t)- k(X_0,0)\right).
    \end{equation}
    On the other hand, from Definition \ref{trasformazione di girsanov}, since \( h\) is by hypothesis a Girsanov-type change of measure, we have that
    \begin{equation}\label{densità girs proof}
        \left.\frac{d\mathbb{Q}}{d\mathbb{P}}\right|_{\mathcal{F}_t} = \exp\left(\int_0^t h(X_s,s)dW_s - \frac{1}{2}\int_0^t h^2(X_s,s)ds\right).
    \end{equation}
    By equating \eqref{densità doob proof} and \eqref{densità girs proof}, we obtain
    \[ \int_0^t h(X_s,s)dW_s - \frac{1}{2}\int_0^t h^2(X_s,s)ds =  k (X_t,t)-k(X_0,0). \]
    However, by applying Itô's formula to the function \( k \), we get
    \[  k (X_t,t)- k (X_0,0) = \int_0^t L_t k (X_s,s)ds + \int_0^t \sigma  k_x (X_s,s) dW_s. \]
    Thus, the first implication follows by the uniqueness of the canonical semimartingale decomposition for continuous processes and by the uniqueness in the martingale representation theorem for processes adapted to Brownian filtrations (see, e.g., Theorem 36.1 in \cite{rogers}).
    
    Let us now assume that \( h\) is a Girsanov-type change of measure (in the sense of Definition \ref{trasformazione di girsanov}) and that there exists a function \( k \) such that \( h=\sigma  k_x\) and \( L_t(k) = -\frac{1}{2} h^2.\) We want to prove that \( h\) is actually a Doob transformation characterized by \( k \) itself. Theorem \ref{teo trasf girsanov} ensures that under the transformation given by \( h\), the diffusion \( X\) satisfies a new SDE with drift
    \[ \mu' = \mu + \sigma h \]
    and the same diffusion coefficient \(\sigma\). Thus, in particular, the transformed infinitesimal generator is given by
    \[ \bar{L}_t = L_t + \sigma h \partial_x = L_t + \sigma^2 k_x \partial_x. \]
    By setting \( \phi = e^{ k }\),
    \[ \bar{L}_t = L_t + \sigma^2\frac{\phi_x}{\phi}\partial_x. \]
    Now, we show that \( \bar{L}_t\) coincides with \(L_t^\phi\) as in Definition \ref{def doob e qd}. Indeed, for every sufficiently smooth function \( g\),
    \[ \bar{L}_tg = L_tg + \sigma^2\frac{\phi_x}{\phi}\partial_x g, \]
    while by construction
    \[ L_t^\phi g = \phi^{-1} L_t(\phi g) = \phi^{-1} \bcancel{(L_t\phi)}g + \bcancel{\phi^{-1}\phi}L_tg + \sigma^2\frac{\phi_x}{\phi}\partial_xg = L_tg + \sigma^2\frac{\phi_x}{\phi}\partial_xg, \]
    since \( \phi=e^{ k }\) and the hypothesis \( L_t(k ) = -\frac{1}{2} h^2\) is equivalent to requiring that \( L_t(e^k) = 0\). Thus, \( \bar{L}_t = L_t^\phi\) and the first requirement of Definition \ref{def doob e qd} regarding Doob transformations is fulfilled. Now, we have to prove that 
    \[ \left.\frac{d\mathbb{Q}}{d\mathbb{P}}\right|_{\mathcal{F}_t} = \phi^{-1}(X_0,0)\phi(X_t,t). \]
    Since \( h\) is a Girsanov-type change of measure,
    and by hypothesis \( h = \sigma k_x\), we get
    \[ \frac{d\mathbb{Q}}{d\mathbb{P}} = \exp\left(\int_0^t \sigma k_x(X_s,s)dW_s - \frac{1}{2}\int_0^t h^2(X_s,s)ds\right). \]
    By Itô's formula,
    \[ \int_0^t \sigma  k_x(X_s,s)dW_s =  k(X_t,t) -  k (X_0,0) - \int_0^t L_tk(X_s,s)ds, \]
    and since by hypothesis \( L_t k  + \frac{1}{2}h^2 = 0\),
    we obtain
    \[  \frac{d\mathbb{Q}}{d\mathbb{P}} = \exp\left( k (X_t,t) - k(X_0,0)\right), \]
    and the thesis follows by replacing \( k  = \ln\phi\).
\end{proof}

\begin{osservazione}
     It is important to note that the fact that a change of measure \(h\) is a
Doob transformation strongly depends on \( SDE_{\mu,\sigma}\). Indeed, in Proposition \ref{caratterizzazione finita d e qd} , the right-hand
side of equation \( h=\sigma k_x\) depends explicitly on \(\sigma\) and the left hand side of equation \( L_t(k)=-\frac{1}{2}h^2\) depends, through the operator \( L_t\), on both \(\mu\) and \(\sigma\). Thus, in the following, when we refer to a Doob transformation, we will always implicitly assume that it is of Doob-type with respect to the original $SDE$ (and not the transformed one).
\end{osservazione}

\begin{osservazione}
    It is immediate to verify that in one dimension, under the hypothesis that the diffusion coefficient \(\sigma\) does not vanish (which is the assumption we made in Section \ref{sezione lamperti} for the well-posedness of the Lamperti map), any Girsanov-type change of measure is a quasi-Doob transformation, since it is always possible to integrate and find a function \(k \) such that \(h = \sigma k_x\). Of course, this does not hold in higher dimensions, where it is possible to define Girsanov-type changes of measure which are neither of Doob type nor quasi-Doob. However, as clarified in the previous section, in the current paper we will always focus on dimension one.
\end{osservazione}

\begin{osservazione}

Historically, the original problem considered by Doob, which led to the introduction of Doob transformations via \( L^\phi \) (see Definition \ref{def doob e qd}), was to condition a Markov process \( X_t \) started at \( X_0 \) to exit a domain \( D \) exclusively through a subset \( B \) of its boundary \( \partial D = B \cup C \), where \( B \cap C = \emptyset \). Let \( \tau_B \) and \( \tau_C \) denote the respective hitting times of \( B \) and \( C \):
\[
    \tau_B = \inf\{t \geq 0 : X_t \in B\}, \quad \tau_C = \inf\{t \geq 0 : X_t \in C\}.
\]
We define the conditioning event as \( \mathcal{B} = \{\tau_B \leq \tau_C\} \). Let \( \mathbb{P} \) be the path measure of the unconditioned Markov process \( X \), and let \( \mathbb{Q} = \mathbb{P}(\cdot \mid \mathcal{B}) \) be the path measure of the conditioned process. From the fundamental properties of conditional probability, the Radon-Nikodym derivative of \( \mathbb{Q} \) with respect to \( \mathbb{P} \) restricted to the \(\sigma\)-algebra \( \mathcal{F}_T = \sigma(X_t, \, 0 \leq t \leq T) \) is given by:
\[
    \left. \frac{d\mathbb{Q}}{d\mathbb{P}} \right|_{\mathcal{F}_T} = \frac{\mathbb{P}(\mathcal{B} \mid \mathcal{F}_T)}{\mathbb{P}(\mathcal{B})}.
\]
Indeed, for every event \( E \in \mathcal{F}_T \), the tower property of conditional expectation yields:
\[
    \mathbb{Q}(E) = \mathbb{P}(E \mid \mathcal{B}) = \frac{\mathbb{P}(E \cap \mathcal{B})}{\mathbb{P}(\mathcal{B})} = \frac{\mathbb{E}[\mathbf{1}_E \mathbf{1}_{\mathcal{B}}]}{\mathbb{P}(\mathcal{B})} = \frac{\mathbb{E}\big[\mathbb{E}[\mathbf{1}_E \mathbf{1}_{\mathcal{B}} \mid \mathcal{F}_T]\big]}{\mathbb{P}(\mathcal{B})} = \frac{\mathbb{E}\big[\mathbf{1}_E \mathbb{P}(\mathcal{B} \mid \mathcal{F}_T)\big]}{\mathbb{P}(\mathcal{B})} = \int_E \frac{\mathbb{P}(\mathcal{B} \mid \mathcal{F}_T)}{\mathbb{P}(\mathcal{B})} \, d\mathbb{P}.
\]
Let us define the probability of successful exit as \( \phi(x) = \mathbb{P}_x(\tau_B \leq \tau_C) \), and let \( \tau = \min(\tau_B, \tau_C) \). Since \( X_t \) starts at \( X_0 \), the unconditional probability is \( \mathbb{P}(\mathcal{B}) = \phi(X_0) \). By the Markov property, the conditional probability simplifies to \( \mathbb{P}(\mathcal{B} \mid \mathcal{F}_T) = \phi(X_{\min(T, \tau)}) \).

Consequently, for \( T \leq \tau \), the Radon-Nikodym derivative becomes:
\[
    \left. \frac{d\mathbb{Q}}{d\mathbb{P}} \right|_{\mathcal{F}_T} = \frac{\phi(X_T)}{\phi(X_0)}.
\]
This matches the form required by Definition \ref{def doob e qd}, provided that \( \phi \) is a harmonic function for the generator \( L_t \). The harmonicity of \( \phi \) (i.e., \( L_t\phi = 0 \)) follows directly from Dynkin's formula:
\[
    \mathbb{E}_x[\phi(X_\tau)] = \phi(x) + \mathbb{E}_x \left[ \int_0^\tau L_t\phi(X_s) \, ds \right],
\]
combined with the boundary property that \( \phi(x) = \mathbb{E}_x[\phi(X_\tau)] \).

A comprehensive discussion on the connection between Doob transformations and conditioned Markov processes is carried out in \cite{chetrite2015nonequilibrium}. This procedure rigorously yields well-known stochastic processes, such as, for example, the three-dimensional Bessel process (which exhibits a drift of \( 1/x \))  (see Remark \ref{Bessel condizionamento} in Section \ref{bessel processes}).

\end{osservazione}

\subsection{General stochastic transformations}\label{section stochastic transformations}
A stochastic transformation is a mapping that acts both on the solution process of a given SDE (via an operator $P_T$, mapping the pair $(X,W)$ into a new pair $(P_T(X),P_T(W))$) and on the equation itself (via an operator $E_T$, transforming the coefficients $(\mu,\sigma)$ into $(E_T(\mu),E_T(\sigma))$). The fundamental concept is structural consistency: we require that if $(X,W)$ solves the equation $\text{SDE}_{\mu,\sigma}$, then the transformed process $(P_T(X),P_T(W))$ must be a solution to the transformed equation $\text{SDE}_{E_T(\mu),E_T(\sigma)}$, possibly with respect to a new filtered probability space $(\Omega, \mathcal{F}, (\mathcal{F}'_t)_{t \ge 0}, \mathbb{Q})$. Furthermore, for the stochastic transformation to be \textit{admissible}, we require that the transformed process $P_T(W)$ remains a standard Brownian motion with respect to this new filtered probability space.

Examples of such stochastic transformations include absolutely continuous time changes $f:\mathbb{R}_+\rightarrow \mathbb{R}_+$, defined as
\begin{equation}\label{f(t)}
    t \mapsto f(t) \coloneq \int_0^t \eta(s) \, ds, \quad \eta>0;
\end{equation}
space-time functions $\Phi:\mathbb{R}\times \mathbb{R}_+\rightarrow\mathbb{R}$, defined as
\begin{equation}\label{Phi}
    x \mapsto x' \coloneq \Phi(x,t);
\end{equation}
and Girsanov-type changes of the probability measure $\mathbb{P}\mapsto \mathbb{Q}$, characterized by a non-explosive function
\begin{equation}\label{h}
    h: \mathbb{R}\times \mathbb{R}_+\rightarrow\mathbb{R}
\end{equation}
appearing in the Doléans-Dade exponential inside the Radon-Nikodym derivative (see Definition \ref{trasformazione di girsanov}). We can unify all these examples into a general definition, where the global action of a generic stochastic transformation encompasses all these individual effects:

\begin{definizione}\label{def trasf stoc}
    Let $\Phi: \mathbb{R}\times \mathbb{R}_+ \rightarrow \mathbb{R}$, $f: \mathbb{R}_+\rightarrow \mathbb{R}_+$, and $h: \mathbb{R}\times \mathbb{R}_+ \rightarrow \mathbb{R}$ be given as in \eqref{Phi}, \eqref{f(t)}, and \eqref{h} such that
    \begin{equation}\label{invertibility}\Phi_x(x,f(t))\cdot f'(t)\neq 0,\end{equation} and let $(X,W)$ be a solution to $\text{SDE}_{\mu,\sigma}$. Introducing the function \(\Phi^f:\mathbb{R}\times \mathbb{R}_+\rightarrow\mathbb{R} \) given by
    \begin{equation}\label{Phif}
        \Phi^f(x,t)=\Phi(x,f(t)),
    \end{equation} let $\tilde{\Phi}: \mathbb{R}\times \mathbb{R}_+\rightarrow \mathbb{R}\times \mathbb{R}_+$ be a (local) diffeomorphism defined as
    \begin{equation}\label{Phitilde} \tilde{\Phi}(x,t) \coloneq (\Phi^f(x,t),f(t))=(\Phi(x,f(t)),f(t)). \end{equation}
    A stochastic transformation is a pair $T=(\tilde{\Phi},h)$ which maps \( (X,W) \mapsto (P_T(X),P_T(W)),\) where
    \begin{equation}\label{PTX}
        P_T(X)_t \coloneq \Phi(X_{f^{-1}(t)}, t), \quad P_T(W)_t \coloneq \int_0^{f^{-1}(t)} \sqrt{f'(s)} \, (dW_s - h(X_s,s) \, ds)
    \end{equation}
    and its action on the SDE coefficients is given by \( (\mu,\sigma) \mapsto (E_T(\mu),E_T(\sigma))  \), where
    \begin{equation}\label{ET}
        E_T(\mu) \coloneq \Big( \frac{1}{f'} [ L_t(\Phi^f) + \Phi^f_x \sigma  h] \Big) \circ \tilde{\Phi}^{-1}, \quad E_T(\sigma) \coloneq \Big(\frac{1}{\sqrt{f'}} \Phi^f_x \sigma \Big) \circ \tilde{\Phi}^{-1}.
    \end{equation}
\end{definizione}

\begin{osservazione}
Condition \eqref{invertibility} ensures the invertibility of \(\tilde{\Phi}\).
    Since by construction $f'(t)=\eta(t)>0$, it is sufficient to impose that $\Phi^f_x \neq 0$.
\end{osservazione}

Structural consistency and admissibility are guaranteed by the following theorem.

\begin{teorema}[General action of a stochastic transformation]\label{trasformate}
    Let $(X,W)$ be a solution to the non-explosive $\text{SDE}_{\mu,\sigma}$ in the filtered probability space $(\Omega, \mathcal{F}, (\mathcal{F}_t)_{t \ge 0}, \mathbb{P})$. Let $T=(\tilde{\Phi}, h)$ be a stochastic transformation, with \( h \) non-explosive for the SDE (according to Definition \ref{def_non_explosive}). Then, the transformed process $P_T(X,W) = (P_T(X),P_T(W))$ defined in \eqref{PTX} is a solution to the transformed equation $E_T(\text{SDE}_{\mu,\sigma}) = \text{SDE}_{E_T(\mu,\sigma)}=\text{SDE}_{E_T(\mu),E_T(\sigma)}$, where $E_T(\mu)$ and $E_T(\sigma)$ are defined in \eqref{ET}, with $P_{T}(W)$ being a Brownian motion in the space $( \Omega, \mathcal{F}', \mathbb{Q})$, where
    \begin{equation*}
        \frac{d \mathbb{Q}}{d \mathbb{P}} \bigg|_{\mathcal{F}_t} = \exp \Big( \int_0^t h(X_s,s) \, dW_s - \frac{1}{2} \int_0^t h^2(X_s,s) \, ds \Big), \quad \mathcal{F}'_t = \mathcal{F}_{f^{-1}(t)}.
    \end{equation*}
\end{teorema}
\begin{proof}
    With slightly different notations the proof can be found in \cite{paper2020}.
\end{proof}

\begin{osservazione}
    The Lamperti and Doob-Girsanov transformations, studied in Section \ref{sezione lamperti} and Section \ref{section_3}, are special cases of stochastic transformations obtained by setting $T=(\tilde{\Phi},0)$ with \( f=\mathrm{id}\) and \(\Phi=F\) given by \eqref{Lamperti},  and $T=(\mathrm{Id}, h)$, respectively.
\end{osservazione}

\begin{definizione}
    A generic stochastic transformation \(T=(\tilde{\Phi},h) \) is of Doob-type if the change of measure \(\tilde{T}=(Id,h)\) is a Doob transformation, in the sense of Definition \ref{def doob e qd}.
\end{definizione}

\begin{osservazione}
    The non-explosiveness assumption is necessary if we are interested in the well-posedness of the definition of the transformation \( P_T\) on the set of solutions \((X,W)\) to the \(\text{SDE}_{\mu,\sigma}\). If we focus only on
\( \text{SDE}_{\mu,\sigma}\) and on the action \( E_T\)
of the stochastic transformation \(T\), we can define the transformed \(\text{SDE}_{E_T(\mu,\sigma)}\) without making any request on
non-explosiveness (see \cite{paper2020}). Henceforth, unless specified otherwise, the assumptions of non-explosiveness will be implicitly made whenever required.
\end{osservazione}

The set of stochastic transformations can be endowed with a group structure and a compatible differentiable manifold structure, defining a one-parameter group of transformations indexed in $\mathbb{R}$. The group structure, which will be needed in the following, is given by
\begin{itemize}
    \item \textbf{Composition:} Given \( T_1=(\tilde{\Phi}_1, h_1)\) and \( T_2=(\tilde{\Phi}_2,  h_2)\),
    \begin{equation}\label{T composizione}
    T_2\circ{T_1} = \Big(\tilde{\Phi}_2\circ\tilde{\Phi}_1,\;   h_1 + \sqrt{f'_1}(h_2\circ\tilde{\Phi}_1) \Big).
    \end{equation}
    \item \textbf{Inverse:} Given \( T=(\tilde{\Phi}, h) \),
    \begin{equation}\label{T inversa}
    T^{-1} = \Big(\tilde{\Phi}^{-1}, \;  -\left(\frac{h}{\sqrt{f'}} \right)\circ\tilde{\Phi}^{-1} \Big).
    \end{equation}
    \item \textbf{Neutral Element:} \( T_0=(\mathrm{Id}, 0) \).
\end{itemize}

\begin{osservazione}\label{composizione ben def}
    The group structure is consistent with the iteration of transformed processes and transformed coefficients. In particular, it can be proved (see \cite{proceeding},\cite{paper2015}) that
    \[ P_{T_2\circ T_1}(X,W)=P_{T_2}(P_{T_1}(X,W)) \quad and \quad E_{T_2\circ T_1}(\mu,\sigma)=E_{T_2}(E_{T_1}(\mu,\sigma)).\]
\end{osservazione}

In particular, one can consider a standard one-parameter group of stochastic transformations $\{T_\lambda\}_{\lambda \in \mathbb{R}}$ and its corresponding Lie algebra. As per the standard construction, the elements of this algebra are obtained by differentiating the group elements with respect to the flow parameter and evaluating the result at zero \footnote{Recall that a Lie algebra is defined as the tangent space at the identity.}. Shifting to the Lie algebra framework allows for the application of powerful linear algebraic tools.

\begin{definizione}
    Given a one-parameter group of stochastic transformations $T_\lambda=(\tilde{\Phi}_\lambda, h_\lambda)$, we define an infinitesimal stochastic transformation as an element $V=(\tilde{Y},H)$ of the associated Lie algebra, with \(\tilde{Y}=(Y,m)\), where:
    \[
    \begin{matrix}
    Y(x,t)=\partial_\lambda \Phi^{f_\lambda}_\lambda(x,t) \vert_{\lambda=0}, & m(t)=\partial_\lambda f_\lambda(t) \vert_{\lambda=0}, & H(x,t)=\partial_\lambda h_\lambda(x,t) \vert_{\lambda=0}.
    \end{matrix}
    \]
    Here, $\tilde{Y}=Y\partial_x+m\partial_t$ is a time-dependent vector field on $\mathbb{R}\times\mathbb{R}_+$ and $H: \mathbb{R} \times \mathbb{R}_+ \to \mathbb{R}$ is a  smooth function.
\end{definizione}

Conversely, any element of the Lie algebra allows for the recovery of the corresponding one-parameter group of transformations by reconstructing the flow according to the following result:

\begin{teorema}[Reconstruction of the flow]\label{teo ricostruzione del flusso}
Let $V = (\tilde{Y}, H)$ be an infinitesimal stochastic transformation. The corresponding one-parameter group $T_\lambda$ is uniquely determined by the solution to the system:
\begin{equation}\label{ricostruzione del flusso}
\begin{cases}
    \partial_{\lambda}\Phi^{f_\lambda}_{\lambda}(x,t) = Y(\tilde{\Phi}_{\lambda}(x,t)), & \Phi_0 = \mathrm{Id}_{\mathbb{R}}, \\
    \partial_{\lambda} f_{\lambda}(t) = m(f_{\lambda}(t)), & f_0 = \mathrm{id}_{\mathbb{R}}, \\
    \partial_{\lambda}h_{\lambda}(x,t) = \sqrt{f'_{\lambda}(t)} H(\tilde{\Phi}_{\lambda}(x,t)), & h_0 = 0.
\end{cases}
\end{equation}
Naturally, the first two equations of system \eqref{ricostruzione del flusso} can be cast into vector form as follows:
\[ \partial_\lambda\tilde{\Phi}_\lambda(x,t)=\tilde{Y}(\tilde{\Phi}_\lambda(x,t)), \quad \quad \tilde{\Phi}_0=\mathrm{Id}_{\mathbb{R}\times \mathbb{R}}.\]
\end{teorema}

\begin{proof}
See \cite{paper2020}.
\end{proof}

Finally, given a finite stochastic transformation $T=(\tilde{\Phi}, h)$ and an infinitesimal transformation $V=(\tilde{Y},H)$ with its associated one-parameter group $T_\lambda$, the \emph{push-forward} of $V$ through $T$ is the infinitesimal transformation related to the finite one-parameter group $\ T\circ T_{\lambda} \circ T^{-1},$ and it is explicitly given by (see \cite{paper2015}):
\begin{equation}\label{pushforward}
T_*(V) = \left( (D(\tilde{\Phi})\tilde{Y})\circ\tilde{\Phi}^{-1},\left( \frac{\frac{m'}{2}h + H + \tilde{Y}(h)}{\sqrt{f'}} \right) \circ \tilde{\Phi}^{-1} \right),
\end{equation}
where \( D\) denotes, as usual, the Jacobian matrix, 
\[ (D(\tilde{\Phi})\tilde{Y})\circ\tilde{\Phi}^{-1}= \left((\Phi_x Y + \Phi_t f' m) \circ \tilde{\Phi}^{-1}, (f' m) \circ f^{-1}\right)\]
and
\[ \tilde{Y}(h)=Yh_x+mh_t.\]
\\
\\
Among the infinitesimal transformations, a subclass emerges as particularly interesting for the development of the present paper.
\begin{definizione}\label{def doob inf}
    An infinitesimal stochastic transformation $V=(\tilde{Y}, H)$ is said to be of Doob-type (respectively, quasi-Doob-type) if the corresponding one-parameter group of finite stochastic transformations $T_\lambda=(\tilde{\Phi}_\lambda,  h_\lambda)$, obtained by solving \eqref{ricostruzione del flusso}, is a finite Doob (respectively, quasi-Doob) transformation in the sense of Definition \ref{def doob e qd}.
\end{definizione}
We can provide a characterization of infinitesimal Doob and quasi-Doob transformations, as done in Proposition \ref{caratterizzazione finita d e qd} for the finite setting:
\begin{proposizione}\label{caratterizzazione doob e qd infin}
    An infinitesimal stochastic transformation \( V=(\tilde{Y},H)\) associated with the one-parameter group \( T_\lambda=(\tilde{\Phi}_\lambda, h_\lambda)\) is of Doob-type (with respect to   $SDE_{\mu,\sigma}$) if and only if there exist a function \( K(x,t)\) and a family of functions $K^{\lambda}(x,t)$ (depending smoothly on $\lambda$), satisfying the relation 
    \begin{equation}\label{def Klambda}K^\lambda = \int \frac{\sigma}{E_{T_\lambda}(\sigma)}K_x dx,\end{equation}
    such that
    \begin{itemize}
        \item[i)] \( H= \sigma K_x,\)
        \item[ii)] \( E_{T_\lambda}(L_t)[K^\lambda]=0,\)
    \end{itemize}
where \( E_{T_\lambda}(L_t)\) is the infinitesimal generator of the transformed SDE, namely
    \[ E_{T_\lambda}(L_t)=\partial_t+E_{T_\lambda}(\mu)\partial_x+\frac{1}{2} (E_{T_\lambda}(\sigma))^2\partial_{xx},\]
    with \( E_{T_\lambda}(\mu)\) and \( E_{T_\lambda}(\sigma)\) given according to Definition \ref{def trasf stoc}.\\
    \noindent
    Similarly, \( V\) is of quasi-Doob-type if and only if condition i) holds.
\end{proposizione}
\begin{proof}
Let \( T_\lambda=(\tilde{\Phi}_\lambda, h_\lambda)\) be the one-parameter group of finite transformations associated with \( V\) via Theorem \ref{teo ricostruzione del flusso}. By Proposition \ref{caratterizzazione finita d e qd}, it suffices to prove that there exists a function \( k_\lambda\) such that
\begin{itemize}
    \item[a)] \( h_\lambda=\sigma (k_\lambda)_x,\)
    \item[b)] \( L_t(k_\lambda)=-\frac{1}{2} h^2_\lambda\)
\end{itemize}
if and only if conditions \( i) \) and \( ii) \) hold.

Let us assume that \( a)\) and \( b)\) hold, and let 
\[ K= \partial_\lambda k_\lambda |_{\lambda=0}.\]
If we differentiate condition \( a)\) with respect to the flow parameter \( \lambda\) and evaluate the result at \( \lambda=0\), we obtain
\[ H=\sigma \partial_\lambda (k_\lambda)_x \Big|_{\lambda=0}= \sigma K_x,\]
and thus condition \( i)\) is satisfied. \\
 Next, we must prove condition \( ii)\). To do so, we first show that 
\[ k_\lambda=\int K^\lambda\circ \tilde{\Phi}_\lambda d\lambda.\]
Indeed, using \( a)\), the flow reconstruction \eqref{ricostruzione del flusso}, and the recently established condition \( i)\), we have:
\[k_\lambda=\int \frac{h_\lambda}{\sigma}dx=\int \int \frac{\sqrt{f'_\lambda}}{\sigma} H \circ \tilde{\Phi}_\lambda d\lambda \ dx= \]
\[=\int \int \frac{ \sqrt{f'_{\lambda}}}{\sigma} \left(\sigma\circ \tilde{\Phi}_\lambda \right)\left(K_x\circ \tilde{\Phi}_{\lambda}\right) dx \ d\lambda =\int\int (\Phi_\lambda)_x K^{\lambda}_x\circ \tilde{\Phi}_\lambda dx d\lambda =\int K^\lambda\circ \tilde{\Phi}_\lambda d\lambda.\]
Furthermore, from \( b)\), we obtain
\[ -\frac{1}{2} h_\lambda^2= L_t(k_\lambda)=\int L_t( K^\lambda\circ \tilde{\Phi}_\lambda) d\lambda= \int f'_\lambda  E_{T_\lambda}(L_t)[K^\lambda] \circ \tilde{\Phi}_\lambda d\lambda - \int \sigma^2 (k_\lambda)_x (K\circ \tilde{\Phi}_\lambda)_x d\lambda = \]
\[=\int f'_\lambda  E_{T_\lambda}(L_t)[K^\lambda] \circ \tilde{\Phi}_\lambda d\lambda - \frac{\sigma^2}{2}\int \partial_\lambda(\partial_x k_\lambda)^2 d\lambda=\int f'_\lambda  E_{T_\lambda}(L_t)[K^\lambda] \circ \tilde{\Phi}_\lambda d\lambda - \frac{\sigma^2}{2}(\partial_x k_\lambda)^2=\]
\[= \int f'_\lambda  E_{T_\lambda}(L_t)[K^\lambda] \circ \tilde{\Phi}_\lambda d\lambda - \frac{1}{2} h_\lambda^2 . \]
Thus, 
\[ \int f'_\lambda  E_{T_\lambda}(L_t)[K^\lambda] \circ \tilde{\Phi}_\lambda d\lambda =0,\]
which implies, after differentiating with respect to \(\lambda\) and bearing in mind that \( f'_\lambda >0\) by construction, 
\[  E_{T_\lambda}(L_t)[K^\lambda] =0,\]
meaning that condition \( ii)\) is also satisfied.\\
\\
Conversely, let us assume that conditions \( i)\) and \( ii)\) hold. First, we prove \( a)\). 
By construction, \( k_\lambda\) must satisfy
\[ k_\lambda= \int \frac{h_\lambda}{\sigma}dx,\]
hence condition \( a)\) is satisfied. 
Next, we prove \( b).\) Using the flow reconstruction \eqref{ricostruzione del flusso} and condition \( i)\), we have that
\[ k_\lambda= \int \int \frac{\sqrt{f'_\lambda}}{\sigma} H\circ \tilde{\Phi}_\lambda d\lambda dx = \int \int \frac{\sqrt{f'_\lambda}}{\sigma} \left(\sigma \circ \tilde{\Phi}_\lambda\right)\left( K_x \circ \tilde{\Phi}_\lambda\right) dx d\lambda = \int \int (K^\lambda \circ \tilde{\Phi}_\lambda)_x dx d\lambda =\int K^\lambda\circ \tilde{\Phi}_\lambda d\lambda.\]
 
 By hypothesis, condition \( ii)\) holds; thus, proceeding as we did for the first part of the proof, we obtain
\[ L_t(k_\lambda)=\int L_t(K^\lambda\circ \tilde{\Phi}_\lambda)d\lambda = \int f'_\lambda \underbrace{E_{T_\lambda}(L_t)[K^\lambda]}_{=0}\circ \tilde{\Phi}_\lambda d\lambda - \frac{1}{2}h_\lambda^2=-\frac{1}{2}h_\lambda^2,\]
which concludes the proof.

The proof for the quasi-Doob case is completely analogous to the argument used to establish the equivalence \( a) \Leftrightarrow i).\)
 
\end{proof}

\begin{osservazione}
  The function \( K^\lambda\) in \eqref{def Klambda} is defined up to a function \( c(\lambda,t)\), but the choice of \( c(\lambda,t)\) is uniquely determined by condition ii) of Proposition \ref{caratterizzazione doob e qd infin}, namely \( E_{T_\lambda}K^\lambda=0.\)
\end{osservazione}
\subsection{Remodeling map: from Brownian motion to SDEs}\label{sezione 1.3}

In this section we introduce the stochastic transformation that allows us to recover any SDE with constant-sign diffusion coefficient starting from the simplest one-dimensional SDE corresponding to the Brownian motion, denoted by $ SDE_{(0,1)}$, and given by
\begin{equation}\label{Brownian_motion_SDE}
 dX_t=0 dt+1 dW_t.   
\end{equation} 

In order to identify the suitable transformation, we will take advantage both of the Lamperti transformation of an SDE and of the random change of measure via the Girsanov theorem, where the latter can be characterized by means of the function $h$. This transformation will play a fundamental role in the remainder of the paper.

\medskip

\begin{proposizione}\label{remodeling_map}
A generic SDE of the form \eqref{SDE} with $\sigma >0$  can be obtained starting from \eqref{Brownian_motion_SDE} by applying
the stochastic transformation $ T=(\tilde{\Phi}, h)$ with \(\tilde{\Phi}=(\Phi(x,f(t)),f(t))\) and where
\begin{equation}\label{Reconstruction map}
   \begin{matrix}
       \Phi(x,t)= F^{-1}(x,t),\\
       f(t)=t\\
       h(x,t)=\Big(\frac{\mu(x,t)}{\sigma(x,t)}-\frac 12\frac{\partial\sigma(x,t)}{ \partial x}\Big)\circ \tilde{\Phi}
       -\frac{1}{\sigma\circ \tilde{\Phi}(x,t)}\frac{\partial \Phi(x,t)}{\partial t}
              \end{matrix} \end{equation}
 and $F(x,t) $ is given by \eqref{Lamperti} (see Remark \ref{invertible lamperti} about the invertibility of \( F\)). 
 \end{proposizione}
 \begin{proof}

 Let us apply $ T$ to the one-dimensional $ SDE_{(0,1)}$ given by \eqref{Brownian_motion_SDE}.
Since \( f=\mathrm{id}\), then \(\Phi^f=\Phi\). Thus, according to Definition \ref{def trasf stoc}, the action of $ T$ on the coefficients of the SDE is given by
\[ E_T(\mu_{B.M.})= E_T(0)=\Big(L_t(\Phi) + \Phi_x\sigma_{B.M.} h\Big)\circ \tilde{\Phi}^{-1}=\Big(L_t(\Phi) + \Phi_x 1  h \Big)\circ \tilde{\Phi}^{-1}=\mu\]
and
\[E_T(\sigma_{B.M.})=E_T(1)= (\Phi_x\sigma_{B.M.})\circ\tilde{\Phi}^{-1}=(\Phi_x \cdot 1)\circ\tilde{\Phi}^{-1}=\sigma(\Phi(x,t),t)\circ\tilde{\Phi}^{-1}=\sigma.\]
\end{proof}
   \noindent The assumption $\sigma>0$ in Proposition \ref{remodeling_map} ensures the well-posedness of the Lamperti map. Naturally, the same would hold for $\sigma<0$ (see Remark \ref{remark segno sigma}).
Thus, every $ SDE_{\mu,\sigma}$ with constant-sign diffusion coefficient can be recovered from Brownian motion with a transformation of the previous type. This justifies the following definitions.

\begin{definizione}\label{def remodeling restoring}
    We call $\ T=(\tilde{\Phi}, h)$, where $\tilde{\Phi}=(\Phi^f, f)$ and with $ \Phi,f,h$ as in  (\ref{Reconstruction map}), a remodeling map.
    Conversely, we call the inverse transformation $\ T^{-1}$ a restoring map.
\end{definizione}
\begin{osservazione}
    The inverse transformation $\ T^{-1}$ permits us to reduce an $ SDE_{\mu,\sigma}$ to the driving Brownian motion \eqref{Brownian_motion_SDE}. 
\end{osservazione}
\bigskip

\section{Transformations of SDEs and of related KBEs}\label{Section sde kbe}
In this section, we investigate the connection between the transformations applied to the SDE \eqref{SDE} and those applied to the corresponding KBE \eqref{kbe a}. The material discussed in this section will subsequently be used to examine, as a special case, the relationship between the symmetries of an SDE and those of the corresponding KBE, which will be the focus of Section \ref{section symmetries}.

\subsection{The correspondence between transformations}
Before stating the main result, we first briefly recall the probabilistic representation of second-order linear parabolic PDEs featuring a zero-order potential term. Consider a PDE of the form
\begin{equation} \label{eq:pde_with_potential}
    u_t + \mu(x,t)u_x + a(x,t)u_{xx} + c(x,t)u = 0, \quad (x,t) \in \mathbb{R} \times [0,T),
\end{equation}
subject to the terminal condition $u(x,T) = g(x)$. Here, we assume that the diffusion coefficient can be written as $a(x,t) = \frac{1}{2}\sigma^2(x,t)$ for some sufficiently regular function $\sigma$. 

According to the well-known Feynman-Kac formula, the solution to equation \eqref{eq:pde_with_potential} admits a stochastic representation intrinsically linked to the associated Itô stochastic differential equation:
\begin{equation} \label{eq:associated_sde_fk}
    dX_s = \mu(X_s,s)\,ds + \sigma(X_s,s)\,dW_s, \quad s \in [t,T],
\end{equation}
with the initial condition $X_t = x$, where $W$ denotes a standard \(\mathbb{P}-\)Brownian motion. Specifically, the solution $u(x,t)$ to \eqref{eq:pde_with_potential} can be expressed as follows
\begin{equation} \label{eq:fk_formula}
    u(x,t) = \mathbb{E}_{\mathbb{P}}\left[ \exp\left( \int_t^T c(X_s,s)\,ds \right) g(X_T)|X_t=x \right],
\end{equation}
provided that the coefficients are regular enough and the expectation is well-defined.\\

The most general form-preserving transformation applicable to a KBE is a mapping \(\hat{T}\) that converts a KBE of the form \eqref{kbe a} into
\begin{equation}\label{KBE2b}
    \hat{u}_{\hat{t}} + \hat{\mu}(\hat{x},\hat{t}) \hat{u}_{\hat{x}} + \hat{a}(\hat{x},\hat{t}) \hat{u}_{\hat{x}\hat{x}} + V(\hat{x},\hat{t}) \hat{u} = 0,
\end{equation}
where the function \(V\) is frequently referred to in the literature as a potential term. Such transformations take the form \(\hat{T} = (\tilde{\Phi}, \theta)\), with \(\tilde{\Phi} = (\Phi^f, f)\), and operate as follows (see, e.g., \cite{Olver}):
\begin{equation}\label{T kbe b}
    t \to \hat{t} := f(t), \quad x \to \hat{x} := \Phi^f(x,t), \quad u \to \hat{u} = \theta(x,t) u + \Psi(x,t), \quad \text{s.t.} \quad L_t\left(\frac{\Psi}{\theta}\right) = 0.
\end{equation}

Hereafter, we shall refer to this type of mappings on the PDE as \textit{admissible transformations}, since they are precisely the transformations that preserve the linear parabolic structure of the KBE, up to the possible addition of a zero-order potential term \(V u\).\\
\\
The following theorem shows that quasi-Doob stochastic transformations provide the natural framework for establishing a complete correspondence between the transformations applied to an SDE and the admissible transformations applied to the corresponding KBE:

\begin{teorema}\label{teo quasi doob}
 Consider the $\text{SDE}_{\mu,\sigma}$ \eqref{SDE}, where $W_t$ is a $\mathbb{P}$-Brownian motion, and the related KBE \eqref{kbe a}, with terminal condition
\begin{equation}\label{terminal condition b}
    u(x,T)=g(x).
\end{equation}
There is a  correspondence, modulo $\Psi$, up to multiplication of \(u\) by nonzero constants and up to perturbations of the quasi-Doob function \(k\) of the form
\(k \mapsto k+g(t)\), between the group of admissible transformations of the form $\hat{T}$ applied to \eqref{kbe a} and the group of stochastic transformations of the form $T=(\tilde{\Phi}, h)$, where \(\tilde{\Phi}=(\Phi^f,f)\) and \( h=\sigma k_x\) is a quasi-Doob-type change of measure, applied to \eqref{SDE}, via $k=-\ln|\theta|$.
\end{teorema}
\begin{proof}
    Transformation \( \hat{T}\) given in \eqref{T kbe b} maps \eqref{kbe a} into (see, e.g. \cite{Olver})
\begin{equation}\label{KBE2bb}
\hat{u}_{\hat{t}}+\hat{\mu} (\hat{x},\hat{t}) \hat{u}_{\hat{x}}+\hat{a}(\hat{x},\hat{t}) \hat{u}_{\hat{x}\hat{x}}+V(\hat{x},\hat{t}) \hat{u}=0
\end{equation}
where 
\begin{equation}\label{coefficienti trasformati pde} V(\hat{x},\hat{t})=\left[ \frac{1}{\theta f'(t)}(\sigma^2 \frac{\theta_x^2}{\theta}-L_t(\theta))\right]\circ \tilde{\Phi}^{-1}\end{equation}
\[ \hat{\mu}(\hat{x},\hat{t}) = \left[\frac{1}{f'(t)}\left(L_t(\Phi^f)-(\Phi^f_x)\sigma^2\frac{\theta_x}{\theta}\right) \right]\circ\tilde{\Phi}^{-1}  \]
\[ \hat{a}(\hat{x},\hat{t})=\left[\frac{1}{f'}(\Phi^f_x)^2\left(\frac 12 \sigma^2 \right)\right]\circ \tilde{\Phi}^{-1}.\]
By hypothesis, KBE \eqref{kbe a} is related to SDE \eqref{SDE}. We aim to recover the stochastic transformation \( T=(\tilde{\Phi},h)\) that maps SDE \eqref{SDE} into the SDE related to KBE \eqref{KBE2bb}. A generic stochastic transformation \( T=(\tilde{\Phi},h)\) maps SDE \eqref{SDE} into (see Theorem \ref{trasformate})
\begin{equation}\label{SDE2}
    dY_{t'}=\mu'(Y_{t'},t')dt'+\sigma'(Y_{t'},t') dW^{\mathbb{Q}}_{t'}
\end{equation}
where \( d\mathbb{Q}/d\mathbb{P}\) is given by the Dooléans-Dade exponential of \( h\) and
\[ \mu'(y,t')=\left[\frac{1}{f'}\left(L_t(\Phi^f)+(\Phi^f_x)\sigma h\right) \right]\circ\tilde{\Phi}^{-1}\]
\[ \sigma'(y,t')=\left[\frac{1}{\sqrt{f'}}(\Phi^f_x)\sigma \right]\circ \tilde{\Phi}^{-1}.\]
To recover the stochastic transformation \( T\) corresponding to the deterministic \( \hat{T}\), it suffices to equate the coefficients:
\[\mu'=\hat{\mu}, \quad \quad \frac{1}{2}\sigma'^2=\hat{a},\]
 which yields
\begin{equation}\label{h dim}h=-\sigma \frac{\theta_x}{\theta}.\end{equation}
Since we are working in 1-dimension, every Girsanov-type  measure change is a quasi-Doob measure change: namely, there exists a function \(k \) such that \( h=\sigma k_x\). So, by substituting \( h\) with \(\sigma  k_x\) and integrating, we immediately obtain 
\begin{equation}\label{k dim} k=-\ln|\theta|.\end{equation}
\end{proof}

\begin{osservazione}
    In one dimension, under the assumption of non-degenerate diffusion-coefficient introduced in Section \ref{sezione lamperti}, any generic stochastic transformation in the sense of Definition \ref{def trasf stoc} is of quasi-Doob type. Indeed it is always possible to find a function $k$ such that $h = \sigma k_x$ by setting $k = \int \frac{h}{\sigma} dx$. Consequently, Theorem \ref{teo quasi doob} states that in one dimension, the admissible transformations on the KBE correspond precisely to the stochastic transformations introduced in Definition \ref{def trasf stoc}.
\end{osservazione}

\begin{osservazione}\label{oss one to one}
In order to preserve the injectivity of the correspondence, it was necessary to quotient out by \(\Psi\), by constant rescalings of \(u\), and, in the quasi-Doob case, also by transformations of the form \(k \mapsto k+g(t).\) Indeed, the transformation of the KBE given by
\[
\hat{t}=t, \qquad \hat{x}=x, \qquad \hat{u}=cu, \qquad c>0,
\]
which, in our notation, corresponds to \(\hat{T}=(Id,\theta=c),\) is associated by Theorem \ref{teo quasi doob} with a stochastic transformation of the form \(T=(Id,h),\) where
\[
h=\sigma k_x
=\sigma\bigl(-\ln(c)\bigr)_x
=0.
\]
Therefore, all nonzero constant rescalings of \(u\) are mapped to the identity transformation of the SDE. This is in fact consistent with the probabilistic interpretation of this correspondence given in Theorem \ref{teo interpretazione probabilistica}; see Remark \ref{oss moltiplicazione c}.
Similarly, replacing the quasi-Doob function \(k\) by
\[
k(x,t)\mapsto k(x,t)+g(t)
\]
leaves \(h=\sigma k_x\) unchanged, and hence does not affect the associated stochastic transformation \(T=(\tilde{\Phi},h).\) On the other hand, it changes \(\theta=e^{-k}\), and therefore changes the corresponding transformation \(\hat{T}=(\tilde{\Phi},\theta)\) of the KBE. This issue does not arise when considering the correspondence with Doob-type transformations; see Corollary \ref{doob kbe fin}. Indeed, in that case the function \(k\) must also satisfy
\(L_t(k)=-\frac{1}{2}h^2\) (see Proposition \ref{Prop conddoob}), and hence the choice of the integration term \(g(t)\) cannot be arbitrary. A possible alternative way of formalizing more precisely the correspondence established in Theorem \ref{teo quasi doob} would be to enrich the definition of stochastic transformation by retaining the function \(k\) itself, rather than only its gradient through \( h=\sigma k_x.\)
Finally, it is worth observing that, since the correspondence between transformations is based on \(k=-\ln|\theta|\), the sign of \(\theta\) is lost. This is again consistent with the probabilistic interpretation of the correspondence; see Remark \ref{oss moltiplicazione c}.
\end{osservazione}

At this point, a natural question arises: what are the proper transformations to consider within the stochastic framework to ensure that the corresponding mapping on the associated KBE strictly preserves its form, without the addition of potential terms? In other words, we aim to map the KBE \eqref{kbe a} into another KBE of the form
\begin{equation}\label{KBE1b}
    \hat{u}_{\hat{t}} + \hat{\mu}(\hat{x},\hat{t}) \hat{u}_{\hat{x}} + \hat{a}(\hat{x},\hat{t}) \hat{u}_{\hat{x}\hat{x}} = 0,
\end{equation}
and to determine the corresponding stochastic transformation. We refer to such mappings on the KBE as \textit{strictly admissible transformations}, precisely because they rigidly preserve its structure without introducing any potential terms. The answer to this question is provided by Doob-type transformations, and this result can be proven as a special case of Theorem \ref{teo quasi doob}:
\begin{corollario}\label{doob kbe fin}
    Consider the SDE \eqref{SDE},  where \( W_t\) is a \(\mathbb{P}-\)Brownian motion,  and the related KBE \eqref{kbe a},  with terminal condition
    \begin{equation}\label{terminal condition}
    u(x,T)=g(x).
    \end{equation}
    The group of strictly admissible transformations preserving the linear parabolic structure of KBE \eqref{kbe a} (without adding a potential term \( V(x,t) u(x,t)\)) is given by  \( \hat{T}=(\tilde{\Phi},\theta)\), acting as \eqref{T kbe b} but with the additional constraint 
 \(L_t(\theta)=\sigma^2 \frac{\theta_x^2}{\theta}\).\\
    There is a correspondence, modulo \( \Psi\) and up to multiplication of \( u\) by nonzero constants, between the group of strictly admissible transformations \( \hat{T}\) applied to \eqref{kbe a} and the group of stochastic transformations \( T=(\tilde{\Phi},h)\), where \(h=\sigma k_x\) is a Doob-type change of measure, applied to \eqref{SDE}, via \( k =-\ln|\theta|\). 
\end{corollario}
\begin{proof}
    The proof proceeds as in Theorem \ref{teo quasi doob}. We consider an admissible transformation $\hat{T}=(\tilde{\Phi},\theta)$ and impose that the potential $V$ in the transformed equation vanishes. From \eqref{coefficienti trasformati pde}, this yields the additional constraint
\[ L_t(\theta) = \sigma^2\frac{\theta_x^2}{\theta}. \]
Proceeding exactly as in the proof of Theorem \ref{teo quasi doob}, we impose the equality between the transformed coefficients of the SDE and those of the KBE, which once again yields \eqref{h dim} and \eqref{k dim}.

The condition
\[ L_t(\theta) = \sigma^2 \frac{\theta_x^2}{\theta} \]
imposed on $\theta$ to preserve the structure of \eqref{kbe a} without introducing a potential term can be rewritten in terms of $k$ (via $|\theta|=e^{-k}$) as
\begin{equation}\label{h theta}
    L_t(k) = -\frac{1}{2}h^2.
\end{equation}
Equation \eqref{h theta} is precisely the condition required for the change of measure to be of Doob type (see Proposition \ref{caratterizzazione finita d e qd}). Thus, the strict admissibility requirement ensuring that the structure of the KBE is preserved corresponds to imposing that the change of measure is of Doob type. In particular (see Section \ref{section_3}), we have
\[ \frac{d\mathbb{Q}}{d\mathbb{P}}\bigg|_{\mathcal{F}_t} = e^{k(X_t,t)-k(X_0,0)}. \]
\end{proof}

The correspondence between stochastic transformations and transformations of the corresponding KBE, as illustrated in Theorem \ref{teo quasi doob} and Corollary \ref{doob kbe fin}, admits an interesting probabilistic interpretation, which we present in the following theorem. Since the function \( \Psi \) plays no role in the correspondence, in what follows, unless explicitly stated otherwise, we shall assume \( \Psi=0 \).

\begin{teorema}[Probabilistic interpretation of the correspondence]\label{teo interpretazione probabilistica}
    Consider the SDE \eqref{SDE}, where $W_t$ is a $\mathbb{P}$-Brownian motion, and the related KBE \eqref{kbe a}, with terminal condition
\begin{equation}\label{terminal condition c}
    u(x,T)=g(x).
\end{equation}

Let us consider the admissible transformation $\hat{T}=(\tilde{\Phi}, \theta)$ on the KBE given in \eqref{T kbe b} and the corresponding quasi-Doob stochastic transformation $T=(\tilde{\Phi},h)$ on the SDE, where $h=\sigma k_x$ and $k=-\ln |\theta|$. 

The transformation  \(u \mapsto \hat{u} \) on the KBE corresponds, in the SDE setting, to a quasi-Doob change of probability measure \( \mathbb{P} \mapsto \mathbb{Q}, \)
with
\[ \left.\frac{d\mathbb{Q}}{d\mathbb{P}}\right|_{\mathcal{F}_t} := \exp\left\{ k(X_t,t)-k(X_0,0)-\int_0^t \left(L_s( k (X_s,s))+\frac{1}{2}h^2(X_s,s)\right)ds\right\},\]
and vice versa.

In particular, the relation 
\[ u=\frac{\hat{u}}{\theta} \]
in the deterministic KBE context corresponds, in the SDE setting, to the well-known Bayes' rule
\[ \mathbb{E}_{\mathbb{P}}\left[g(X_T) \mid X_t=x\right] = \mathbb{E}_{\mathbb{Q}}\left[\frac{Z_T}{Z_t}g(X_T) \mid X_t=x\right], \]
where 
\[ Z_t = \left.\frac{d\mathbb{P}}{d\mathbb{Q}}\right|_{\mathcal{F}_t}. \]

In addition, if $\hat{T}$ is strictly admissible (i.e., $L_t(\theta)=\sigma^2\frac{\theta^2_x}{\theta}$) and, equivalently, $T$ is of Doob type (i.e., $L_t(k)=-\frac{1}{2}h^2$), then the change of probability measure $\mathbb{P}\mapsto \mathbb{Q}$ corresponding to $u\mapsto \hat{u}$ is of Doob type, given by
\[ \left.\frac{d\mathbb{Q}}{d\mathbb{P}}\right|_{\mathcal{F}_t} := \exp\{ k(X_t,t)-k(X_0,0) \}. \]
\end{teorema}
\begin{proof}
We already proved in Theorem \ref{teo quasi doob} and Corollary \ref{doob kbe fin} the correspondence between (strictly) admissible transformations on the KBE and (quasi-)Doob transformations on the related SDE via \( k=-\ln|\theta|\).  The Feynman-Kac representation \eqref{eq:fk_formula} for the solution of \eqref{kbe a} with terminal condition \( u(x,T)=g(x)\) is
 \[ u(x,t)=\mathbb{E}_{\mathbb{P}}[g(X_T)|X_t=x].\]
From \( \hat{u}(\hat{x},\hat{t})=\theta(x,t) u(x,t)\), the terminal condition for \eqref{KBE2bb} becomes
\[ \hat{u}(\hat{x},\hat{T})=\theta(x,T)u(x,T)=\theta(x,T)g(x).\]
Therefore, the Feynman-Kac representation for the solution \(\hat{u}\) to the transformed KBE is
\[ \hat{u}(\hat{x},\hat{t})=\mathbb{E}_{\mathbb{Q}}[e^{\int_{\hat{t}}^{\hat{T}}V(X_s,s)ds}\theta(X_T,T)g(X_T)|\hat{X}_{\hat{t}}=\hat{x}].\]
Using \eqref{coefficienti trasformati pde}, \eqref{T kbe b}, the properties of conditional expected value, the invertibility of \( \tilde{\Phi}\) and  a time change in the integral we obtain
\[\hat{u}(\hat{x},\hat{t})=\mathbb{E}_{\mathbb{Q}}\left[e^{\int_{\hat{t}}^{\hat{T}}V(X_s,s)ds}\theta(X_T,T)g(X_T)|\hat{X}_{\hat{t}}=\hat{x}\right]=\]\[=\mathbb{E}_{\mathbb{Q}}\left[e^{-\int_t^T \frac{1}{\theta}\Big(L_t(\theta)-\sigma^2 \frac{\theta_x^2}{\theta}\Big)ds}\theta(X_T,T)g(X_T)|\Phi^f(X_t,t))=\Phi^f(x,t)\right]=\]\[=\mathbb{E}_{\mathbb{Q}}\left[e^{-\int_t^T \frac{1}{\theta}\Big(L_t(\theta)-\sigma^2 \frac{\theta_x^2}{\theta}\Big)ds}\theta(X_T,T)g(X_T)|X_t=x\right].\]
Defining
\[ Z_t:=\frac{d\mathbb{P}}{d\mathbb{Q}}|_{\mathcal{F}_t}=e^{-\Big( k(X_t,t)-k(X_0,0)-\int_0^t L_t(k)+\frac{1}{2} h^2 (X_s,s)ds\Big)}\]
then, substituting the Feynman-Kac representations for \( u\) and \(\hat{u}\) in 
\[ u(x,t)=\frac{\hat{u}(\hat{x},\hat{t})}{\theta(x,t)},\]
 using the \(\mathcal{F}_t-\)measurability of \( \theta(X_t,t)\)  and the fact that
\[e^{-\int_t^T \frac{1}{\theta}\Big(L_t(\theta)-\sigma^2 \frac{\theta_x^2}{\theta}\Big)ds}=e^{-\int_t^T \big(L_t(  k)+\frac{1}{2} h^2\big) ds}=\frac{e^{-\int_0^T \big(L_t( k )+\frac{1}{2} h^2\big) ds}}{e^{-\int_0^t \big(L_t( k )+\frac{1}{2} h^2\big) ds}},\] 
we obtain
\[ \mathbb{E}_{\mathbb{P}}[g(X_T)|X_t=x]=\mathbb{E}_{\mathbb{Q}}\left[\frac{Z_T}{Z_t}g(X_T)|X_t=x\right],\]
which is exactly Bayes formula for changing measures into conditional expectations.\\
\\
The claim for the Doob case (or, equivalently, for strictly admissible transformations) follows immediately, since Corollary \ref{doob kbe fin} ensures that in this case \[V=0=L_t(\theta)-\sigma^2\frac{\theta^2_x}{\theta}=L_t(k)+\frac{1}{2}h^2.\]
\end{proof}

\begin{osservazione}
    Notice that, using the notation from Section \ref{section_3}, the function \(\theta\) of Theorem \ref{teo interpretazione probabilistica} corresponds to the one characterizing a (quasi)-Doob transformation via \(\theta=\frac{1}{\phi}\). In particular, \(\phi\) is associated with the change of measure \( \mathbb{P}\mapsto \mathbb{Q}\), whereas \(\theta\) is associated with the inverse change of measure \(\mathbb{Q}\mapsto \mathbb{P}\).
\end{osservazione}

\begin{osservazione}\label{oss moltiplicazione c}
    The correspondence between transformations is understood up to constant rescalings of \(u\), which, as shown in Remark \ref{oss one to one}, are all mapped to the identity transformation of the SDE. This is consistent with the probabilistic interpretation provided by Theorem \ref{teo interpretazione probabilistica}. Indeed, consider the transformation
\[
\hat{t}=t, \qquad \hat{x}=x, \qquad \hat{u}=cu,
\]
of the KBE, with \(c\in\mathbb{R}\), and assume for simplicity that \(c>0\). The probabilistic interpretation relates the multiplicative function \(\theta\), which in this case is the constant \(\theta=c\), to the Radon--Nikodym density associated with the change of measure. If such a density is constant, the normalization of the new probability measure requires that this constant be equal to \(1\). Therefore, from the stochastic point of view, all constant rescalings of \(u\) are identified with a single transformation, namely the identity transformation of the SDE.

The same observation also explains why the correspondence identifies transformations that differ only by the sign of \(\theta\). Indeed, the correspondence is defined through
\[
k=-\ln|\theta|,
\]
and therefore does not retain information about the sign of \(\theta\). This is again consistent with its probabilistic interpretation, since the Radon--Nikodym density defined by the stochastic exponential is always non-negative.

Finally, in the case of quasi-Doob transformations, the replacement
\[
k(x,t)\mapsto \bar{k}(x,t)\coloneqq k(x,t)+g(t)
\]
is likewise invisible at the level of the stochastic transformation. Indeed, since
\(h=\sigma k_x=\sigma \bar{k}_x,\) the corresponding change of measure is unchanged. More explicitly, the Radon--Nikodym density satisfies
\[
Z_t
=
\left.\frac{d\mathbb{P}}{d\mathbb{Q}}\right|_{\mathcal{F}_t}
=
\exp\left\{
-\left(
k(X_t,t)-k(X_0,0)
-\int_0^t
\left(
L_s(k)+\frac12 h^2
\right)(X_s,s)\,ds
\right)
\right\},
\]
and we have that
\[
\begin{aligned}
&\bar{k}(X_t,t)-\bar{k}(X_0,0)
-\int_0^t
\left(
L_s(\bar{k})+\frac12 h^2
\right)(X_s,s)\,ds
\\
&\qquad =
k(X_t,t)-k(X_0,0)
+g(t)-g(0)
-\int_0^t
\left(
L_s(k)+g'(s)+\frac12 h^2
\right)(X_s,s)\,ds
\\
&\qquad =
k(X_t,t)-k(X_0,0)
-\int_0^t
\left(
L_s(k)+\frac12 h^2
\right)(X_s,s)\,ds.
\end{aligned}\]
Hence the replacement \(k\mapsto k+g(t)\) leaves the Radon--Nikodym density, and therefore the associated stochastic transformation, unchanged.
\end{osservazione}

In Theorem \ref{teo quasi doob} and Corollary \ref{doob kbe fin}, we proved the correspondence between transformations on the KBE and stochastic transformations using a finite approach, namely, by starting from the finite stochastic transformation. This approach is intrinsically more probabilistic, as it offers the advantage of immediately exhibiting the action of the transformation on the equation. Equivalently, this correspondence could have been established using an approach more characteristic of mathematical physics, by starting from infinitesimal transformations. For the sake of completeness, in the following theorem we state the infinitesimal counterpart to Theorem \ref{teo quasi doob} and Corollary \ref{doob kbe fin}. The two approaches are equivalent: the finite version can be recovered from the infinitesimal one via flow reconstruction, whereas the infinitesimal version is derived from the finite one by differentiating with respect to the parameter and evaluating at zero.

\begin{teorema}\label{doob kbe inf}
    Consider the SDE \eqref{SDE}, where $W_t$ is a $\mathbb{P}$-Brownian motion, and the associated KBE \eqref{kbe a}, with terminal condition \eqref{terminal condition}.

Let $\Xi = m\partial_t + Y\partial_x - K u \partial_u$ be an infinitesimal transformation acting on the KBE \eqref{kbe a}, with the corresponding finite transformation given by
\begin{equation}\label{T kbe 1}
    t \mapsto t_\lambda := f_\lambda(t), \quad \quad x \mapsto x_\lambda := \Phi^{f_\lambda}_\lambda(x,t), \quad \quad u \mapsto u_\lambda := \theta_\lambda(x,t) u,
\end{equation}
where the flow reconstruction equations are
\begin{equation}\label{flow rec kbe}
    \partial_\lambda f_\lambda = m(f_\lambda), \quad \partial_\lambda \Phi^{f_\lambda}_\lambda = Y \circ \tilde{\Phi}_\lambda, \quad \partial_\lambda \theta_\lambda = -(K^\lambda \circ \tilde{\Phi}_\lambda) \exp\left(-\int K^\lambda \circ \tilde{\Phi}_\lambda \, d\lambda\right),
\end{equation}
with $\tilde{\Phi}_\lambda = (\Phi_\lambda^{f_\lambda}, f_\lambda)$ and 
\begin{equation}\label{Klambda dim inf}
    K^\lambda = \int \frac{\sigma}{E_{T_\lambda}(\sigma)} K_x \, dx,
\end{equation}
where $E_{T_\lambda}(\sigma)$ is defined according to Definition \ref{def trasf stoc}.

There is a correspondence (up to multiplication of \( K\) by nonzero constants) between the infinitesimal transformations $\Xi$ of the KBE \eqref{kbe a} and the infinitesimal stochastic transformations $V = (\tilde{Y}, H)$ of the SDE \eqref{SDE},  via the relations $H = \sigma K_x$ and $\tilde{Y} = (Y, m)$. Under this identification, admissible infinitesimal transformations---which preserve the linear parabolic structure of the KBE up to the possible addition of a zero-order potential term---correspond to quasi-Doob infinitesimal stochastic transformations. Furthermore, this correspondence restricts to one between strictly admissible infinitesimal transformations (which introduce no potential term) and strict Doob infinitesimal stochastic transformations if and only if $E_{T_\lambda}(L_t)[K^\lambda] = 0$, where 
$$
    E_{T_\lambda}(L_t) = \partial_t + \left[ \left( \frac{1}{f'_\lambda} \left[ L_t(\Phi^{f_\lambda}_\lambda) - (\Phi^{f_\lambda}_\lambda)_x \sigma^2 \frac{(\theta_\lambda)_x}{\theta_\lambda} \right] \right) \circ \tilde{\Phi}_\lambda^{-1} \right] \partial_x + \frac{1}{2} \left( \frac{1}{f'_\lambda} (\Phi^{f_\lambda}_\lambda)_x^2 \sigma^2 \right) \circ \tilde{\Phi}_\lambda^{-1} \partial_{xx}
$$
and \( K^\lambda\) is given by \eqref{Klambda dim inf}.\\
The probabilistic interpretation given by Theorem \ref{teo interpretazione probabilistica} still holds.
\end{teorema}
\begin{proof}
The correspondence follows straightforward by reconstructing the flow of \( \Xi\) via \eqref{flow rec kbe} and the flow of \( V\) via \eqref{ricostruzione del flusso}. Indeed, it suffices to compute the transformed KBE under the action of the finite transformation related to \( \Xi\) and the transformed SDE under the action of the finite transformation related to \( V\), matching the coefficients as in the proof of Corollary \ref{doob kbe fin}. \\
The admissibility condition \( E_{T_\lambda}(L_t)[K^\lambda]=0 \), which ensures that the transformation introduces no zero-order terms to the KBE \eqref{kbe a}, is obtained from the admissibility condition for the corresponding finite transformation \( L_t(\theta_\lambda)=\sigma^2 \frac{(\theta_\lambda)_x^2}{\theta_\lambda} \) as in the proof of Proposition \ref{caratterizzazione doob e qd infin}, bearing in mind that, by matching the coefficients of the finite transformations related to \( V\) and \( \Xi\), we have that
\[
h_\lambda=-\sigma\frac{(\theta_\lambda)_x}{\theta_\lambda}.
\]
Concerning the probabilistic interpretation, we first recover, as in the proof of Proposition \ref{caratterizzazione doob e qd infin},  that 
\begin{equation}\label{kk}k_\lambda=\int K^\lambda\circ \tilde{\Phi}_\lambda d\lambda.\end{equation}
    From \eqref{flow rec kbe}, we have that
    \[ \theta_\lambda=\int  -\left( K^\lambda \circ \tilde{\Phi}_\lambda \right) \exp\left(-\int K^\lambda \circ \tilde{\Phi}_\lambda \right) d\lambda = \exp\left(-\int K^\lambda \circ \tilde{\Phi}_\lambda \right) , \]
 so, the transformed solution process $ u_{\lambda}(x_{\lambda},t_{\lambda})$ is obtained by reconstructing the flow 
    \[\        \frac{du_{\lambda}}{d\lambda}= \frac{d\theta_\lambda}{d\lambda} u =    -\left( K^\lambda \circ \tilde{\Phi}_\lambda \right) \exp\left(-\int K^\lambda \circ \tilde{\Phi}_\lambda \right) u=  -\left(K^\lambda \circ \tilde{\Phi}_\lambda \right) \theta_\lambda u= 
    \left(- K^\lambda\circ\tilde{\Phi}_\lambda \right) u_\lambda,
    \]
    which yields, using \eqref{kk},
    \[u_{\lambda} =  u e^{-\int K^\lambda\circ\tilde{\Phi}_\lambda d\lambda}=u e^{-k_\lambda(x,t)}.\] 
The assert follows by proceeding as in the proof of Corollary \ref{doob kbe fin}.
 
\end{proof}

\subsection{Example:  Brownian motion and heat equation}

Let us consider a \(\mathbb{P}-\)Brownian motion
\begin{equation}\label{example bm}
dX_t=dW_t
\end{equation}
and the related heat equation
\begin{equation}\label{example heat}
    u_t+\frac{1}{2}u_{xx}=0.
\end{equation}

\paragraph{A quasi-Doob  admissible transformation}
We first apply to \eqref{example bm} a quasi-Doob one-parameter group of stochastic transformations \( T_\lambda=(\tilde{\Phi}_\lambda, h_\lambda)\), with \( \tilde{\Phi}_\lambda=(\Phi^{f_\lambda}_\lambda, f_\lambda)\) and
\[f_\lambda(t)=t, \quad \quad\Phi^{f_\lambda}_\lambda(x,t)=\Phi_\lambda(x,f_\lambda(t))=\Phi_\lambda(x,t)=e^\lambda x, \quad \quad h_\lambda=(e^\lambda-1)x.\]
Since we are working in one dimension with a non-vanishing diffusion coefficient, any stochastic transformation is a quasi-Doob transformation, and we have to find a function \( k_\lambda\) such that \( h_\lambda=\sigma (k_\lambda)_x.\) Integrating, we obtain
\[ k_\lambda= \int \frac{h_\lambda}{\sigma}dx=\int \frac{(e^\lambda-1)x}{1}dx= (e^\lambda-1)\frac{x^2}{2}+g(\lambda,t).\]
Since there does not exist a function \( g(\lambda,t)\) such that \( L_t((e^\lambda-1)\frac{x^2}{2}+g(\lambda,t))=-\frac{1}{2}(e^\lambda-1)^2x^2=-\frac{1}{2}h_\lambda^2\), from Proposition \ref{caratterizzazione finita d e qd} \( T_\lambda\) cannot be a Doob-type transformation, and thus we can assume that the integration constant \( g(\lambda,t)\) is equal to zero. Hence,
\[ k_\lambda=(e^\lambda-1)\frac{x^2}{2}.\]
From Theorem \ref{trasformate}, if \( X\) solves \eqref{example bm}, then \( Y\) solves
\begin{equation}\label{bm trasf qd}
    dY_t=\bar{\mu}(Y_t)dt+\bar{\sigma}(Y_t)d\bar{W}^{\mathbb{Q}}_t,
\end{equation}
where, according to  Definition \ref{def trasf stoc},
\[ Y_t=P_{T_\lambda}(X_t)=e^\lambda X_t, \quad \quad\quad\bar{W}^{\mathbb{Q}}_t=P_{T_\lambda}(W_t)=W_t-\int_0^t (e^\lambda-1)X_sds, \]\[ \bar{\mu}(x,t)=E_{T_\lambda}(\mu)(x,t)=(e^\lambda-1)x, \quad \quad \quad \bar{\sigma}(x,t)=E_{T_\lambda}(\sigma)(x,t)=e^\lambda.\]
This is straightforward to verify, since if \( dX_t=dW_t\), then
\[ d(e^\lambda X_t)=\bcancel{(e^\lambda-1)e^\lambda X_t}dt + e^\lambda \left(dW_t - \bcancel{(e^\lambda-1)X_tdt}\right).\]
Theorem \ref{teo quasi doob}, jointly with formulas \eqref{T kbe b}, ensures that while \eqref{example bm} is mapped into \eqref{bm trasf qd}, \eqref{example heat} is transformed into
\begin{equation}\label{heat trasf qd}
    v_{\hat{t}}+ (e^\lambda-1)\hat{x} v_{\hat{x}}+\frac{1}{2} e^{2\lambda} v_{\hat{x}\hat{x}} + V(\hat{x},\hat{t}) v =0, 
\end{equation}
where 
\[ \hat{x}=e^\lambda x, \quad \hat{t}=t, \quad v(\hat{x},\hat{t})=u_\lambda(\hat{x},\hat{t})=\theta_\lambda(x,t) u(x,t), \quad \theta_\lambda=e^{-k_\lambda}=\exp\left(-(e^\lambda-1)\frac{x^2}{2}\right)\]
and \( V\) is given by \eqref{coefficienti trasformati pde}
\[ V(\hat{x},\hat{t})=\frac{1}{2}(e^\lambda-1)\left[(e^\lambda-1)e^{-2\lambda}\hat{x}^2+1\right].\]
Thus,
\begin{equation}\label{v ex} v(\hat{x},\hat{t})=\theta_\lambda(x,t)u(x,t)=\left(\theta_\lambda \circ \Phi_\lambda^{-1}(\hat{x},\hat{t})\right) \left( u \circ \Phi_\lambda^{-1}(\hat{x},\hat{t})\right)=\exp\left(-(e^\lambda-1)e^{-2\lambda}\frac{\hat{x}^2}{2}\right) u (e^{-\lambda}\hat{x},\hat{t}) \end{equation}
and a straight calculation shows that if \eqref{example heat} holds, then also \eqref{heat trasf qd} holds.\\
\\
The same equation can equivalently be derived using the infinitesimal approach of Theorem \ref{doob kbe inf}. Namely, we apply to SDE \eqref{example bm} the infinitesimal transformation \(V_1=(x\partial_x, x).\) By reconstructing the flow (Theorem \ref{teo ricostruzione del flusso}) with \( Y=x, m=0, H=x\), it is immediately verified that \( T_\lambda\) is the one-parameter group of finite transformations related to \(V_1.\) Since we are working in one dimension with a non-vanishing diffusion coefficient, any infinitesimal stochastic transformation is an infinitesimal quasi-Doob transformation, and we have to find a function \( K\) such that \( H=\sigma K_x\). Integrating, we obtain\footnote{Once again, since \(V_1\) cannot be of Doob-type, w.l.o.g. we assume the integration constant equal to zero.}
\[ K=\int \frac{H}{\sigma}dx=\int \frac{x}{1}dx=\frac{x^2}{2}.\]
Theorem \ref{doob kbe inf} ensures that the infinitesimal transformation applied to KBE \eqref{example heat} corresponding to \(V_1\) is given by \( \Xi_1=x\partial_x - \frac{x^2}{2}u\partial_u\). Formulas \eqref{T kbe 1}, \eqref{flow rec kbe} and \eqref{Klambda dim inf} guarantee that \eqref{heat trasf qd} still holds, since
\[ K^\lambda=\int \frac{\sigma}{E_{T_\lambda}(\sigma)} K_xdx=\int \frac{1}{e^\lambda}xdx=e^{-\lambda}\frac{x^2}{2}\]
and 
\[ \theta_\lambda (x,t)= \exp\left(-\int K^\lambda \circ \tilde{\Phi}_\lambda d\lambda \right)= \exp\left(-(e^\lambda - 1 ) \frac{x^2}{2}\right).\]

\paragraph{A Doob strictly admissible transformation}
We now apply to SDE \eqref{example bm} a Doob transformation, namely the one-parameter group of finite transformations \( T_\lambda=(\tilde{\Phi}_\lambda, h_\lambda)\), with \( \tilde{\Phi}_\lambda=(\Phi^{f_\lambda}_\lambda, f_\lambda)\) and
\[ f_\lambda(t)=t, \quad \Phi_\lambda^{f_\lambda}(x,t)=\Phi_\lambda(x,f_\lambda(t))=\Phi_\lambda(x,t)=x+\lambda t^2, \quad h_\lambda=-\lambda,\]
or, equivalently, the infinitesimal transformation \( V_2=(\tilde{Y}, H)\) with \( \tilde{Y}=(Y,m)\) and
\[ Y=t^2, \quad m=0, \quad H=-1.\]
Indeed, we can obtain \( V_2\) from \( T_\lambda\) by deriving its component with respect to \(\lambda\) and evaluating the result in \(\lambda=0\), and we can obtain \( T_\lambda\) from \( V_2\) by solving system \eqref{ricostruzione del flusso}. From Theorem \ref{trasformate} and Definition \ref{def trasf stoc}, this transformation acts in the following manner:
\begin{align}\label{Tex2} \notag &X_t \mapsto P_{T_\lambda}(X_t)=X_t+\lambda t^2,  &W_t \mapsto \bar{W}^{\mathbb{Q}}_t= W_t+\lambda t,\\  &\mu=0 \mapsto \bar{\mu}=E_{T_\lambda}(\mu)=2\lambda t -\lambda, &\sigma=1 \mapsto \bar{\sigma}=1.\end{align}
In fact, it is clear that if \( dX_t=dW_t\), then
\[ d(X_t+\lambda t^2)=(2\lambda t - \bcancel{\lambda}) dt + 1 (dW_t+\bcancel{\lambda}dt).\]
Since we are working in dimension one with a non-vanishing diffusion coefficient, any Girsanov transformation is a quasi-Doob transformation. We first deal with the finite setting, and we exhibit a function \( k_\lambda\) such that \( h_\lambda=\sigma (k_\lambda)_x\). By integrating, we obtain
\[ k_\lambda=\int \frac{h_\lambda}{\sigma}dx=\int -\lambda dx = -\lambda x + g(\lambda,t).\]
We find the constant of integration $g$ by requiring Doob condition (Proposition \ref{caratterizzazione finita d e qd}) to hold
\[ \partial_t g(\lambda,t)=L_t(k_\lambda)=-\frac{1}{2}h_\lambda^2=-\frac{\lambda^2}{2}\]
from which we obtain
\begin{equation}\label{klambda ex2}
    k_\lambda=-\lambda x - \frac{\lambda^2}{2}t.
\end{equation}
In the same way, from the infinitesimal point of view, since we are working in one dimension with a unit diffusion coefficient, \( V_2\) is a quasi-Doob infinitesimal transformation, and there exists a function \( K\) such that \( H=\sigma K_x\). By integrating, we obtain
\[ K=\int \frac{H}{\sigma}dx=\int -1 dx= -x.\]
We can now construct \( K^\lambda\) according to formula \eqref{Klambda dim inf}
\[ K^\lambda = \int \frac{\sigma}{E_{T_\lambda}(\sigma)} K_x dx = \int -1 dx = -x + c(\lambda,t).\]
We find the constant of integration \( c\) by requiring the Doob infinitesimal condition (Proposition \ref{caratterizzazione doob e qd infin}) to hold
\[ 0= E_{T_\lambda}(L_t)[K^\lambda]=(\partial_t+E_{T_\lambda}(\mu)\partial_x + \frac{1}{2} E_{T_\lambda}(\sigma)^2\partial_{xx})[-x + c(\lambda,t)]= \partial_t c -(2\lambda t-\lambda),\]
from which we obtain \( c(\lambda,t)=\lambda t^2-\lambda t\) and
\begin{equation}\label{Kgrandelambda ex2}
    K^\lambda=-x+\lambda t^2 - \lambda t.
\end{equation}
The transformation is well-posed, as it is easy to verify that both \( K=\partial_\lambda (k_\lambda)_{\lambda=0}\) and \( k_\lambda = \int K^\lambda \circ \tilde{\Phi}_\lambda d\lambda\) hold.\\
\\
Corollary \ref{doob kbe fin} (using \eqref{klambda ex2}) or, equivalently, Theorem \ref{doob kbe inf} (using \eqref{Kgrandelambda ex2}) ensure that 
the transformation corresponding to \eqref{Tex2} maps the heat equation \eqref{example heat} into
\begin{equation}\label{heat trasf2}
    v_{\hat{t}}+(2\lambda t - \lambda) v_{\hat{x}}+\frac{1}{2} v_{\hat{x}\hat{x}}=0,
\end{equation}
where
\begin{equation}\label{vtrasf} v(\hat{x},\hat{t})=\theta_\lambda(x,t) u(x,t)= e^{-k_\lambda(x,t)}u(x,t)= e^{-k_\lambda\circ \Phi_\lambda^{-1}(\hat{x},\hat{t})}u\circ\Phi_\lambda^{-1}(\hat{x},\hat{t})=e^{\lambda \hat{x} - \lambda^2 \hat{t}^2 + \frac{\lambda^2}{2}t}u(\hat{x}-\lambda \hat{t}^2,\hat{t}).\end{equation}
A direct computation confirms the theorem: if \eqref{example heat} holds, then \eqref{vtrasf} solves \eqref{heat trasf2}.

\section{Symmetry theory}\label{section symmetries}
In this section, we investigate a specific subclass of transformations, namely symmetries. We introduce their finite and infinitesimal definitions and characterizations, primarily focusing on the connection between the symmetries of an SDE and those of its associated KBE. Subsequently, we outline a nearly algorithmic procedure to derive the symmetries of a general SDE (along with its related KBE) starting from Brownian motion. Finally, we study in detail the Doob-type symmetries of Brownian motion and the related symmetries of the corresponding heat equation.

\subsection{Relation between symmetries of SDEs and KBEs}\label{section relazione sym}
Symmetries represent a special class of transformations that leave the original equation invariant. These transformations are of fundamental importance, as they allow for the study of invariance properties related to the solutions and can sometimes enable a dimensional reduction of the equations themselves. For a comprehensive treatment, we refer the reader to \cite{Olver} for the deterministic setting, and to \cite{paper2015,gaetarodriguez} for the stochastic framework.

\begin{definizione}[Finite and infinitesimal symmetry] \label{def sym}
A stochastic transformation $T$ is a \emph{(finite weak) symmetry} of a given $SDE_{\mu,\sigma}$ if, for every solution process $(X,W)$, $P_T(X,W)$ is a solution process to the same SDE.
An infinitesimal stochastic transformation $V$ generating a one-parameter group $T_\lambda$ is called an \emph{infinitesimal symmetry} of a given $SDE_{\mu,\sigma}$ if $T_\lambda$ is a symmetry of $SDE_{\mu,\sigma}$.
\end{definizione}

Symmetries can be characterized through necessary and sufficient conditions commonly referred to as determining equations (finite and infinitesimal, respectively). In the stochastic framework, one of the main reasons for shifting from the structure of finite transformations to the infinitesimal structure of the corresponding Lie algebra (see Section \ref{section stochastic transformations}) is that infinitesimal determining equations are much easier to handle than their finite counterparts (see, for instance, \cite{proceeding}).

\begin{proposizione}\label{prop:pra}
A stochastic transformation $T=(\tilde{\Phi},h)$, with \(\tilde{\Phi}=(\Phi^f,f)\), is a symmetry of $SDE_{\mu,\sigma}$ if and only if
\begin{align}\label{eqdet finite}
\left( \frac{1}{f'} [L_t(\Phi^f)+(\Phi^f)_x\sigma{h}] \right) \circ\tilde\Phi^{-1} &=\mu \notag \\
\left( \frac{1}{\sqrt{f'}} (\Phi^f)_x \sigma \right) \circ\tilde\Phi^{-1} &=\sigma 
\end{align}
An infinitesimal stochastic transformation $V=(\tilde{Y},H)$, where $\tilde Y=Y\partial_x+m\partial_t$ is an infinitesimal symmetry of $SDE_{\mu,\sigma}$ if and only if $V$ generates a one-parameter group and the following equations hold:
\begin{align}
\label{equa:eqa}
Y(\mu)-L_t(Y)-\sigma{H}+m'\mu&=0 \\
\label{equa:eqb}
[Y,\sigma]+\frac{1}{2}m'\sigma&=0,
\end{align}
where \( [Y,\sigma]= Y \partial_x \sigma - \sigma \partial_x Y.\)
\end{proposizione}
\begin{proof}
    See \cite{paper2020}.
\end{proof}
\begin{osservazione}\label{remark eqdetfin}
    Finite determining equations \eqref{eqdet finite} simply state that \( T\) is a symmetry if and only if \( E_T(\mu)=\mu\) and \( E_T(\sigma)=\sigma.\)
\end{osservazione}
As with general transformations, we can also introduce the notion of Doob and quasi-Doob symmetries, building upon the results established in Section \ref{section_3} for the finite case and in Section \ref{section stochastic transformations} for the infinitesimal one.

\begin{definizione}
    A finite symmetry $T$ is said to be of Doob (respectively, quasi-Doob) type if it also satisfies the requirements of Definition \ref{def doob e qd}. 
    
    Similarly, an infinitesimal symmetry $V$ is said to be of Doob (respectively, quasi-Doob) type if it satisfies the requirements of Definition \ref{def doob inf}.
\end{definizione}

While the conditions characterizing a finite Doob-type symmetry remain those described in Proposition \ref{caratterizzazione finita d e qd}, in the infinitesimal case, Proposition \ref{caratterizzazione doob e qd infin} simplifies significantly when restricted to symmetries.

\begin{proposizione}\label{caratterizzazione sym doob infin}
    An infinitesimal symmetry $V=(\tilde{Y},H)$ is of Doob-type if and only if there exists a function $K$ such that
    \begin{itemize}
        \item[i)] $H=\sigma K_x$,
        \item[ii)] $L_t(K)=0$.
    \end{itemize}
    Similarly, an infinitesimal symmetry $V=(\tilde{Y},H)$ is of quasi-Doob-type if and only if there exists a function $K$ such that condition $i)$ holds.
\end{proposizione}
\begin{proof}
    The proof is analogous to that of Proposition \ref{caratterizzazione doob e qd infin}, recalling that, since $V$ is a symmetry, $E_{T_\lambda}(\sigma)=\sigma$ and $E_{T_\lambda}(L_t)=L_t$. 
\end{proof}

The same geometric and algebraic approach can be applied to the study of KBEs associated with  diffusion processes. A Lie point symmetry for equation \eqref{kbe a} is defined as a one-parameter local group of invertible transformations of the space of independent and dependent variables 
that maps solutions into solutions.

Infinitesimally, this group is generated by a vector field $\Xi$ of the form:
\begin{equation} \label{eq:lie_generator}
    \Xi = m(x,t) \frac{\partial}{\partial t} +  Y(x, t) \frac{\partial}{\partial x} + \psi(x, t, u) \frac{\partial}{\partial u}
\end{equation}
where $m$, $Y$, and $\psi$ are smooth functions of their arguments. The necessary and sufficient condition for $\Xi$ to be the generator of a Lie point symmetry is that its second prolongation $\text{pr}^{(2)}\Xi$ applied to  the equation vanishes on its solution manifold:
\begin{equation} \label{eq:symmetry_condition}
    \left. \text{pr}^{(2)}\Xi (L_t[u]) \right|_{L_t[u]=0} = 0.
\end{equation}

By writing out condition \eqref{eq:symmetry_condition}, one can provide explicit determining equations for the vector field $\Xi$ to be a symmetry of the KBE \eqref{kbe a}:
\begin{align}\label{eq:kbe_determining}
    \psi(x,t,u) &= -K(x,t)u + \Psi(x,t) \notag \\
    L_t(\Psi) &= 0 \notag \\
    L_t(K) &= 0 \notag \\
    L_t(Y) - \Xi(\mu) + \sigma^2 K_x - L_t(m)\mu &= 0 \\
    L_t(m)\sigma^2 + \Xi(\sigma^2) - 2 \sigma^2 Y_x &= 0 \notag \\
    \sigma^2 m_x &= 0 \notag
\end{align}

Under the non-vanishing hypothesis imposed on $\sigma$ throughout the paper (see Section \ref{sezione lamperti}), the last condition implies that $m=m(t)$ does not depend on the spatial variable $x$. We also remark that $\Psi$ only appears in the first two conditions and corresponds to the trivial symmetry $\Psi \partial_u$, reflecting the superposition principle for linear PDEs. Consequently, in the following, we will often omit $\Psi$ and condition $L_t(\Psi) = 0$.\\
\\

Building on the connection between SDEs and second-order linear parabolic PDEs investigated in Section \ref{Section sde kbe}, the study of infinitesimal symmetries of SDEs proves to be a powerful tool for also determining the Lie-point symmetries of the associated PDEs. In \cite{paper2020}, a theorem established an equivalence between the symmetries of the KBE and those of the corresponding SDE. For the sake of completeness, we recall this theorem in the specific case of a one-dimensional SDE and its corresponding KBE:

\begin{teorema}\label{doob kbe}
    Let $ V=(\tilde{Y}, H)$ be an infinitesimal symmetry of Doob-type for the SDE \eqref{SDE}, with $ H=\sigma K_x$ and \( \tilde{Y}=(Y,m)\). Then, the vector field  
    \begin{equation} \label{vector_field}
    \Xi=m\partial_t+Y\partial_x-K u \partial_u    
    \end{equation} is a Lie point infinitesimal symmetry for the related KBE \eqref{kbe a}.\\
    Conversely, if \eqref{vector_field} is a Lie point infinitesimal symmetry for the KBE \eqref{kbe a}, then $ V=(\tilde{Y},H)$, with $ H=\sigma K_x$ and \( \tilde{Y}=(Y,m)\), is a Doob infinitesimal symmetry for the related SDE \eqref{SDE}.
\end{teorema}

This theorem constitutes a significant result and, together with the results obtained in Section \ref{Section sde kbe}, clarifies once again the importance of considering Doob transformations among the stochastic transformations under study. Indeed, if admissible transformations for an SDE are restricted to space-time diffeomorphisms \((x,t)\mapsto (\Phi(x,f(t)), f(t))\), then the symmetry group of an SDE generally proves to be a proper subset of the symmetry group of the corresponding associated KBE (see \cite{simmetriekolmogorov}, \cite{gaetarodriguez1}). On the other hand, if generic Girsanov measure changes are included among the admissible transformations, then the symmetries of a generic SDE become infinite-dimensional; in this case, they strictly contain the symmetry group of the associated KBE as special cases. If a full correspondence between the symmetries of SDEs and their associated KBEs is desired, the correct notion to consider is the specialization of Girsanov transformations to measure changes of the Doob type. Indeed, if we adopt the following notations
\begin{align*}
    \mathrm{Sym}_{(\Phi,f)}(\text{SDE}) &:= \{ T \mid T \text{ is a symmetry acting via } \tilde{\Phi}=(\Phi^f,f) \}, \\
    \mathrm{Sym}_{\text{Doob}}(\text{SDE}) &:= \{ T = (\tilde{\Phi}, h_D) \mid h_D \text{ is a Doob-type measure change} \}, \\
    \mathrm{Sym}_{\text{Gir}}(\text{SDE}) &:= \{ T = (\tilde{\Phi}, h_G) \mid h_G \text{ is a Girsanov-type measure change} \}, \\
    \mathrm{Sym}(\text{KBE}) &:= \{ \text{symmetries of the associated KBE} \},
\end{align*}
we have the following chain of inclusions:
\[
    \mathrm{Sym}_{(\Phi,f)}(\text{SDE}) \subset \mathrm{Sym}_{\text{Doob}}(\text{SDE}) \cong \mathrm{Sym}(\text{KBE}) \subset \mathrm{Sym}_{\text{Gir}}(\text{SDE}).
\]

\begin{osservazione}
    Theorem \ref{doob kbe} can now be recovered as a particular case of Theorem \ref{doob kbe inf}, simply considering symmetries as special cases of general stochastic transformations. 
\end{osservazione}

We can now summarize the results obtained in Theorem \ref{doob kbe}, Corollary \ref{doob kbe fin}, Theorem \ref{doob kbe inf}, and Theorem \ref{teo quasi doob} in the following classification corollary regarding the link between the symmetries of SDEs and those of the corresponding KBEs, whose proof follows immediately from the previous results.
\begin{corollario}\label{corollario quasi doob}
   Consider the SDE \eqref{SDE}, where \( W_t\) is a \(\mathbb{P}-\)Brownian motion, and the related KBE \eqref{kbe a}.
   If \( \hat{T}=(\tilde{\Phi}, \theta)\), acting as
   \[ t \rightarrow \hat{t}=f(t), \quad x\rightarrow\hat{x}=\Phi^f(x,t), \quad u \rightarrow \hat{u}(\hat{x},\hat{t})=\theta(x,t)u(x,t)+\Psi(x,t)\]
  and subject to strictly admissibility requirements \( L_t(\frac{\Psi}{\theta})=0\) and \(L_t(\theta)=\sigma^2\frac{\theta^2_x}{\theta}\), is a finite symmetry of KBE \eqref{kbe a}, then \( T=(\tilde{\Phi},h)\), with \( h=-\sigma\frac{\theta_x}{\theta} \), is a Doob-type symmetry for SDE \eqref{SDE}. Conversely, if \( T=(\tilde{\Phi},h)\) is a finite symmetry for SDE \eqref{SDE}, then
  \begin{itemize}
      \item if \( T\) is of Doob-type, then \( \hat{T}=(\tilde{\Phi},\theta)\), with \( |\theta|=e^{- k}\), is a finite symmetry for KBE \eqref{kbe a};
      \item if \( T\) is a generic quasi-Doob symmetry, then \(\hat{T}=(\tilde{\Phi},\theta)\), with \(\ln|\theta|=-k\) and \( h=\sigma k_x\), acting as 
      \[t \rightarrow \hat{t}=f(t), \quad x\rightarrow\hat{x}=\Phi^f(x,t), \quad u \rightarrow \hat{u}(\hat{x},\hat{t})=\theta(x,t)u(x,t),\]
      is not a proper symmetry of KBE \eqref{kbe a}, but it maps KBE \eqref{kbe a} into 
      \[ \hat{u}_{\hat{t}}+\mu(\hat{x},\hat{t})\hat{u}_{\hat{x}}+\frac{1}{2}\sigma^2(\hat{x},\hat{t})\hat{u}_{\hat{x}\hat{x}}+V(\hat{x},\hat{t})\hat{u}=0,\]
      where \( V(\hat{x},\hat{t})\) is given by
      \[ V(\hat{x},\hat{t})=\frac{1}{\theta f'}\big(\sigma^2 \frac{\theta_x^2}{\theta}-L_t(\theta)\big)\circ \tilde{\Phi}^{-1}.\] \end{itemize}
      In the same way, from the infinitesimal point of view, if  \( \ \Xi=m(t)\partial_t+Y(x,t)\partial_x-K(x,t)u(x,t)\partial_u\) is a Lie point infinitesimal symmetry of KBE \eqref{kbe a}, then \( V=(\tilde{Y},H)\), with \( \tilde{Y}=(Y,m)\) and \( H=\sigma K_x\), is an infinitesimal symmetry of SDE \eqref{SDE}. Conversely, if \( V=(\tilde{Y},H)\), with \( \tilde{Y}=(Y,m)\), is an infinitesimal symmetry of SDE \eqref{SDE}, then
      \begin{itemize}
          \item if \( V\) is of Doob-type, then \( \Xi=m(t)\partial_t+Y(x,t)\partial_x-K(x,t)u(x,t)\partial_u\) is a Lie point infinitesimal symmetry of KBE \eqref{kbe a};
          \item If \( V\) is a generic quasi-Doob infinitesimal symmetry, then \( \Xi=m(t)\partial_t+Y(x,t)\partial_x-K(x,t)u(x,t)\partial_u\) is not a proper symmetry of KBE \eqref{kbe a} in the sense of condition \eqref{eq:symmetry_condition}, but \(\left. \text{pr}^{(2)}\Xi (L_t[u]) \right|_{L_t[u]=0} = -u L_t(K).\)
      \end{itemize}
\end{corollario}

\begin{osservazione}
   Recall that in dimension 1, under regularity assumptions on $\sigma$, any generic Girsanov transformation is also of quasi-Doob type; therefore, in this setting, Corollary \ref{corollario quasi doob} provides a complete classification. These results can be generalized to dimensions $d \geq 1$, but in this case, there may exist generic Girsanov transformations that are neither of Doob type nor of quasi-Doob type. This framework will be the subject of future research.
\end{osservazione}

\subsection{Moving Lie symmetries from Brownian motion }\label{section_2}
Having established a correspondence result between the symmetries of SDEs and KBEs, we can now focus solely on SDEs and adopt a purely probabilistic approach. In this section, we prove that, using the remodeling map introduced in Section \ref{sezione 1.3}, it is always possible to recover the infinitesimal symmetries of a generic $SDE_{\mu,\sigma}$ with a constant-sign diffusion coefficient $\sigma$ from those of standard Brownian motion. 

\begin{lemma}\label{lemma T*}
    Given $SDE_{\mu,\sigma}$, consider a stochastic transformation $T$ which maps it into a new equation $SDE_{\mu_1,\sigma_1} \coloneqq SDE_{E_T(\mu),E_T(\sigma)}$. If $V$ is an infinitesimal symmetry of $SDE_{\mu,\sigma}$, then the push-forward $T_*(V)$ is an infinitesimal symmetry of $SDE_{\mu_1,\sigma_1}$.
\end{lemma}
\begin{proof}
    Let $T_\lambda$ be the one-parameter group of finite transformations generated by $V$ via the flow reconstruction \eqref{ricostruzione del flusso}. By construction, $T_*(V)$ is the infinitesimal transformation related to the one-parameter group $T\circ T_\lambda \circ T^{-1}$ (see Section \ref{section stochastic transformations}). In view of Remark \ref{remark eqdetfin}, $T \circ T_\lambda \circ T^{-1}$ fulfills the finite determining equations \eqref{eqdet finite}, since, using Remark \ref{composizione ben def}, we have that 
    \[ E_{T\circ T_{\lambda}\circ T^{-1}}(\mu_1,\sigma_1)=E_T(E_{T_{\lambda}}(E_{T^{-1}}(\mu_1,\sigma_1)))=E_T(E_{T_{\lambda}}(\mu,\sigma))=E_T(\mu,\sigma)=(\mu_1,\sigma_1).\]
    Thus, $T_*(V)$ is the infinitesimal transformation associated with a one-parameter group of finite symmetries, meaning it is an infinitesimal symmetry in the sense of Definition \ref{def sym}.
\end{proof}
This result, in conjunction with the remodeling map introduced in Section \ref{sezione 1.3}, provides a highly efficient strategy for computing the symmetries of a generic $SDE_{\mu,\sigma}$ with a constant-sign diffusion coefficient. Indeed, rather than solving the associated determining equations directly, one can derive the symmetries of the target equation by applying the push-forward via the remodeling map to the symmetries of standard Brownian motion. Consequently, determining the complete symmetry group of standard Brownian motion becomes a step of fundamental importance, which is addressed in the following lemma:

\begin{lemma}
 The one-dimensional Brownian motion admits the following  infinite-dimensional family of symmetries $\ V_{\alpha, \beta}$ :
\begin{equation}\label{Valpha}
V_{\alpha, \beta}=\Bigg( 
    \left(\frac{1}{2}\alpha'(t)x+ \beta(t) \right)\partial_x + \alpha(t)\partial_t, \  -\frac{1}{2}\alpha''(t) x -\beta'(t) \Bigg),  
\end{equation}
depending on two arbitrary functions of time $\alpha$ and $\beta$.   
\end{lemma}
\begin{proof}
    It is sufficient to solve infinitesimal determining equations \eqref{equa:eqa},\eqref{equa:eqb} with \( \mu=0\) and \(\sigma=1\). See Example 7.1 in \cite{paper2023}.
\end{proof}
\begin{osservazione}
The family of symmetries \eqref{Valpha} depends on two arbitrary functions of time, $\alpha$ and $\beta$. 
In particular, by reconstructing the flow, the action of the finite family of symmetries related to the parameter $\alpha$ (i.e., $V_{\alpha,0}$, obtained from \eqref{Valpha} by setting $\beta=0$) is related to the time invariance properties of Brownian motion. On the other hand, the parameter $\beta$ (i.e., $V_{0,\beta}$, obtained from \eqref{Valpha} by setting $\alpha=0$) is related, in one dimension, to the translational invariance due to Girsanov's Theorem. In higher dimensions, $V_{0,\beta}$ is related to rotational invariance. See \cite{proceeding} for a detailed discussion.
\end{osservazione}

\medskip

\begin{proposizione}\label{sym sde generiche}
The infinitesimal symmetries of a generic $ SDE_{\mu,\sigma}$ with constant-sign diffusion coefficient are exactly those given by the following infinite-dimensional family:
\begin{equation}\label{T*Va}
    V^\prime_{\alpha,\beta}=\Big( 
    \left( \frac{1}{2}\alpha'(t)\sigma(x,t)\cdot \Phi^{-f}(x,t)+  \sigma(x,t) \beta(t) + \alpha(t) \partial_t \Phi(\tilde{\Phi}^{-1}(x,t)) \right)  \partial_x +  \alpha(t)\partial_t, \ H_{\alpha, \beta}   \Big)
\end{equation}
where
\begin{align*}
    H_{\alpha, \beta} =\ & -\frac{1}{2}\alpha''(t) \cdot \Phi^{-f}(x,t) + \frac{\alpha'(t)}{2}\cdot h(\tilde{\Phi}^{-1}(x,t)) +  \frac{1}{2}\alpha'(t) \Phi^{-f}(x,t) \cdot h_x(\tilde{\Phi}^{-1}(x,t)) \\
    & +\alpha(t) h_t(\tilde{\Phi}^{-1}(x,t))-\beta'(t) + \beta(t) \cdot h_x(\tilde{\Phi}^{-1}(x,t)),
\end{align*}
\(\Phi^{-f}\) denotes the first component of \(\tilde{\Phi}^{-1}\) and $\Phi$ and $h$ are given by the remodeling map in Proposition \ref{remodeling_map}.
\end{proposizione}
\begin{proof}
Let $T=(\tilde{\Phi},h)$, with \(\tilde{\Phi}=(\Phi^f,f)\) and \( f=\mathrm{id}\), be the remodeling map introduced in Proposition \ref{remodeling_map}, which allows us to recover $SDE_{\mu,\sigma}$ from standard Brownian motion. Since \( f=\mathrm{id}\), \(\Phi^f=\Phi\). Lemma \ref{lemma T*} ensures that $T_*(V_{\alpha,\beta})$ is a symmetry for $SDE_{\mu,\sigma}$, where $V_{\alpha,\beta}$ is the infinite-dimensional family of symmetries \eqref{Valpha} of standard Brownian motion. Applying the push-forward formula \eqref{pushforward} to $V_{\alpha,\beta}$ yields exactly \eqref{T*Va}.

Conversely, we want to show that any symmetry of $SDE_{\mu,\sigma}$ is of the form \eqref{T*Va}. This is straightforward: if we denote by $T^{-1}$ the restoring map (see Definition \ref{def remodeling restoring}) that allows us to recover standard Brownian motion from $SDE_{\mu,\sigma}$, then Lemma \ref{lemma T*} ensures that for any symmetry $\tilde{V}$ of $SDE_{\mu,\sigma}$, $V\coloneqq(T^{-1})_*(\tilde{V})$ is a symmetry for standard Brownian motion. Thus, $\tilde{V}=T_*((T^{-1})_*(\tilde{V}))=T_*(V)$ is the push-forward of a symmetry of standard Brownian motion of the form \eqref{Valpha}. Again, the push-forward formula \eqref{pushforward} applied to the family \eqref{Valpha} yields \eqref{T*Va}.

\end{proof}

\begin{osservazione}
    Sections \ref{Section sde kbe} and \ref{section relazione sym} might erroneously suggest that, at least for the study of symmetries, stochastic transformations could be restricted solely to Doob-type transformations rather than general Girsanov ones, without any loss of information. However, the importance of general Girsanov-type transformations is clarified in the present section, since the remodeling map is not, in general, of Doob type, and in this work it plays a central role in determining the symmetries.
\end{osservazione}

\subsection{Strategy for the generic  KBE}\label{Strategy for the generic KBE}

In order to compute the Lie point symmetries of a generic  KBE \begin{equation}\label{generic kbe}
    u_t+\mu u_x +\frac{1}{2}\sigma^2 u_{xx}=0,
\end{equation} we may invoke Theorem \ref{doob kbe} and determine instead the Doob symmetries of the associated SDE \begin{equation}\label{generic sde}
    dX_t=\mu(X_t,t) dt+\sigma(X_t,t) dW_t.
\end{equation}  Rather than solving the determining equations directly, we plan to leverage the results established in Section \ref{section_2}, by transporting the generic family of symmetries of Brownian motion and requiring that its push-forward constitutes a Doob symmetry. Furthermore, to obtain more tractable conditions, we may reduce the problem to the case of stochastic differential equations with a unit diffusion coefficient $\sigma\equiv1$, and arbitrary drift $\mu$. The symmetries of a generic $\ SDE_{\mu,\sigma}$ can then be recovered via a Lamperti transformation. This reduction is justified by the well-known fact that diffeomorphisms preserve the Lie point symmetries of partial differential equations - a fact that, in our framework, can alternatively be derived by analyzing the Doob symmetries of the corresponding SDE, as in the following result.
\begin{lemma}\label{mu 1}
Diffeomorphisms and time changes preserve the set of Doob symmetries: if $\ V=(\tilde{Y},H)$ is a Doob symmetry for $ SDE_{\mu,\sigma}$, with \( H=\sigma K_x\) and \(\tilde{Y}=(Y,m)\), and $T=(\tilde{\Phi},0)$ is a stochastic transformation with \(\tilde{\Phi}=(\Phi^f,f)\), then \( T_*(V)\) (see \eqref{pushforward}) is still a Doob symmetry for the transformed equation $ SDE_{E_T(\mu),E_T(\sigma)},$ with \[ (\frac{1}{\sqrt{f'}}H)\circ \tilde{\Phi}^{-1}=E_T(\sigma) (K\circ \tilde{\Phi}^{-1})_x.\]  In particular, diffeomorphisms and time changes preserve the set of Lie point symmetry of the associated KBE.
\end{lemma}
\begin{proof}
    Since $ V$ is a Doob symmetry, from infinitesimal characterization of Doob transformations (Proposition \ref{caratterizzazione doob e qd infin}), there exists a function $ K$ such that $ H=\sigma  K_x$ and $ L_t(K)=0.$ The action of $ T$ on the SDE is given by Theorem \ref{trasformate}, while the push-forward of $ V$ is obtained from formula \eqref{pushforward}.
    We must prove that $ T_*(V)$ is a symmetry of Doob type for  $ SDE_{E_T(\mu),E_T(\sigma)},$ so we must provide a function $\tilde{K}$ such that $ (\frac{1}{\sqrt{f'}}H)\circ \tilde{\Phi}^{-1}=E_T(\sigma)\tilde{K}_x$ and $ E_T(L_t)[\tilde{K}]=0.$ By integrating the first condition, remembering that since $ V$ is Doob for $ SDE_{\mu,\sigma}$ it holds that $ H=\sigma K_x$, we find that $\tilde{K}$ must take the following form
    \[\tilde{K}=\int \Big(\frac{H}{\Phi^f_x \sigma}\Big)\circ \tilde{\Phi}^{-1}dx=\int \Big(\frac{K_x}{\Phi^f_x}\Big)\circ\tilde{\Phi}^{-1}dx=K\circ\tilde{\Phi}^{-1}+\gamma(t),\]
    where $\gamma$ (constant of integration with respect to $ x$) is an arbitrary function of time chosen so that $ L_t(\tilde{K})=0.$ A straightforward computation shows that it suffices to take $\gamma=0$:
    \[ E_T(L_t)[K\circ\tilde{\Phi}^{-1}]=\partial_t(K\circ\tilde{\Phi}^{-1})+E_T(\mu)\partial_x(K\circ\tilde{\Phi}^{-1})+\frac{1}{2}E_T(\sigma)^2\partial_{xx}(K\circ\tilde{\Phi}^{-1})
    =\]
    \[ -\frac{K_x \Phi^f_t}{\Phi^f_x}\circ\tilde{\Phi}^{-1}
    +\frac{1}{f'}K_t\circ\tilde{\Phi}^{-1}+\frac{1}{f'}\left(L_t(\Phi^f)\circ\tilde{\Phi}^{-1}\right) \Big(\frac{K_x}{\Phi^f_x}\Big)\circ\tilde{\Phi}^{-1}+\]
    \[
    +\frac{1}{2}\left(\frac{(\Phi^f_x)^2\sigma^2}{f'}\right)\circ\tilde{\Phi}^{-1}\frac{(K_{xx}\Phi^f_x)\circ\tilde{\Phi}^{-1}-(K_x\circ\Phi^f_{xx})\circ\tilde{\Phi}^{-1}}{(\Phi^f_x\circ\tilde{\Phi}^{-1})^3}=\]
    \[\frac{1}{f'}\big(K_t+\mu K_x+\frac{1}{2}\sigma^2 K_{xx}\big)\circ\tilde{\Phi}^{-1}=\frac{1}{f'}L_t(K)\circ\tilde{\Phi}^{-1}=0,\]
    and this completes the proof.
\end{proof}
\noindent Summarizing, our strategy for studying  the symmetries of a generic KBE \eqref{generic kbe}, is the following: 
\begin{itemize}
    \item Study the Doob symmetries of SDE \eqref{generic sde}, invoking Theorem \ref{doob kbe};
    \item Instead of finding the Doob symmetries of SDE \eqref{generic sde}, study  the Doob symmetries of $ SDE_{\mu,1}$ by Lemma \ref{mu 1};
    \item Instead of determining the Doob symmetries of $ SDE_{\mu,1}$ by solving the determining equations, invoke Section \ref{section_2} and study the conditions under which the push-forward of the generic family of symmetries~\eqref{doob mb} of Brownian motion constitutes a Doob symmetry for \( \mathrm{SDE}_{\mu,1} \);
    \item 
   Finally, to reconstruct the Lie point symmetries of the original generic KBE~\eqref{generic kbe}, it suffices to transport the Doob symmetries of $\mathrm{SDE}_{\mu,1}$ via the pushforward of the Lamperti-type transformation $\Phi$ that links $\mathrm{SDE}_{\mu,1}$ to $\mathrm{SDE}_{\mu,\sigma}$ (recall that Lemma~\ref{mu 1} ensures that diffeomorphisms preserve Doob symmetries). Subsequently, one applies Theorem~\ref{doob kbe} to convert the Doob symmetries of $\mathrm{SDE}$~\eqref{generic sde} into Lie point symmetries of the KBE~\eqref{generic kbe}. 
    \end{itemize} 
    
    We note that the transformation that maps the Brownian motion to $\mathrm{SDE}_{\mu,1}$ consists solely of a Girsanov-type contribution, without any Lamperti transformation, namely 
\begin{equation}\label{h=mu}
T = \big(\mathrm{Id}_{\mathbb{R}^2},\, h\big) \quad \text{with} \quad h = \mu.
\end{equation}
 Furthermore, the fact that the diffeomorphism is trivial allows one to handle more easily the algebraic conditions ensuring that the pushforward of $T$ is of Doob type for the transformed SDE, as we shall see in the next section.
\\
\\
\noindent The procedure just described is straightforward and almost algorithmic (see the examples in Section \ref{sezione esempi}).

\subsection{Doob symmetries of one-dimensional Brownian motion}\label{sottosezione 4.1}

The correspondence proved in Theorem \ref{doob kbe}  ensures that in order to study the symmetries of the KBE it is sufficient to study the Doob symmetries of the related SDE. Furthermore, as pointed out in Sections \ref{section_2} and \ref{Strategy for the generic KBE}, we can use the remodeling maps: instead of solving the determining equations for Doob symmetries for each SDE, we can simply push-forward the general Brownian motion symmetries and check whether they qualify as Doob-type symmetries. Thus, as a starting case, we aim to investigate Doob symmetries for one-dimensional Brownian motion.\\
\\
Concerning one-dimensional Brownian motion SDE (\ref{Brownian_motion_SDE}), we recall (see \eqref{Valpha}) that the most general time-dependent Girsanov symmetry of Browniam motion is given by the following infinite-dimensional family, depending on two arbitrary functions $\alpha(t)$ and $\beta(t)$:
\begin{equation}\label{Vab}
V_{\alpha, \beta}= \Bigg( 
   \left( \frac{1}{2}\alpha'(t)x +\beta(t)\right) \partial_x + \alpha(t)\partial_t, \ \ -\frac{1}{2}x \alpha''(t)-\beta'(t)
\Bigg) .   
\end{equation}
Here and in the following with the notation $F'$ we denote the derivative of a one variable function $ F$ with respect to its natural variable. 
Since we are interested in Doob symmetries, we have both to introduce a function $K(x,t)$ satisfying $K_x=-\frac{1}{2}x\alpha''(t)-\beta'(t)$, which gives
\begin{equation}\label{kBM}
K(x,t)=-\frac 14\alpha''(t)x^2-\beta'(t)x +\epsilon(t),
\end{equation}
where $\epsilon(t)$ is an arbitrary function of time, and  to impose $L_t(K)=K_t+\frac{1}{2}K_{xx}=0$ (see Proposition \ref{caratterizzazione doob e qd infin}), which yields
\begin{equation}\label{simmBM}
-\frac 14 \alpha''(t)-\frac 14 \alpha'''(t)x^2-\beta''(t)x+\epsilon'(t)=0.
\end{equation}
This means that
\begin{equation*}
    \begin{matrix}
        \alpha'''(t)=0; &
          \beta''(t)=0 & \epsilon'(t)=\frac14 \alpha''(t)\\
    \end{matrix}.
\end{equation*}
The first two conditions give $\alpha(t)=c_0+c_1t+c_2t^2$ and $\beta(t)=c_3+c_4t$, while the last condition allows to determine $\epsilon(t)=\frac 12 c_2t+c_5$.
In this way,  we recover the whole family of Doob symmetries of one-dimensional Brownian motion, whose generators are
\begin{equation}\label{doob mb}
     \begin{matrix}
        V_1=\Big(\partial_t, \ 0
            \Big)\\
        V_2=\Big(
            x\partial_x + 2t\partial_t, \ 0
        \Big)\\
        V_3=\Big(
            2xt \partial_x +  2t^2 \partial_t, \ -2x\Big)\\
         V_4=\Big(\partial_x,  \ 0 \Big)\\
        V_5=\Big(             t\partial_x ,  \ -1 \Big),
    \end{matrix}
\end{equation}
obtained, respectively, after choosing $(\alpha,\beta)=(1,0), (2t,0), (2t^2,0), (0,1), (0,t)$, that is, respectively, for constant( $c_0$)/linear($c_1$)/quadratic($c_2$) $\alpha$, and for constant($c_3$)/linear($c_4$) $\beta$.\\
\begin{osservazione}
Generators (\ref{doob mb}) are consistent - although it may not appear so at first - with formulas (4.18) and (4.20) in \cite{simmetriekolmogorov}. In fact, \cite{simmetriekolmogorov} states that the SDE~(\ref{Brownian_motion_SDE}) admits only three symmetries: \( X_1 = \partial_t \), which in our notation corresponds to the symmetry \( V_1 \); \( X_2 = 2t\partial_t + x\partial_x \), corresponding to \( V_2 \); and \( X_3 = \partial_x \), corresponding to \( V_4 \). This is unsurprising, since \cite{simmetriekolmogorov} restricts admissible symmetries of an SDE to the form \( X = m\partial_t + Y\partial_x \), without including any \( H \) term, as Girsanov-type changes of measure are not considered. As a result, only those symmetries in our framework with \( H = 0 \) - namely, \( V_1, V_2, V_4 \) - are recovered in \cite{simmetriekolmogorov}.
Nevertheless, \cite{simmetriekolmogorov} observes that the corresponding KBE admits additional symmetries beyond \( X_1, X_2, X_3 \), namely \( Y_1 = t\partial_x + xu\partial_u \) and \( Y_2 = 2t^2\partial_t + 2xt\partial_x + (x^2 - t)u\partial_u \). In our framework, Theorem~\ref{doob kbe} ensures that the number of Lie point symmetries of the KBE coincides with those of Doob-type of the related SDE. Accordingly, the SDE~(\ref{Brownian_motion_SDE}) also possesses two additional symmetries: \( V_5 \), corresponding to the KBE symmetry \( Y_1 \), and \( V_3 \), corresponding to \( Y_2 \).
In particular, \( V_5 \) is obtained by setting \( \alpha = 0 \) (i.e., \( c_0 = c_1 = c_2 = 0 \)) and \( \beta = t \) (i.e., \( c_3 = 0 \), \( c_4 = 1 \)). From~(\ref{vector_field}), this gives \( \epsilon = c_5 \) and \( K = -x + c_5 \) \footnote{It should be noted that constant $c_5$ has been omitted from the symmetry, as it represents its trivial part related to the superposition principle, see Section \ref{sottosezione 4.2}}. According to~(\ref{simmBM}), the symmetry of the KBE associated with \( V_5 \) is:
\[
t\partial_x + xu\partial_u,
\]
which is precisely \( Y_1 \). Similarly, \( V_3 \) corresponds to \( Y_2 \) with \( \alpha = 2t^2 \) (i.e., \( c_0 = c_1 = 0 \), \( c_2 = 2 \)), \( \beta = 0 \) (i.e., \( c_3 = c_4 = 0 \)), \( \epsilon = t + c_5 \), and \( K = -x^2 + t+ c_5 \).     
\end{osservazione}

This correspondence highlights the soundness of our theoretical framework and the importance of including Doob–Girsanov-type transformations among the admissible stochastic transformations of a SDE.

    \subsection{Lie point symmetries of backward heat equation }\label{sottosezione 4.2}
Consider the KBE associated with SDE \eqref{Brownian_motion_SDE}:
\begin{equation}\label{Brownian_motion_KBE}
    u_t+\frac{1}{2}u_{xx}=0.
\end{equation}
It is well known that equation~(\ref{Brownian_motion_KBE}) can be reduced to the classical backward heat equation \( u_t + u_{xx} = 0 \) via the diffeomorphism \footnote{Using the notation from Section \ref{sezione 1.3}, consider the transformation \( T = ((\Phi, f), h) = ((\sqrt{2}x, \mathrm{id}), 0) \).} \( \Phi(x) = \sqrt{2}x \).  Hence, up to a change of variables, the Lie point symmetries of~(\ref{Brownian_motion_KBE}) coincide with those of the classical heat equation (see also Lemma~\ref{mu 1} for a proof from a probabilistic perspective). For this reason, several authors~\cite{bluman,ricciardi}, when analyzing the symmetries of the KBE associated with one-dimensional Brownian motion, choose to study the classical heat equation instead; the original symmetries can then be recovered via push-forward of the inverse transformation.\\ 
\\
From the correspondence outlined in Theorem \ref{doob kbe},  the symmetries of the KBE \eqref{Brownian_motion_KBE} can be recovered from the Doob symmetries \eqref{doob mb} of the associated SDE \eqref{Brownian_motion_SDE}.
By adding the trivial symmetry \( \Xi = u \partial_u \), which is associated with the linearity of the equation, we recover - within our framework - the following generators of the well-known six-dimensional algebra of Lie point symmetries for the KBE associated with one-dimensional Brownian motion: 
    \begin{equation}\label{simmetrie kbe mb}
           \begin{matrix}
      \Xi_1=\partial_t \\
        \Xi_2= 2 t \partial_t+ x\partial_x\\
        \Xi_3=\partial_x\\
        \Xi_4=t\partial_x+x u \partial_u\\
        \Xi_5=2t^2\partial_t+2tx\partial_x+(x^2-t)u\partial_u\\
        \Xi_6=u\partial_u.
    \end{matrix}  
    \end{equation}

    \begin{osservazione}\label{Xi6}
        The presence of the symmetry \( \Xi_6 \) does not contradict Theorem \ref{doob kbe}, which states that the symmetries of the KBE correspond to the subgroup of Doob-type symmetries of the associated SDE. In fact, by applying Theorem \ref{doob kbe}, the symmetry \( \Xi_6 \) corresponds to a degenerate Doob symmetry with \( H = 0 \), that is, the trivial symmetry of the SDE corresponding to the identity transformation.
    \end{osservazione}

    \begin{osservazione}
   Among the generators (\ref{simmetrie kbe mb}), \(\Xi_1\) and \(\Xi_6\) are trivial symmetries corresponding, respectively, to the equation’s autonomy and linearity. From a probabilistic point of view, \(\Xi_1\) and \(\Xi_6\) are in fact related to the concept of \textit{first integral} (or \textit{conserved quantity}) of the associated SDE, i.e., a function \( I(x, t) \) that remains constant along solutions of the SDE: \( I(X_t, t) = \text{const} \). For \( I \) to be a first integral, it must satisfy \( \mathrm{d}I(X_t, t) = 0 \), which, by Itô’s formula, translates into the conditions
\begin{equation}\label{first int}
L_t(I) = 0 \quad \text{and} \quad \sigma^T \nabla I = 0.
    \end{equation}
From the infinitesimal characterization of Doob transformations (see Proposition \ref{caratterizzazione doob e qd infin}), it is clear that condition (\ref{first int}) corresponds to a trivial Doob transformation, with \( K=I\) and \( H=0.\)
As shown in \cite{simmetriekolmogorov, unal}, if \( I \) is a first integral, then the KBE always admits \( Iu \, \partial_u\) and, in the autonomous setting, \( I\partial_t \) as symmetries. Symmetries \( \Xi_1, \Xi_6 \)  are precisely of this type, as constant functions (e.g., \( I = 1 \)) are trivially preserved along the flow. Moreover, since such constants are preserved not only for Brownian motion but for any one-dimensional SDE, it is natural to expect - as we will later show - that the pushforward of this type of symmetry remains a symmetry for any one-dimensional autonomous SDE.  
   \end{osservazione}

\noindent Hence, the Lie algebra of infinitesimal symmetries of equation (\ref{Brownian_motion_KBE}) is spanned by the six vector fields \( \Xi_1,..,\Xi_6\) along with the infinite-dimensional subalgebra
\[ v_{K_0}=K_0(x,t)\partial_u,\]
where \( K_0\) is an arbitrary solution to the equation. This latter symmetry  \( v_{K_0}\) reflects the superposition principle, which holds for solutions of any linear PDE. \\
\\
Up to a change of variables (see Remark \ref{schrodinger}), with generators \(\Xi_1,...,\Xi_6\) we recover the well-known Schr\"odinger group Sch(1) of symmetries for one-dimensional Schr\"odinger equation/heat equation, which consists of six fundamental continuous symmetries of the free particle Schr\"odinger equation. These include 
\begin{itemize}
    \item \emph{Time translations} generated by \(\Xi_1\), associated with the Hamiltonian and the energy conservation, encoding the invariance of the system under shifts in time \( t \mapsto t+\lambda\). \footnote{Here and in the following, the finite transformation associated with a given infinitesimal generator is recovered integrating via the reconstruction of the flow \eqref{ricostruzione del flusso} .}
    \item \emph{Space translations} generated by \(\Xi_3,\) related to the homogeneity of space and the conservation of linear momentum,  corresponding to spatial shifts of the form \(x \mapsto x + \lambda\). From a probabilistic point of view, since KBE (\ref{Brownian_motion_KBE}) is related to one-dimensional Brownian motion SDE (\ref{Brownian_motion_SDE}), symmetry \(\Xi_3\) along with symmetry \(\Xi_1\) correspond to the following invariance property of Brownian motion (see, e.g., \cite{baldi}): if \(W_t\) is a Brownian motion with respect to filtration \( \mathcal{F}_t\), then \( W_{t+s}-W_s\) is a Brownian motion with respect to filtration \(\mathcal{F}_{t+s}.\)
    \item \emph{Galilean boosts} generated by \(\Xi_4\), corresponding to the invariance of the system under changes from one inertial reference frame to another moving at a constant velocity relative to the first (\( x \mapsto x + \lambda t\)). From a probabilistic perspective, this symmetry encodes Girsanov theorem with constant drift: if \( W_t\) is a $\mathbb{P}$-Brownian motion, then also \( W'_t=W_t+\int_0^t \lambda dt=W_t+\lambda t\) is a Brownian motion, but with respect to a novel probability measure $\mathbb{Q}$ with Radon Nikodym density given by \(  \frac{d\mathbb{Q}}{d\mathbb{P}}_{|_{\mathcal{F}_t}}=\exp\Big(-\lambda W_t-\frac{1}{2}\lambda^2 t\Big)\). The change of the underlying probability measure is reflected by the presence of a  non-trivial term in $\partial_u$, corresponding to non-trivial $ h$ terms in the stochastic transformation involving the related SDE (see again Theorem \ref{doob kbe} for correspondences between Lie point symmetries of KBE and Lie symmetries of SDE).
    \item \emph{Dilatations} generated by \( \Xi_2 \), corresponding to scaling transformations that rescale space and time differently, typically as \( x \mapsto \lambda x \) and \( t \mapsto \lambda^2 t \). This reflects the fundamental anisotropic scaling symmetry of the heat and Schrödinger equations. From a probabilistic perspective, the dilation symmetry corresponds to the fact that for any \( \lambda > 0 \), the rescaled process
\[
W_t^\lambda := \lambda W_{\frac{t}{\lambda^2}}
\]
is again a Brownian motion with respect to the rescaled filtration \(\mathcal{F}_{\frac{t}{\lambda^2}} \) (see again \cite{baldi}). 
\item \emph{Special conformal transformations} generated by  \( \Xi_5\), corresponding to nonlinear time reparametrizations of the form \( t \mapsto \frac{t}{1 + 2\lambda t} \), combined with a simultaneous rescaling of space \( x \mapsto \frac{x}{1 + 2\lambda t} \). From a probabilistic point of view, this symmetry encodes Girsanov theorem with non-constant drift, reflected by the presence of a non-trivial and time-dependent term in \( u\partial_u\).  This symmetry does not represent the most general Girsanov-type change of measure, since, in general, Girsanov-type symmetries do not transfer from the SDE to the KBE. Instead, it corresponds to the most general subclass of Doob-type changes of measure within the Girsanov framework, as again  Theorem \ref{doob kbe} ensures a correspondence between Doob symmetries of the SDE and Lie point symmetries of the KBE.
\item \emph{Global phase transformations} generated by \(\Xi_6\), corresponding to multiplying the  solution by a constant positive phase: \(  u \mapsto e^{\lambda} u \), related to the conservation of the total particle number and the linearity of the equation. This symmetry has a trivial probabilistic counterpart, since, as we already mentioned, it corresponds to the identity transformation on the related SDE.
\end{itemize}

\section{Symmetries classification in the autonomous case}\label{sezione classificazione}
 In this section, we present a constructive classification theorem for the symmetry algebra of an \textbf{autonomous} SDE and its associated KBE based on the drift and diffusion coefficients. Henceforth, the coefficients shall not depend on time: \(\mu=\mu(x) \) and \( \sigma=\sigma(x)\).\\
 \\

\begin{lemma}\label{Prop conddoob}
   An autonomous $ SDE_{\mu,1}$ admits Doob symmetries if and only if the following equation is satisfied
  \begin{equation}\label{CondDoob}
\left[\frac 12 x\alpha'(t)+\beta(t)\right] M'(x)+\alpha'(t) M(x)-\frac 14 \alpha'''(t)x^2-\beta''(t)x-\frac 14 \alpha''(t)+\tilde \epsilon'(t)=0
\end{equation}
where
\begin{equation}\label{Expression_M}
M(x)=\frac 12\left[\mu'(x)+\mu^2(x)\right]
\end{equation}
and \(\alpha, \beta\) and \(\tilde{\epsilon}\) are arbitrary functions of time.
\end{lemma}
\begin{proof}
    We first observe that a $ SDE_{\mu,1}$ can be obtained by Brownian motion via the remodeling map \(T=(Id_{\mathbb{R}^2}, \mu)\) (see Section \ref{Strategy for the generic KBE}). Thus, applying formula \eqref{T*Va} to our case, i.e with \(\Phi=\mathrm{Id}\), \(f=\mathrm{id}\) and $\sigma=1$, we have that $ SDE_{\mu,1}$ admits the following family of symmetries $\tilde{V}_{\alpha,\beta}:=T_*(V_{\alpha,\beta})$ \footnote{Remember that since \(h=\mu\) and \(\mu=\mu(x)\), then \(h_t=0.\)} given by
    \begin{equation}\label{simtras}
 \tilde{V}_{\alpha,\beta}=\Big( 
    \left(\frac{1}{2}\alpha'(t)x+\beta(t)\right) \partial_x+ \alpha(t) \partial_t, \
 \underbrace{ -\frac{1}{2}\alpha''(t)  x -\beta'(t)  + \frac{\alpha'(t)}{2} \mu(x) +  \left(\frac{1}{2}\alpha'(t) x  +\beta(t) \right) \mu'(x)}_{\tilde{H}}\Big).
    \end{equation}
    By Proposition \ref{caratterizzazione sym doob infin}, since \(\sigma=1 \), $ SDE_{\mu,1}$ admits Doob symmetries if and only if there exists a function \(\tilde{K}=\tilde{K}(x,t)\) such that
\[ \tilde{K}_x=\tilde{H} \quad and \quad \tilde{L}_t(\tilde{K})=0,\]
where \( \tilde{L}_t\) is the infinitesimal generator related to \( SDE_{\mu,1}\), namely \(\tilde{L}_t=\partial_t+\mu(x)\partial_x+\frac{1}{2}\partial^2_x.\)
Integrating the first Doob condition, we find
\begin{equation}\label{k}
\tilde K(x,t)=-\frac 14\alpha''(t)x^2-\beta'(t)x +\beta(t)\mu(x)+\frac 12 \alpha'(t)x \mu(x)+\tilde \epsilon(t),
\end{equation}
where $\tilde \epsilon$ is an arbitrary function of $t$ obtained as an integration constant. By imposing $\tilde{L}_t(\tilde K)=0$ we get \eqref{CondDoob} and \eqref{Expression_M}.

\end{proof}
The next result is the core of our symmetry analysis and provides a complete classification of an isotropic SDE (that is with unit diffusion coefficient), in agreement with previous complete classifications of the related KBE (see Remarks 6.6 and 6.7). 
\begin{teorema}[Classification theorem]\label{classification theorem}
Let us consider an autonomous $ SDE_{\mu,1}$
    A necessary condition for the $ SDE_{\mu,1}$ to have non-trivial Doob symmetries (i.e., symmetries other than those induced by autonomy) is that
    \begin{equation}\label{Cnec}
        M(x)=m_0+m_1 x + m_2 x^2 + \frac{m_3}{(x+m_4)^2},
    \end{equation}
    where \( M(x)\) is given by \eqref{Expression_M} and \( m_0,...,m_4\) are arbitrary constants. In other words, the associated KBE given by \eqref{generic kbe} admits non-trivial symmetries only if \eqref{Cnec} is satisfied. \\
    Conversely, considering the function $M(x)$ in \eqref{Cnec}, we can distinguish the following cases. 
    \begin{itemize}
        \item If \( m_3=0\), \( SDE_{\mu,1}\) admits 6 Doob symmetries, i.e. the related KBE \eqref{generic kbe}  has a maximal algebra of  symmetries with dimension 6. In this case, \( M'''=0.\)
        \item If \( m_3\not=0\) \( (M'''\not=0\)) but \( m_1=2m_2m_4\), then $ SDE_{\mu,1}$ admits 4 Doob symmetries, i.e. the related KBE  has a four-dimensional symmetry algebra.
        \item If \( m_3\not=0\) \( (M'''\not=0\)) and \( m_1\not=2m_2m_4\), then \( SDE_{\mu,1}\) admits only the trivial symmetry related to the autonomy, and  the related KBE admits only two trivial symmetries related to the linearity and the autonomy.
        \end{itemize}
\end{teorema}
\begin{proof}
    Let us first begin with the necessary condition: we assume that $ SDE_{\mu,1}$ admits Doob symmetries. From Lemma \ref{Prop conddoob}, equation \eqref{CondDoob} holds. If \( \alpha'=\beta=0\), equation \eqref{CondDoob} reduces to \( 0=0\) and we are referring to the trivial case of autonomy symmetry. Conversely, if \( \alpha'\not =0\), 
    the equation obtained by differentiating \eqref{CondDoob} three times with respect to \( x\) must also necessarily hold:
    \[ 5 \alpha'(t) M'''(x)+[\alpha'(t)x+2\beta(t)]M''''=0.\]
    This is a separable differential equation, which can be integrated to obtain
    \[ M'''=\frac{c(t)}{[\alpha'(t)x+2\beta(t)]^5}.\]
    Since we are in an autonomous regime, \( M\) depends only on \( x\), so also the right-hand side of the latter equation must depend only on \( x\). The only way to achieve this is imposing \( \beta(t)= k \alpha'(t)\) and \( c(t)=k_1 (\alpha'(t))^5\), with \( k,k_1 \in \mathbb{R}\) (possibly with \( k=0\) if \( \beta \equiv 0\)), so that
    \[ M'''=\frac{k_1}{(x+2k)^5}.\]
    Integrating three times with respect to \(x\), up to a renaming of the arbitrary constants, we obtain as a necessary condition
    \[  M(x)=m_0+m_1 x + m_2 x^2 + \frac{m_3}{(x+m_4)^2}.\]
Finally, if \( \alpha'=0\) but \( \beta \not =0\), then equation \eqref{CondDoob} becomes 
\[ \beta(t)M'(x)-\beta''(t)x+\epsilon'(t)=0,\]
which is again a separable differentiable equation that we can integrate to obtain
\[ M(x)= \frac{\beta''(t)}{2\beta(t)}x^2-\frac{\epsilon'(t)}{\beta(t)}x+c(t).\]
Since we are in an autonomous regime, \( M\) depends only on \(x\), so also the right-hand side of the latter equation must depend only on \( x\). The only way to achieve this is imposing \( \beta''(t)=m_2 \beta(t)\), \( \epsilon'(t)=m_1\beta(t)\) and \( c(t)=m_0\) for some constants \( m_0,m_1,m_2\), so that
\[ M(x)=m_0+m_1x+m_2 x^2,\]
as wanted (this is a particular assert of the thesis, with \(m_3=0\)) .
\\
    
    Let us now prove the sufficient conditions: we assume that \( M\) has the form \eqref{Cnec} and we first prove that if \( m_3=0\) then we have a maximal algebra of symmetries. Deriving \eqref{Cnec} with respect to \(x\) we get
    \[ M'=m_1+2m_2x-\frac{2m_3}{(x+m_4)^3}.\]
    We now substitute \( M\) and \( M'\) with these forms in the characterizing equation \eqref{CondDoob} of Lemma \ref{Prop conddoob}, and after some straightforward algebraic manipulations we obtain
    \[ x^2\left[2\alpha'(t)m_2
    -\frac{1}{4}\alpha'''(t)\right]+x\left[\frac{3}{2}\alpha'(t)m_1+2\beta(t)m_2-\beta''(t)\right]+\Big[\beta(t)m_1+\]\begin{equation}\label{eq M}+\alpha'(t)m_0-\frac{1}{4}\alpha''(t)+\tilde{\epsilon}'(t)\Big]+\frac{m_3\left(\alpha'(t)m_4-2\beta(t)\right)}{(x+m_4)^3}=0.\end{equation}
    If \( m_3=0\) (that is, if \(M'''=0\)), then \eqref{eq M} leads to the following conditions:
    \begin{equation}\label{sistema}
        \begin{cases}
            \alpha'''(t)=8m_2 \alpha'(t)\\
            \beta''(t)=2m_2\beta(t)+\frac{3}{2}m_1\alpha'(t)\\
            \tilde{\epsilon}'(t)=\frac{1}{4}\alpha''(t)-m_0\alpha'(t)-m_1\beta(t)
        \end{cases}
    \end{equation}
    Thus, if \( m_3=0\) (\(M'''=0\)), we have a maximal algebra of 6 symmetries of Doob-type, since the first equation in \eqref{sistema} gives 3 arbitrary constants for \(\alpha\), the second one gives 2 arbitrary constants for \(\beta\) and the case \((\alpha,\beta)=(0,0)\) gives \( V=0\) for \( SDE_{\mu,1}\) (i.e. the trivial symmetry given by the identity), corresponding to symmetry \(\Xi=u\partial_u\) for the related KBE (See Remark \ref{Xi6}). \\
    Let us now consider the case \( m_3\not=0\) (\( M'''\not=0\)). Then, in order for \eqref{eq M} to be satisfied, in addition to \eqref{sistema} we must impose the further condition that
    \begin{equation}\label{beta 4dim} \beta(t)=\frac{1}{2}\alpha'(t)m_4.\end{equation}
  This latter condition has to be consistent with the second equation of \eqref{sistema}: deriving two times with respect to time and substituting the result into (\ref{sistema}b), we get
  \begin{equation}\label{alpha 4dim} \alpha'(t)(2m_2m_4-m_1)=0,\end{equation}
  which implies that either \( m_1=2m_2m_4\) or \(\alpha'=0\). If  \( m_1=2m_2m_4\) , then the dimension of the algebra of symmetries is equal to 4, since now we have only 3 arbitrary constants for \(\alpha\) (\ref{sistema}.a) and, once again, the case \((\alpha,\beta)=(0,0)\) corresponding to the trivial symmetries \( V=0\) and \( \Xi=u\partial_u\) for the SDE and the related KBE, respectively, while \(\beta\) must obey \(\beta=\frac{1}{2}\alpha'm_4.\)\\

  Finally, if \(m_3\not=0\) and \(m_1\not=2m_2m_4\), then the only way to maintain consistency is to require \(\alpha'=0\): this leads to a minimal algebra of dimension 2, since we have only one arbitrary constant for \(\alpha\) (\(\alpha'=0\)) and, again the trivial case corresponding to \( (\alpha,\beta)=(0,0)\), while \(\beta=\frac{1}{2}\alpha'm_4=0\). These two symmetries correspond respectively to the autonomy of the system (\(\tilde{V}_{\alpha,\beta}\) in \eqref{simtras} with \((\alpha,\beta)=(1,0)\) gives \( V=\partial_t\)) and its linearity ( \(\tilde{V}_{\alpha,\beta}\) in \eqref{simtras} with \((\alpha,\beta)=(0,0)\) gives \( V=0\) for \( SDE_{\mu,1}\), which corresponds to symmetry \(\Xi=u\partial_u\) for the related KBE). 
\end{proof}

\begin{osservazione}
  In Theorem \ref{classification theorem}, as well as throughout the remainder of the paper, with a slight abuse of notation we regard the symmetry \(V=(0,0)\) of the SDE
as an independent trivial symmetry. As already explained in Remark \ref{Xi6}, this convention is justified by the need to establish a complete correspondence with the finite-dimensional basis of Lie symmetries of the associated KBE.
\end{osservazione}

\begin{osservazione}\label{costruttivo}
 The proof of the previous theorem is constructive. If $M'''=0$, the Doob-type symmetries for $\text{SDE}_{\mu,1}$ are given by \eqref{simtras}, where $\alpha$ and $\beta$ solve the system \eqref{sistema}.
If $M''' \neq 0$ but $m_1=2m_2m_4$, then the family of symmetries associated with the parameter $\beta$ degenerates and becomes linearly dependent on the family associated with the parameter $\alpha$ (see \eqref{beta 4dim}).
Finally, if $M''' \neq 0$ and $m_1 \neq 2m_2m_4$, then \eqref{beta 4dim} still holds, and furthermore we have $\alpha'=0$. Thus, the equation admits only the two trivial symmetries \( V=0\) and \( V=\partial_t\).
Similarly, we can explicitly construct the Lie point symmetries of the associated KBE.
Leveraging Theorem \ref{doob kbe}, the symmetries of the KBE take the form $\Xi=m\partial_t+Y\partial_x-Ku\partial_u$, where $m$ and $Y$ are the same as those appearing in the Doob symmetries of the associated SDE, and $K$ is given by \eqref{k},
where $\tilde{\epsilon}$ is determined by solving the corresponding equation in the system \eqref{sistema}.
\end{osservazione}

\begin{corollario}\label{corollario doob}
    Any $\text{SDE}_{\mu,1}$ of the form
    \[ dX_t=\mu(X_t)dt+dW_t, \]
    obtained from standard Brownian motion
    \[ dX_t=dW_t \]
    via a Doob transformation, admits six Doob symmetries. Consequently, the associated KBE admits a maximal six-dimensional Lie point symmetry algebra.
\end{corollario}

\begin{proof}
    Theorem \ref{classification theorem} ensures that $\text{SDE}_{\mu,1}$ admits six Doob symmetries if and only if 
    \[ M'''=0, \]
    where 
    \( M\) is given by \eqref{Expression_M}.
    However, $\text{SDE}_{\mu,1}$ is obtained from standard Brownian motion via the remodeling map $T=(\text{Id}, \text{id}, h)$ with $h=\mu$ (see Section \ref{sezione 1.3}). Since we are in one dimension, every Girsanov-type transformation is a quasi-Doob transformation. Thus, in this case, we have 
    \[ \mu=\partial_x \int \mu(x)\,dx, \]
    so we define 
    \[  k =\int \mu(x)\,dx. \]
    If we denote by $L_t=\partial_t+\frac{1}{2}\partial_{xx}$ the infinitesimal generator of standard Brownian motion, we have that 
    \[ M(x)=\frac{1}{2}(\mu'+\mu^2)=L_t( k )+\frac{1}{2}h^2. \]
    A particular case of the condition $M'''=0$ occurs when $M=0$. This implies that there exists a function $k$ such that $h=\partial_x k$ and 
    \[ L_t(k)+\frac{1}{2}h^2=0, \]
    which, given the unit diffusion coefficient of standard Brownian motion, means that $T$ is a Doob transformation from Brownian motion to $\text{SDE}_{\mu,1}$ (see Section \ref{section_3}).
\end{proof}

\begin{osservazione}
  The case discussed in Corollary \ref{corollario doob} occurs, for instance, in the Bessel process with a drift term of $\frac{1}{x}$ (see Example \ref{bessel processes}). Naturally, not every equation admitting a maximal algebra of Doob symmetries is obtained via a Doob-type remodeling map. A notable example is the Ornstein-Uhlenbeck process, which does not arise from a Doob transformation (see Example \ref{ornstein uhlenbeck}).
\end{osservazione}

\begin{osservazione}
The standard reference for a comprehensive symmetry classification of autonomous Fokker-Planck-type equations is arguably \cite{cicogna}. In this paper, the authors provide the following classification (rewritten using our notation):
\begin{itemize}
    \item If $M'''=0$, the symmetry algebra is maximal and six-dimensional (four non-trivial symmetries plus the two trivial ones arising from autonomy and linearity).
    \item If $M''' \neq 0$ but $M$ satisfies
    \begin{equation}\label{M cicogna}
        3M'=(c_1-M'')(x-c_0),
    \end{equation}
    then the symmetry algebra is four-dimensional (two non-trivial symmetries plus the two trivial ones).
    \item If $M''' \neq 0$ and \eqref{M cicogna} does not hold, the equation only admits the trivial symmetries.
\end{itemize}

Although derived using different techniques, this classification perfectly matches the one we obtained in Theorem \ref{classification theorem} via a probabilistic approach. Indeed, the condition for the maximal algebra is identical. Regarding condition \eqref{M cicogna}, we know from the first part of Theorem \ref{classification theorem}, together with the correspondence result in Theorem \ref{doob kbe}, that a necessary condition to admit non-trivial symmetries is that
\[ M(x)=m_0+m_1x+m_2x^2+\frac{m_3}{(x+m_4)^2}, \]
for some constants $m_i$, $i=0,\dots,4$.

Substituting the expressions for $M'$ and $M''$ into \eqref{M cicogna}, we obtain
\[ (3m_1+c_0c_1-2c_0m_2)+x(8m_2-c_1)-\frac{6m_3}{(x+m_4)^4}(c_0+m_4)=0, \]
which yields $c_0=-m_4$, $c_1=8m_2$, and
\[ m_1=2m_2m_4. \]

Thus, our conditions match exactly those given by \cite{cicogna}, demonstrating that the probabilistic approach is both fruitful and effective.
\end{osservazione}

\begin{osservazione}
    In \cite{gungor}, Theorem 3.1 provides a comprehensive classification of the symmetry algebras of one-dimensional, second-order linear parabolic PDEs, derived using different techniques. This classification is consistent with the one we established in Theorem \ref{classification theorem} using a probabilistic approach. Indeed, rewritten in our notation, Theorem 3.1 of \cite{gungor} states that a PDE of the form \eqref{generic kbe} admits a six-dimensional symmetry algebra if and only if the quantity \(M\) in \eqref{Expression_M} is a polynomial of degree two. Furthermore, it admits a four-dimensional symmetry algebra if and only if 
    $$M(x)=c_0+c_2x^2+\frac{c_3}{x^2},$$
    for some constants $c_i$, $i=0,2,3$, with $c_3 \neq 0$. 
    
    This is completely consistent with our Theorem \ref{classification theorem}. The condition for the maximal algebra is exactly the same, while the condition $m_1=2m_2m_4$, which we derived for the four-dimensional algebra, corresponds to the one given by \cite{gungor} after completing the square and applying Lemma \ref{mu 1}. Indeed, if we impose $m_1=2m_2m_4$ in 
    $$M(x)=m_0+m_1x+m_2x^2+\frac{m_3}{(x+m_4)^2},$$
    we can rewrite $M$ as follows:
    $$M(x)=\tilde{m}_0+m_2(x+m_4)^2+\frac{m_3}{(x+m_4)^2},$$
    where $\tilde{m}_0$ is a suitable constant. Performing the change of variables $x \mapsto x-m_4$, Lemma \ref{mu 1} ensures that the transformed equation shares the same number of symmetries as the original KBE. After this change of variables, the new potential is simply
    $$M(x)=\tilde{m}_0+m_2x^2+\frac{m_3}{x^2},$$
    which exactly matches the condition presented in \cite{gungor}.
\end{osservazione}

\begin{osservazione}
It is natural that, to achieve a maximal symmetry algebra, the quantity $M$ defined in \eqref{Expression_M} must be quadratic. Indeed, any $\text{SDE}_{\mu,1}$ can be mapped to standard Brownian motion via a change of measure with $h=\mu$. Given that $M$ acts as an invariant determining the number of symmetries of a given equation, and since Brownian motion exhibits a maximal symmetry algebra, the functional form of $M$ is intrinsically linked to the function $h$ induced by the symmetries of Brownian motion. Specifically, reconstructing the flow of \eqref{Vab} yields a quadratic $h$ for the family $V_\alpha$ and a constant $h$ for the family $V_\beta$. Consequently, the invariant $M=\frac{1}{2}(\mu'+\mu^2)=\frac{1}{2}(h_x+h^2)$ is necessarily quadratic.
\end{osservazione}

\begin{osservazione}\label{schrodinger}
    The quantity M, which characterizes the number of symmetry generators, also admits an interesting physical interpretation. In fact, formally it is possible to recast the KBE \eqref{generic kbe} into a Schrödinger-type equation through the following change of variables
\[t\mapsto -it, \ \ \ \ u\mapsto e^{-\int \mu(y)dy} u,\]
where the first change of time variable is known in literature as 'Wick rotation'. Indeed, with these changes, the KBE related to \( \text{SDE}_{\mu,1}\) becomes
\[ i u_t= \frac{1}{2} u_{xx}-\frac{1}{2}(\mu'+\mu^2)u,\]
that is a Schrodinger-type equation with potential $ V=M=\frac{1}{2}(\mu'+\mu^2).$ Thus, the condition $ M'''=0$ allows an at most quadratic potential: this is consistent with what is already well known in literature \cite{niederer,boyer}, since Schrodinger-type equations with at most quadratic potential admit the same extended group of symmetries as the free equation ($ V=0$). 
\end{osservazione}

We can now extend our classification to SDEs with non-normalized diffusion coefficient via the following theorem.
\begin{teorema}\label{classification mu sigma}
Consider an autonomous $\text{SDE}_{\mu,\sigma}$ given by 
\begin{equation}\label{generic sde a}
    dX_t=\mu(X_t)dt+\sigma(X_t)dW_t
\end{equation}
and the associated KBE given by 
\begin{equation}\label{generic kbe a}
    u_t+\mu(x)u_x+\frac{1}{2}\sigma^2(x)u_{xx}=0.
\end{equation}
A necessary condition for \eqref{generic sde a} to admit Doob symmetries is that
\begin{equation}\label{Cnec1}
    M(x)=m_0+m_1 x + m_2 x^2 + \frac{m_3}{(x+m_4)^2},
\end{equation}
where 
\begin{equation}
    \label{M mu sigma}
    M(x)=\frac{1}{2}\left( \frac{\mu^2}{\sigma^2}+\mu'-\frac{2\mu\sigma'}{\sigma}+\frac{1}{4}\sigma'^2-\frac{1}{2}\sigma\sigma''\right)\circ\Phi^{-1}
\end{equation}
and $m_0,\dots,m_4$ are arbitrary constants. In other words, the KBE \eqref{generic kbe a} admits non-trivial symmetries only if \eqref{Cnec1} is satisfied. \\
Conversely,  considering \eqref{Cnec1}, we have the following cases:
\begin{itemize}
    \item If $m_3=0$, then \eqref{generic sde a} admits six Doob symmetries; i.e., the KBE \eqref{generic kbe a} possesses a maximal symmetry algebra of dimension six. In this case, $M'''=0$.
    \item If $m_3 \neq 0$ (so $M''' \neq 0$), but $m_1=2m_2m_4$, then \eqref{generic sde a} admits four Doob symmetries; i.e., the associated KBE \eqref{generic kbe a} admits a four-dimensional Lie symmetry algebra.
    \item If $m_3 \neq 0$ ($M''' \neq 0$) and $m_1 \neq 2m_2m_4$, then the KBE \eqref{generic kbe a} admits only the two trivial symmetries arising from linearity and autonomy.
\end{itemize}
\end{teorema}

\begin{proof}
    The proof is a direct consequence of Theorem \ref{classification theorem} and Lemma \ref{mu 1}. Indeed, applying the Lamperti transformation $\Phi$ (see Section \ref{sezione lamperti}) to \eqref{generic sde a}, the transformed equation, denoted as $\text{SDE}_{\tilde{\mu},1}$, reads
    \[ dX_t=\tilde{\mu}(X_t)dt+dW_t, \]
    where $\tilde{\mu}=E_T(\mu)$, as given by \eqref{ET}, satisfies (remember that since \( f=\mathrm{id}\), then \(\Phi^f=\Phi\))
    \[ \tilde{\mu}=E_T(\mu)=L_t(\Phi)\circ\tilde{\Phi}^{-1}. \]
    Lemma \ref{mu 1} ensures that $\text{SDE}_{\tilde{\mu},1}$ and $\text{SDE}_{\mu,\sigma}$ admit the same number of symmetries. Hence, we can leverage Theorem \ref{classification theorem} to classify the dimension of the symmetry algebra by substituting $\tilde{\mu}$ in terms of $\mu$, noting that 
    \[ M(x)=\frac{1}{2}(\tilde{\mu}'+\tilde{\mu}^2)=\frac{1}{2}\left( \frac{\mu^2}{\sigma^2}+\mu'-\frac{2\mu\sigma'}{\sigma}+\frac{1}{4}\sigma'^2-\frac{1}{2}\sigma\sigma''\right)\circ\tilde{\Phi}^{-1}.\]
\end{proof}

\begin{osservazione}\label{costruttivo sigma}
The previous result is constructive too. We can explicitly determine the Doob symmetries of $\text{SDE}_{\mu,\sigma}$ in \eqref{generic sde a} by leveraging Remark \ref{costruttivo}, the Lamperti transformation, and Lemma \ref{mu 1}. First, by applying Remark \ref{costruttivo}, we compute the Doob symmetries of $\text{SDE}_{\tilde{\mu},1}$. Then, since $\text{SDE}_{\mu,\sigma}$ is obtained from $\text{SDE}_{\tilde{\mu},1}$ via the inverse Lamperti transformation, we apply Lemma \ref{mu 1}. Specifically, given an infinitesimal symmetry \(V=(\tilde{Y}, H ) \), with \( \tilde{Y}=(Y,m)\), of $\text{SDE}_{\tilde{\mu},1}$, the corresponding symmetry of $\text{SDE}_{\mu,\sigma}$ is given by $T_*(V)$. Here, $T=(\tilde{\Phi}, 0)$ with \( \tilde{\Phi}=(\Phi,\mathrm{id})\) and $\Phi=F^{-1}$ is the inverse of the Lamperti map \eqref{Lamperti}. The push-forward $T_*(V)$ is explicitly given by \eqref{pushforward}.
\end{osservazione}

\section{Examples}\label{sezione esempi}

\subsection{Bessel processes}\label{bessel processes}
Consider the generalized Bessel process
\begin{equation}\label{Bessel}
    dX_t= \frac{a}{X_t}dt + dW_t
\end{equation}
and the related KBE
\begin{equation}\label{kbe bessel}
    u_t+\frac{a}{x}u_x+\frac{1}{2}u_{xx}=0.
\end{equation}
Applying formulas (\ref{Reconstruction map}), the reconstruction map that allows us to recover equation (\ref{Bessel}) from one-dimensional Brownian motion (\ref{Brownian_motion_SDE}) is
\begin{equation}\label{trasformazione bessel}
    \Phi=\mathrm{id}; \quad h=\frac{a}{x}.
\end{equation}
Since the diffusion coefficient is already equal to 1, we do not need to apply the Lamperti transformation and we can directly apply Theorem \ref{classification theorem} and compute
\[ M = \frac{1}{2} (\mu'+\mu^2) = \frac{1}{2} \frac{a(a-1)}{x^2}.  \]
Thus, with the notations of Theorem \ref{classification theorem}, 
\begin{equation}\label{coefficienti} m_0 = m_1 = m_2 = m_4 = 0, \quad \quad m_3=\frac{a(a-1)}{2}.\end{equation}
Therefore, we have the following cases: 
\begin{itemize}
    \item If $a=0$, then $M=0$, and \eqref{Bessel} admits six Doob symmetries (including the trivial symmetry given by the identity; see Remark \ref{Xi6}). Consequently, by Theorem \ref{doob kbe fin}, \eqref{kbe bessel} also admits a maximal algebra of six Lie point symmetries. This is a trivial case, since for $a=0$, \eqref{Bessel} reduces to standard Brownian motion. Hence, the Doob symmetries of \eqref{Bessel} with $a=0$ are the same as those computed in Section \ref{sottosezione 4.1}, while the Lie point symmetries of \eqref{kbe bessel} with $a=0$ are presented in Section \ref{sottosezione 4.2}.

\item If $a=1$, then $M=0$ again, and Theorem \ref{classification theorem} ensures that \eqref{Bessel} admits six Doob symmetries (including the trivial symmetry given by the identity; see Remark \ref{Xi6}). To compute these symmetries, we exploit Remark \ref{costruttivo}: we substitute $m_0=m_1=m_2=m_3=m_4=0$, obtained from \eqref{coefficienti}, into \eqref{sistema} and solve the resulting system. We thus find $\alpha'''=0$ and $\beta''=0$. Substituting these into \eqref{simtras} and including the trivial symmetry, we conclude that the Doob symmetries of \eqref{Bessel} for $a=1$ are given by:
\begin{align*}
    &V_1=\Big( \partial_t, \ 0\Big),\\ 
    &V_2=\Big( x \partial_x +  2t\partial_t, \ 0\Big),\\
    &V_3=\Big( 2xt\partial_x + 2t^2\partial_t, \ -2x\Big),\\
    &V_4=\Big( \partial_x ,   \ -\frac{1}{x^2}\Big),\\
    &V_5=\Big(t \partial_x ,  \ -1-\frac{t}{x^2}\Big),\\
    &V_6=\Big( 0,  \ 0\Big).
\end{align*}

Theorem \ref{doob kbe} ensures that \eqref{kbe bessel} also admits a maximal algebra of six Lie point symmetries, given by $\Xi_i=m_i(t)\partial_t+Y_i(x,t)\partial_x-K_i(x,t)u\partial_u$ for $i=1,\dots,6$, where $m_i$ and $Y_i$ correspond to those in $V_i$, and $K_i$ is determined by \eqref{k} after solving the equation for $\epsilon$ in \eqref{sistema}. Thus, the six-dimensional symmetry algebra of \eqref{kbe bessel} with $a=1$ is spanned by the following generators:
\begin{align*}
    &\Xi_1=\partial_t,\\
    &\Xi_2=2t\partial_t+x\partial_x,\\
    &\Xi_3=2t^2\partial_t+2xt\partial_x- (3t-x^2)u\partial_u,\\
    &\Xi_4=\partial_x- \frac{1}{x}u\partial_u,\\
    &\Xi_5=t\partial_x-  \left(\frac{t}{x}-x\right)u\partial_u,\\
    &\Xi_6=u\partial_u.
\end{align*}
  \item If $a \neq 0$ and $a \neq 1$, then from \eqref{coefficienti} we have $M''' \neq 0$ but $m_1=2m_2m_4$. Hence, Theorem \ref{classification theorem} ensures that \eqref{Bessel} admits four Doob symmetries (including the trivial symmetry given by the identity; see Remark \ref{Xi6}). To compute these symmetries, we exploit Remark \ref{costruttivo}: substituting $m_0=m_1=m_2=m_4=0$ and $m_3=\frac{a(a-1)}{2}$, obtained from \eqref{coefficienti}, into \eqref{sistema}, we solve the resulting system, recalling that \eqref{beta 4dim} and \eqref{alpha 4dim} hold in this regime. We thus find $\alpha''=0$ and $\beta=0$. Substituting these into \eqref{simtras} and including the trivial symmetry, we conclude that the Doob symmetries of \eqref{Bessel} for a generic $a \neq 0,1$ are given by:
\begin{align*}
    &V_1=\Big(\partial_t, \ 0\Big),\\
    &V_2=\Big( x\partial_x + 2t \partial_t, \ 0\Big),\\
    &V_3=\Big( 2xt\partial_x + 2t^2\partial_t, \ -2x\Big),\\
    &V_4=\Big( 0,  \ 0\Big).
\end{align*}

Theorem \ref{doob kbe} ensures that \eqref{kbe bessel} also admits an algebra of four Lie point symmetries, given by $\Xi_i=m_i(t)\partial_t+Y_i(x,t)\partial_x-K_i(x,t)u\partial_u$ for $i=1,\dots,4$, where $m_i$ and $Y_i$ correspond to those in $V_i$, and $K_i$ is determined by \eqref{k} after solving the equation for $\epsilon$ in \eqref{sistema}. Thus, the four-dimensional symmetry algebra of \eqref{kbe bessel} with a generic $a \neq 0,1$ is spanned by the following generators:
\begin{align*}
    &\Xi_1=\partial_t,\\
    &\Xi_2=2t\partial_t+x\partial_x,\\
    &\Xi_3=2t^2\partial_t+2xt\partial_x-[t(2a+1)-x^2]u\partial_u,\\
    &\Xi_4=u\partial_u.
\end{align*}
\end{itemize}

\begin{osservazione}\label{Bessel condizionamento}
It is not surprising that for $a=1$, \eqref{Bessel} admits a maximal symmetry algebra. Indeed, in this regime, it can be shown that the Bessel process is actually a Brownian motion conditioned never to return to the origin (see \cite[Example 1]{chetrite2015nonequilibrium}). Furthermore, in this case, we have $M=0$, which corresponds exactly to the scenario discussed in Corollary \ref{corollario doob} for a process obtained from standard Brownian motion via a Doob transformation.
\end{osservazione}

\begin{osservazione}
    The fact that, even in the case \( a \neq 0,1 \), we determine the symmetries of \eqref{Bessel} via a reduction map to Brownian motion might appear to contradict the classical result by Sophus Lie (\cite{lie1881}). According to this result, a linear parabolic PDE  can be mapped, one-to-one, into the heat equation (the governing PDE of a Wiener process) if and only if it is invariant under a six-parameter continuous group (Lie group) of point transformations. 

Indeed, for \( a \neq 0,1 \), the PDE \eqref{kbe bessel} admits only 4 symmetries. However, from Theorem \ref{teo quasi doob}, we know that the transformation corresponding to \eqref{trasformazione bessel} is \( \hat{T}=(\tilde{\Phi}, \theta) \), with \( \tilde{\Phi}=\mathrm{Id} \) and \( \theta = \frac{1}{x^a} \), which maps the heat equation \eqref{Brownian_motion_KBE} into
\[
    u_t + \frac{a}{x} u_x + \frac{1}{2} u_{xx} + \frac{a(a-1)}{2x^2}u = 0.
\]
This latter equation does, in fact, admit 6 symmetries. Consequently, Lie's result is entirely consistent with our strategy for determining symmetries: the introduction of the potential term is precisely what compensates the equation, enabling it to attain the maximal symmetry algebra.
\end{osservazione}

\begin{osservazione}\label{remark bessel localization}
The reconstruction map \eqref{trasformazione bessel} has to be interpreted with some care, due to the singularity of
\[
    h(x)=\frac{a}{x}
\]
at the origin. Let \(B_t=x_0+W_t\), with \(x_0>0\), and consider the state space \(I=(0,+\infty)\). Since Brownian motion may hit the boundary \(0\) in finite time, the global non-explosion criterion for Girsanov transformations does not apply directly.

Nevertheless, the transformation is well defined locally. Indeed, for
\[
    \tau_n
    :=
    \inf\left\{
        t\geq0:
        B_t\notin\left(\frac1n,n\right)
    \right\},
\]
the process \(h(B_{t\wedge\tau_n})\) is bounded. Hence Novikov's condition is satisfied on every finite time interval, and the corresponding Dol\'eans--Dade exponential is a true martingale. By Girsanov's theorem, under the transformed measure,
\[
    W_t^{(n)}
    :=
    W_t-
    \int_0^{t\wedge\tau_n}
    \frac{a}{B_s}\,ds
\]
is a Brownian motion and, up to \(\tau_n\),
\[
    dB_t=\frac{a}{B_t}\,dt+dW_t^{(n)}.
\]
Thus \eqref{trasformazione bessel} is rigorously well defined after localization, and hence locally on \((0,+\infty)\).

Moreover, when \(a\geq\frac12\), the corresponding Bessel process does not hit the origin, so the localized construction extends to the global dynamics on \((0,+\infty)\). For \(a=0\), the transformation is trivial, since \(h\equiv0\).
\end{osservazione}

\subsection{Ornstein-Uhlenbeck}\label{ornstein uhlenbeck}\label{Esempio OU}

Consider the Ornstein-Uhlenbeck (OU) process defined by the SDE
\begin{equation}\label{SDE OU}
    dX_t = \theta(\mu - X_t)dt + \sigma dW_t,
\end{equation}
with \( \theta,\sigma >0\),
and its associated KBE
\begin{equation}\label{KBE OU}
    u_t + \theta(\mu - x)u_x + \frac{1}{2}\sigma^2 u_{xx} = 0.
\end{equation}
Since the diffusion coefficient is not normalized, we apply Theorem \ref{classification mu sigma} (instead of Theorem \ref{classification theorem}) to compute 
\[ M(x) = \frac{1}{2}\left( \frac{\mu^2}{\sigma^2} + \mu' - \frac{2\mu\sigma'}{\sigma} + \frac{1}{4}\sigma'^2 - \frac{1}{2}\sigma\sigma'' \right) \circ F^{-1}, \]
where $F$ is the Lamperti map $F(x) = \int \frac{dx}{\sigma} = \frac{x}{\sigma}$. Substituting the parameters of the OU process, we obtain
\[ M(x) = \left( \frac{\theta^2 \mu^2}{2\sigma^2} - \frac{\theta}{2} \right) - \frac{\theta^2 \mu}{\sigma}x + \frac{\theta^2}{2}x^2. \]
Thus, following the notation of Theorem \ref{classification mu sigma}, the coefficients are:
\begin{equation}\label{mi OU}
    m_0 = \frac{\theta^2 \mu^2}{2\sigma^2} - \frac{\theta}{2}, \quad m_1 = -\frac{\theta^2 \mu}{\sigma}, \quad m_2 = \frac{\theta^2}{2}, \quad m_3 = m_4 = 0.
\end{equation}
Theorem \ref{classification mu sigma} ensures that the SDE \eqref{SDE OU} admits six Doob symmetries (including the trivial identity symmetry). To compute these, we substitute the coefficients \eqref{mi OU} into the system \eqref{sistema}, yielding:
\[
\begin{cases}
    \alpha'''(t) = 4\theta^2 \alpha'(t), \\
    \beta''(t) = \theta^2 \beta(t) - \frac{3\theta^2 \mu}{2\sigma} \alpha'(t).
\end{cases}
\]
Solving this system leads to the following solutions $(\alpha, \beta)$:
\begin{equation}\label{sol ou}
    (1, 0), \quad (e^{2\theta t}, -\tfrac{\theta \mu}{\sigma}e^{2\theta t}), \quad (e^{-2\theta t}, \tfrac{\theta \mu}{\sigma}e^{-2\theta t}), \quad (0, e^{\theta t}), \quad (0, e^{-\theta t}).
\end{equation}
We then apply the Lamperti map \( F \) to leverage Remark \ref{costruttivo} and map the symmetries to the reduced SDE \[dX_t = \tilde{\mu}dt + dW_t,\] where \[\tilde{\mu} = \frac{\theta}{\sigma}(\mu - \sigma x).\] Using the push-forward defined by \eqref{pushforward} with $\Phi = F^{-1}$, we conclude that the six Doob symmetries $V_i$ of \eqref{SDE OU} are:
\begin{align*}
    V_1 &= \left( \partial_t, \ 0 \right), \\
    V_2 &= \left(  -e^{2\theta t} \theta (\mu - x)\partial_x +  e^{2\theta t}\partial_t, \ \frac{4\theta^2}{\sigma} e^{2\theta t} (\mu - x) \right), \\
    V_3 &= \left(  \theta e^{-2\theta t}(\mu - x)\partial_x + e^{-2\theta t}\partial_t, \  0 \right), \\
    V_4 &= \left( \sigma e^{\theta t}\partial_x,  \ -2\theta e^{\theta t} \right), \\
    V_5 &= \left( \sigma e^{-\theta t}\partial_x,   \ 0 \right), \\
    V_6 &= \left(  0,  \  0 \right).
\end{align*}
Theorem \ref{doob kbe} ensures that \eqref{KBE OU} also admits a six-dimensional Lie point symmetry algebra, spanned by $\Xi_i = m_i(t)\partial_t + Y_i(x,t)\partial_x - \bar{K}_i(x,t)u\partial_u$ for $i=1,\dots,6$. 
The coefficients $m_i$ and $Y_i$ correspond to those in the Doob symmetries $V_i$, while $\bar{K}_i = K \circ \tilde{\Phi}^{-1}$ (see Lemma \ref{mu 1}). 
Here, $K$ is determined by the relation \eqref{k}, which follows from solving the equation for $\epsilon$ in the system \eqref{sistema}. 
Thus, the symmetry algebra of \eqref{KBE OU} is spanned by:

\begin{align*}
    \Xi_1 &= \partial_t, \\
    \Xi_2 &= -e^{2\theta t}\theta (\mu - x)\partial_x + e^{2\theta t}\partial_t - \left[ -\frac{2 \theta^2 e^{2\theta t}}{\sigma^2}x^2 + \frac{4 \theta^2\mu e^{2\theta t}}{\sigma^2}x - 2 \frac{\theta^2 \mu^2 e^{2\theta t}}{\sigma^2} + \theta e^{2 \theta t} \right]u\partial_u, \\
    \Xi_3 &= \theta e^{-2\theta t}(\mu - x)\partial_x + e^{-2\theta t} \partial_t, \\
    \Xi_4 &= \sigma e^{\theta t}\partial_x - \frac{2\theta e^{\theta t}}{\sigma}(\mu - x) u \partial_u, \\
    \Xi_5 &= \sigma e^{-\theta t}\partial_x, \\
    \Xi_6 &= u\partial_u.
\end{align*}

\subsection{CIR model}
Consider the CIR model
\begin{equation}\label{CIR SDE}
dX_t=(aX_t+b)dt+\sigma \sqrt{X_t}dW_t
\end{equation}
and the related KBE 
\begin{equation}\label{CIR KBE}
u_t+(ax+b)u_x+\frac{1}{2}\sigma^2 x u_{xx}=0.
\end{equation}
We assume throughout this example that \(X_0=x_0>0\), \(\sigma>0\), \(b>0\), and that the Feller condition
\[
    2b\geq \sigma^2
\]
is satisfied. Under this condition, the boundary \(0\) is unattainable for the CIR process, and therefore
\[
    \mathbb{P}(X_t>0 \text{ for all } t\geq0)=1.
\]
Hence, the natural state space can be taken to be \(I=(0,+\infty)\), on which both the coefficients of the SDE and the Lamperti transformation used below are smooth. 
Since the diffusion coefficient is not normalized, we apply Theorem \ref{classification mu sigma} (instead of Theorem \ref{classification theorem}) to compute
\[ M(x) = \frac{1}{2}\left( \frac{\mu^2}{\sigma^2} + \mu' - \frac{2\mu\sigma'}{\sigma} + \frac{1}{4}\sigma'^2 - \frac{1}{2}\sigma\sigma'' \right) \circ F^{-1}, \]
where $F$ is the Lamperti map $F(x) = \int \frac{dx}{\sigma \sqrt{x}} = \frac{2\sqrt{x}}{\sigma}$. Substituting the parameters of the CIR model, we obtain
\[ M(x) = \left[\frac{ab}{\sigma^2}+\frac{a^2}{8}x^2+\left(\frac{2b^2}{\sigma^4}-\frac{2b}{\sigma^2}+\frac{3}{8}\right)\frac{1}{x^2} \right]. \]
Thus, following the notation of Theorem \ref{classification mu sigma}, the coefficients are:
\begin{equation}\label{mi CIR}
    m_0 =\frac{ab}{\sigma^2} , \quad m_1 =0 , \quad m_2 =\frac{a^2}{8}, \quad m_3 =\frac{2b^2}{\sigma^4}-\frac{2b}{\sigma^2}+\frac{3}{8}, \quad m_4 =0 .
\end{equation}

Setting
\[
    r=\frac{b}{\sigma^2},
\]
the coefficient \(m_3\) in \eqref{mi CIR} can be written as
\[
    m_3=2r^2-2r+\frac{3}{8}.
\]
Hence
\[
    m_3=0
    \quad\Longleftrightarrow\quad
    r=\frac14
    \quad\text{or}\quad
    r=\frac34.
\]
The first possibility is excluded by the Feller condition, since
\(r\geq1/2\). Therefore, within the standard CIR regime, the only exceptional
case is
\[
    b=\frac34\sigma^2.
\]

If
\[
    b\neq\frac34\sigma^2,
\]
then \(m_3\neq0\), while \(m_1=2m_2m_4=0\), and Theorem \ref{classification mu sigma} ensures that SDE \eqref{CIR SDE} admits four Doob symmetries (including the trivial identity symmetry). To compute these, we substitute the coefficients \eqref{mi CIR} into the system \eqref{sistema}, recalling that \eqref{beta 4dim} and \eqref{alpha 4dim} hold in this regime. We thus find \( \beta=0\) and
\[
    \alpha'''(t) =a^2 \alpha'(t) 
\]
Solving these equations leads to the following solutions $(\alpha, \beta)$:
\begin{equation}\label{sol cir}
(1,0), \quad (e^{at},0), \quad (e^{-at},0).
\end{equation}
We then apply the Lamperti map to leverage Remark \ref{costruttivo} and map the symmetries to the reduced SDE \[dX_t = \tilde{\mu}dt + dW_t,\] where \[\tilde{\mu} = L_t(F)\circ F^{-1}=\frac{a}{2}x+\left(\frac{2b}{\sigma^2}-\frac{1}{2}\right)\frac{1}{x}. \] Using the push-forward defined by \eqref{pushforward} with $\Phi = F^{-1}$, we conclude that the four Doob symmetries $V_i$ of \eqref{CIR SDE} are:
\begin{align*}
    V_1 &= \left( \partial_t, \ 0\right) , \\
    V_2 &= \left( a e^{a t}x\partial_x + e^{at}\partial_t, \  0 \right),\\
    V_3 &=\left( -a e^{-at}x \partial_x + e^{-at}\partial_t, \  - \frac{2a^2}{\sigma} e^{-at} \sqrt{x}\right),\\
    V_4 &= \left(  0 , \ 0 \right).
\end{align*}
Theorem \ref{doob kbe} ensures that \eqref{CIR KBE} also admits a four-dimensional Lie point symmetry algebra, spanned by $\Xi_i = m_i(t)\partial_t + Y_i(x,t)\partial_x - \bar{K}_i(x,t)u\partial_u$ for $i=1,\dots,4$. 
The coefficients $m_i$ and $Y_i$ correspond to those in the Doob symmetries $V_i$, while $\bar{K}_i = K \circ \Phi^{-1}$ (see Lemma \ref{mu 1}). 
Here, $K$ is determined by the relation \eqref{k}, which follows from solving the equation for $\epsilon$ in the system \eqref{sistema}. 
Thus, the symmetry algebra of \eqref{CIR KBE} is spanned by:

\begin{align*}
    \Xi_1 &= \partial_t, \\
    \Xi_2 &= a e^{at}x\partial_x+e^{at}\partial_t, \\
    \Xi_3 &= -a e^{-at}x\partial_x+e^{-at}\partial_t+\frac{2a e^{-at}}{\sigma^2}(ax+b) u \partial_u, \\
    \Xi_4 &= u\partial_u.
\end{align*}

On the other hand, if
\[
    b=\frac34\sigma^2,
\]
then \(m_3=0\) and
\[
    M(x)=m_0+m_2x^2,
\]
so that \(M'''=0\). Therefore, by Theorem
\ref{classification mu sigma}, the CIR equation belongs to the maximal
case and admits six Doob symmetries, rather than four.

\bibliographystyle{plain}
\bibliography{references}
\end{document}